\RequirePackage{fix-cm}
\documentclass[smallextended,numbook,runningheads]{svjour3}
\smartqed  
\usepackage{xcolor}
\usepackage{graphicx}
\usepackage{amsmath,amssymb,bm}
\usepackage{graphicx,graphics,color}
\usepackage[colorlinks,linkcolor=blue]{hyperref}
\usepackage{tikz}

\journalname{...}
\date{ \phantom{b} \vspace{45mm}\phantom{e}}
\def\makeheadbox{\hspace{-4mm}Version of \today}

\usepackage{color}

\def\half{\frac{1}{2}}

\newcommand\bfd{{\mathbf d}}
\newcommand\bfe{{\mathbf e}}
\newcommand\bff{{\mathbf f}}
\newcommand\bfg{{\mathbf g}}

\newcommand\bfn{{\mathbf n}}
\newcommand\bfr{{\mathbf r}}
\newcommand\bfs{{\mathbf s}}
\newcommand\bfu{{\mathbf u}}
\newcommand\bfv{{\mathbf v}}
\newcommand\bfw{{\mathbf w}}
\newcommand\bfx{{\mathbf x}}
\newcommand\bfy{{\mathbf y}}
\newcommand\bfz{{\mathbf z}}
\newcommand\bfA{{\mathbf A}}

\newcommand\bfE{{\mathbf E}}
\newcommand\bfH{{\mathbf H}}

\newcommand\bfK{{\mathbf K}}
\newcommand\bfM{{\mathbf M}}

\newcommand\bfV{{\mathbf V}}

\newcommand\bfzero{{\mathbf 0}}

\newcommand\andquad{\quad\hbox{ and }\quad}

\renewcommand\d{\text{d}}

\newcommand{\Ga}{\Gamma}

\newcommand{\laplace}{\Delta}

\newcommand{\mat}{\partial^{\bullet}}

\newcommand{\diff}{\frac{\d}{\d t}}
\newcommand{\eps}{\varepsilon}

\newcommand{\inv}{^{-1}}
\newcommand{\la}{\langle}

\newcommand{\nb}{\nabla}

\newcommand{\pa}{\partial}
\newcommand{\R}{\mathbb{R}}
\newcommand{\ra}{\rangle}

\newcommand{\spn}{\textnormal{span}}
\newcommand{\st}{such that}

\def \t {(t)}

\def \to {\rightarrow}

\newcommand{\phiv}{\varphi^v}
\newcommand{\phin}{\varphi^\n}
\newcommand{\phiH}{\varphi^H}

\newcommand{\fg}{G}

\newcommand{\normM}[1]{\| #1 \|_{\bfM(\xs)}}
\newcommand{\normMnum}[1]{\| #1 \|_{\bfM(\bfx)}}
\newcommand{\normMs}[1]{\| #1 \|_{\bfM(\xs(s))}}
\newcommand{\normA}[1]{\| #1 \|_{\bfA(\xs)}}
\newcommand{\normAnum}[1]{\| #1 \|_{\bfA(\bfx)}}
\newcommand{\normAt}[1]{\| #1 \|_{\bfA(\xs(t))}}
\newcommand{\normK}[1]{\| #1\|_{\bfK(\xs)}}
\newcommand{\normKt}[1]{\| #1\|_{\bfK(\xs(t))}}
\newcommand{\normKs}[1]{\| #1\|_{\bfK(\xs(s))}}
\newcommand{\normKo}[1]{\| #1\|_{\bfK(\xs(0))}}

\newcommand{\wtx}{\widetilde{\bfx}}
\newcommand{\us}{\bfu^\ast}
\newcommand{\vs}{\bfv^\ast}
\newcommand{\xs}{\bfx^\ast}
\newcommand{\dotus}{\dot\bfu^\ast}

\newcommand{\dotxs}{\dot\bfx^\ast}
\newcommand{\uls}{\bfu_\ast}
\newcommand{\vls}{\bfv_\ast}
\newcommand{\xls}{\bfx_\ast}
\newcommand{\wtuls}{\widetilde{\bfu}_\ast}

\newcommand{\wtxls}{\widetilde{\bfx}_\ast}

\newcommand{\eu}{\bfe_\bfu}
\newcommand{\ev}{\bfe_\bfv}
\newcommand{\ex}{\bfe_\bfx}
\newcommand{\wteu}{{\bfe}_{\widetilde\bfu}}

\newcommand{\wtex}{{\bfe}_{\widetilde\bfx}}
\newcommand{\doteu}{\dot\bfe_\bfu}

\newcommand{\dotex}{\dot\bfe_\bfx}
\newcommand{\du}{\bfd_\bfu}
\newcommand{\dv}{\bfd_\bfv}

\newcommand{\dx}{\bfd_\bfx}
\newcommand{\ru}{\bfr_\bfu}
\newcommand{\rv}{\bfr_\bfv}

\newcommand{\n}{\nu}

\newcommand{\dof}{N}

\newcommand{\blueon}{\color{black}}
\newcommand{\blueoff}{\color{black}}

\newcommand{\redon}{\color{black}}
\newcommand{\redoff}{\color{black}}

\newcommand{\bcl}{\color{black}}
\newcommand{\ecl}{\color{black}}

\newcommand{\bbk}{\color{black}}
\newcommand{\ebk}{\color{black}}

\newcommand{\be}{\partial^\tau\!}

\begin{document}

\title{A convergent evolving finite element algorithm for mean curvature flow of closed surfaces}

\titlerunning{A convergent finite element algorithm for mean curvature flow}        

\author{Bal\'{a}zs~Kov\'{a}cs \and
        Buyang~Li \and
        Christian~Lubich
}

\authorrunning{B.~Kov\'{a}cs, B.~Li and Ch.~Lubich} 

\institute{B. Kov\'{a}cs and Ch. Lubich  \at
              Mathematisches Institut, Universit\"at T\"{u}bingen,\\
              Auf der Morgenstelle 10, 72076 T\"{u}bingen, Germany \\
              \email{\{kovacs,lubich\}@na.uni-tuebingen.de}
           \and
           B. Li \at
              Department of Applied Mathematics, Hong Kong Polytechnic University,\\
              Kowloon, Hong Kong \\
              \email{buyang.li@polyu.edu.hk}
}

\dedication{This paper is dedicated to Gerhard Dziuk on the occasion of his 70th birthday and to Gerhard Huisken on the occasion of his 60th birthday.}

\date{}

\maketitle

\begin{abstract}
A proof of convergence is given for semi- and full discretizations of mean curvature flow of closed two-dimensional surfaces. The numerical method proposed and studied here combines evolving finite elements, whose nodes determine the discrete surface like in {Dziuk's} method, and linearly implicit backward difference formulae for time integration. The proposed method differs from Dziuk's approach in that it discretizes Huisken's evolution equations for the normal vector and mean curvature and uses these evolving geometric quantities in the velocity law projected to the finite element space. This numerical method admits a convergence analysis \bcl in the case of finite elements of polynomial degree at least two and backward difference formulae of orders two to five. The error analysis \ecl combines
stability estimates and consistency estimates to yield optimal-order $H^1$-norm error bounds for the computed surface position, velocity, normal vector and mean curvature. The
stability analysis is based on the matrix--vector formulation of the finite element method and does not use geometric arguments. The geometry enters only into the consistency estimates. Numerical experiments illustrate and complement the theoretical results.

\end{abstract}

\section{Introduction}

Mean curvature flow is a geometric evolution equation that has been studied intensively in geometric analysis in the last decades, as is evidenced by the recent monographs by Ecker \cite{Ecker2012} and Mantegazza \cite{Mantegazza} and the numerous references therein, going back to the pioneering work by Brakke~\cite{Brakke1978} and Huisken~\cite{Huisken1984}. As White \cite{White2002} puts it succinctly in his review article:  ``There are many processes by which
a curve or surface can evolve, but among them one is arguably the most natural: the mean curvature flow.''

Approximating the mean curvature flow by numerical methods was first addressed by Dziuk \cite{Dziuk90} in 1990. He proposed a finite element method based on a weak formulation of the mean curvature flow as a (formally) heat-like partial differential equation, in which the moving nodes of the finite element mesh determine the approximate evolving surface. However, proving convergence of Dziuk's method or \bcl of other evolving finite element methods, such as the method proposed by Barrett, Garcke \& N\"urnberg \cite{BGN2008}, has remained an open problem for the mean curvature flow of closed two-dimensional surfaces (or higher-dimensional hypersurfaces).\ecl

Convergence of an evolving finite element method was shown for the mean curvature flow of closed {\it curves} (known as curve-shortening flow) in  \cite{Dziuk94,DeckelnickDziuk};  see also \cite{BDS}. For a \bcl non-evolving \ecl finite element discretization of the nonlinear parabolic equation describing the mean curvature flow of a two-dimensional surface that is a {\it graph} (with the evolving height over a fixed domain as the unknown variable), convergence of optimal order was shown in \cite{deckelnick1995convergence,deckelnick2000error-graph}.  To our knowledge, the only convergence result of a numerical method for mean curvature flow of closed surfaces in the literature is for a level set method~\cite{deckelnick2000error-levelset}, under very weak regularity assumptions but consequently with very low order of convergence and only under very restrictive conditions between the mesh size and the regularization parameter that appears in the method. We further refer to Deckelnick, Dziuk \& Elliott \cite{DeckelnickDE2005} for an excellent review of numerical approaches to mean curvature flow and related problems, including applications in various fields of science.

In the present paper we formulate a novel evolving finite element method for mean curvature flow of closed two-dimensional surfaces and prove optimal-order convergence over time intervals on which the evolving surface remains sufficiently regular. We study stability and convergence for both the finite element semi-discretization and the full discretization obtained with a linearly implicit backward difference time discretization.  Our approach shares with Dziuk's method the property that the moving nodes of a finite element mesh determine the approximate evolving surface. However, the method presented here discretizes equations that are different from the equation discretized by Dziuk. In his approach, a weak formulation of the quasi-heat equation describing mean curvature flow is discretized, whereas in the present work evolution equations for the normal vector and the mean curvature are discretized, which then yield the velocity of the surface evolving under mean curvature flow. Evolution equations for geometric quantities on a surface evolving under mean curvature flow have been an important tool in the analysis of mean curvature flow ever since Huisken's 1984 paper \cite{Huisken1984}, but apparently they have so far not been used in the numerical approximation of mean curvature flow. 

The numerical method based on the discretization of evolution equations of geometric quantities, as presented here, is computationally more expensive than Dziuk's method (roughly by about a factor 2), but on the other hand it provides full-order approximations to basic geometric quantities --- the normal vector and curvature --- in addition to the position and velocity of the surface.

Our numerical approach is related to our previous paper \cite{KLLP2017},  where we study the convergence of finite elements on an evolving surface driven by diffusion on the surface. The convergence analysis of the present paper relies on techniques developed in that previous paper. As in \cite{KLLP2017},  the stability
analysis  works  with  the  matrix--vector  formulation  of  the  method  and  does
not use geometric arguments. The geometry only enters into the analysis of the consistency error.

The paper is organized as follows:

In Section 2 we recall Huisken's evolution equations for the normal vector and mean curvature of surfaces under mean curvature flow and formulate the system of equations that will be discretized in the following.

In Section 3 we formulate the semi-discretization in space by the evolving surface finite element method, and in Section 4 we state and discuss the convergence result for the semi-discretization (Theorem 4.1).

In Section 5 we formulate the full discretization with linearly implicit backward difference formulae, and in Section 6 we state the convergence result for the full discretization (Theorem 6.1).

In Sections 7 to 12 we give the proofs of the convergence results.
Theorem 4.1 is proved in Sections 7, 8 and 9, where we study stability, consistency and their combination, respectively, for the semi-discretization.
Theorem 6.1 is proved in Sections 10, 11 and 12, where again we study stability, consistency and their combination, respectively, this time for the full discretization.

In Section 13 we present results of some numerical experiments.

Finally, Section 14 collects some conclusions and adds comments on related further topics that are not addressed in the main text.

We use the notational convention to denote vectors in $\R^3$ by italic letters, but to denote finite element nodal vectors in $\R^N$ and $\R^{3N}$ by boldface lowercase letters and finite element mass and stiffness matrices by boldface capitals. All boldface symbols in this paper will thus be related to the matrix--vector formulation of the finite element method.

\section{Evolution equations for mean curvature flow}

\subsection{Basic notions and notation}
\label{subsection: basic notions}

We consider the evolving two-dimensional closed surface $\Gamma(t)\subset\R^3$ as the image
$$
\Ga(t) = \{ X(p,t) \,:\, p \in \Ga^0 \}
$$
of a smooth mapping $X:\Ga^0\times [0,T]\to \R^3$ such that $X(\cdot,t)$ is an embedding for every $t$. Here, $\Ga^0$ is a smooth closed initial surface, and $X(p,0)=p$. 
In view of the subsequent numerical discretization, it is convenient to think of $X(p,t)$ as the position at time $t$ of a moving particle with label $p$, and of $\Ga(t) $ as a collection of such particles. 
To indicate the dependence of the surface on~$X$, we will write
$$
\Ga(t) = 
\Ga[X(\cdot,t)] , \quad\hbox{ or briefly}\quad \Ga[X]
$$
when the time $t$ is clear from the context. The {\it velocity} $v(x,t)\in\R^3$ at a point $x=X(p,t)\in\Gamma(t)$  equals
\begin{equation} \label{velocity}
 \partial_t X(p,t)= v(X(p,t),t).
\end{equation}
For a known velocity field  $v$,  the position $X(p,t)$ at time $t$ of the particle with label $p$ is obtained by solving the ordinary differential equation \eqref{velocity} from $0$ to $t$ for a fixed $p$.

For a function $u(x,t)$ ($x\in \Gamma(t)$, $0\le t \le T$) we denote the {\it material derivative} (with respect to the parametrization $X$) as
$$
\mat u(x,t) = \frac \d{\d t} \,u(X(p,t),t) \quad\hbox{ for } \ x=X(p,t).
$$

On any regular surface $\Gamma\subset\R^3$, we denote by $\nabla_{\Ga}u:\Gamma\to\R^3$ the  {\it tangential gradient} of a function $u:\Gamma\to\R$, and in the case of a vector-valued function $u=(u_1,u_2,u_3)^T:\Gamma\to\R^3$, we let
$\nabla_{\Ga}u=
(\nabla_{\Ga}u_1,
 \nabla_{\Ga}u_2,
\nabla_{\Ga}u_3)$. \bbk We thus use the convention that the gradient of $u$ has the gradient of the components as column vectors. \ebk 
We denote by $\nabla_{\Ga} \cdot f$ the {\it surface divergence} of a vector field $f$ on $\Gamma$, 
and by 
$\laplace_{\Ga} u=\nabla_{\Ga}\cdot \nabla_{\Ga}u$ the {\it Laplace--Beltrami operator} applied to $u$; see the review \cite{DeckelnickDE2005} or \cite[Appendix~A]{Ecker2012} or any textbook on differential geometry for these notions. 

We denote the unit outer normal vector field to $\Gamma$ by $\n:\Gamma\to\R^3$. Its surface gradient contains the (extrinsic) curvature data of the surface $\Gamma$. At every $x\in\Gamma$, the matrix of the extended Weingarten map,
$$
A(x)=\nabla_\Gamma \n(x),
$$ 
is a symmetric $3\times 3$ matrix (see, e.g., \cite[Proposition~20]{Walker2015}). Apart from the eigenvalue $0$ with eigenvector $\n(x)$, its other two eigenvalues are
the principal curvatures $\kappa_1$ and $\kappa_2$ at the point $x$ on the surface. They determine the fundamental quantities
\begin{align}\label{Def-H-A2}
H:={\rm tr}(A)=\kappa_1+\kappa_2, \qquad 
 |A|^2 = \kappa_1^2 +\kappa_2^2 ,
\end{align}
where $|A|$ denotes the Frobenius norm of the matrix $A$.
Here, $H$ is called the {\it mean curvature} (as in most of the literature, we do not put a factor 1/2). 

\subsection{Evolution equations for normal vector and mean curvature of a surface moving under mean curvature flow}
Mean curvature flow sets the velocity \eqref{velocity} of the surface $\Ga[X]$ to
\begin{equation}
\label{mcf-v}
v= -H\n.
\end{equation}
The geometric quantities on the right-hand side are known to satisfy the following evolution equations.

\begin{lemma} \label{lem:mat-Hn} {\rm (Huisken \cite{Huisken1984})}
 For a regular surface $\Ga[X]$ moving under mean curvature flow,  the normal vector and the mean curvature satisfy
\begin{align}
\mat \n &= \laplace_{\Ga[X]} \n + |A|^2 \,\n, \label{Eq_n} \\
\mat H &= \laplace_{\Ga[X]} H + |A|^2 H. \label{Eq_H}
\end{align}
\end{lemma}
\begin{proof}$\,$ 
The evolution equation for $H$ is given in \cite[Corollary 3.5]{Huisken1984}. In \cite[Lemma 3.3]{Huisken1984},
the following evolution equation for the normal vector is derived: 
$$
\mat \n = \nabla_{\Ga[X]}H  . 
$$
On any surface $\Gamma$, it holds true that (see \cite[(A.9)]{Ecker2012} or \cite[Proposition~24]{Walker2015})
$$
\nabla_{\Ga[X]} H = \Delta_{\Ga[X]} \n + |A|^2 \n,
$$
which then gives the stated evolution equation for $\n$; see also \cite[(B.11)]{Ecker2012}.
\qed
\end{proof}

\subsection{The system of equations used for discretization}
Collecting the above equations, we have reformulated mean curvature flow as the system of semilinear parabolic equations \eqref{Eq_n}--\eqref{Eq_H} on the surface coupled to the velocity law \eqref{mcf-v}.
The numerical discretization is based on a weak formulation of \eqref{mcf-v}-\eqref{Eq_H}: we use the velocity equation \eqref{velocity} together with
\begin{subequations}
\label{weak form}
\begin{align}
&	 \int_{\Ga[X]}\!\! \nabla_{\Ga[X]} v \cdot  \nabla_{\Ga[X]} \phiv +  \int_{\Ga[X]} \!\! v \cdot \phiv 
\\ \nonumber
& \hskip 2cm
=   -\!\int_{\Ga[X]} \!\!\nabla_{\Ga[X]}(H\n) \cdot \nabla_{\Ga[X]}\phiv  -\!\int_{\Ga[X]} \!\! H\n \cdot \phiv
	\\
&	 \int_{\Ga[X]} \mat \n \cdot \phin + \int_{\Ga[X]} \nabla_{\Ga[X]} \n \cdot \nabla_{\Ga[X]} \phin =  \int_{\Ga[X]} | \nabla_{\Ga[X]} \n|^2\,   \n\, \cdot \phin
	\\
&	 \int_{\Ga[X]} \!\!\mat H \, \phiH + \int_{\Ga[X]} \!\!\nabla_{\Ga[X]} H \cdot  \nabla_{\Ga[X]} \phiH =   \int_{\Ga[X]} | \nabla_{\Ga[X]} \n|^2 \,  H\, \phiH
\end{align}
\end{subequations}
for all test functions $\phiv \in H^1(\Ga[X])^3$ and $\phin \in H^1(\Ga[X])^3$, $\phiH \in H^1(\Ga[X])$. Here,
we use the Sobolev space 
$H^1(\Ga)=\{ u \in L^2(\Gamma)\,:\, \nabla_\Gamma u \in L^2(\Gamma) \}$. This system is complemented with the initial data $X^0$, $\n^0$ and $H^0$.

\bbk Throughout the paper both the usual Euclidean scalar product for vectors and the Frobenius inner product for matrices (which equals to the Euclidean product using an arbitrary vectorisation) are denoted by a dot. \ebk

It is instructive to compare this system with the equations on which Dziuk's discretization in \cite{Dziuk90} is based. There, he uses that
$-H\nu =  \Delta_\Gamma x_\Gamma$, with $x_\Gamma$ denoting the identity map on $\Gamma$, and employs the velocity equation \eqref{velocity} together with the weak formulation of the velocity law $v = \Delta_{\Gamma[X]} x_{\Gamma[X]}$:
\begin{equation} \label{weak form Dziuk}
\int_{\Ga[X]} \! v \cdot \varphi = -  \int_{\Ga[X]} \nabla_{\Ga[X]} \bcl x_{\Gamma[X]} \ecl \cdot \nabla_{\Ga[X]}\varphi
\end{equation}
for all test functions $\varphi\in H^1(\Ga[X])^3$. This weak formulation clearly has the charm of greater simplicity than \eqref{weak form},  but no convergence proof for a discretization based on this formulation is known to date.

\section{Evolving finite element semi-discretization}
\label{section:ESFEM}

\subsection{Evolving surface finite elements}
We formulate the evolving surface finite element (ESFEM) discretization for the velocity law coupled with evolution equations on the evolving surface, following the description in \cite{KLLP2017}, which is based on \cite{Dziuk88} and \cite{Demlow2009}. We use simplicial finite elements and continuous piecewise polynomial basis functions of degree~$k$, as defined in \cite[Section 2.5]{Demlow2009}.

We triangulate the given smooth initial surface $\Gamma^0$ by an admissible family of triangulations $\mathcal{T}_h$ of decreasing maximal element diameter $h$; see \cite{DziukElliott_ESFEM} for the notion of an admissible triangulation, which includes quasi-uniformity and shape regularity. For a momentarily fixed $h$, we denote by $\bfx^0 $  the vector in $\R^{3\dof}$ that collects all nodes $p_j$ $(j=1,\dots,\dof)$ of the initial triangulation. By piecewise polynomial interpolation of degree $k$, the nodal vector defines an approximate surface $\Gamma_h^0$ that interpolates $\Gamma^0$ in the nodes $p_j$. We will evolve the $j$th node in time, denoted $x_j(t)$ with $x_j(0)=p_j$, and collect the nodes at time $t$ in a column vector
$$
	\bfx(t) \in \R^{3\dof}. 
$$
We just write $\bfx$ for $\bfx(t)$ when the dependence on $t$ is not important.

By piecewise polynomial interpolation on the  plane reference triangle that corresponds to every
curved triangle of the triangulation, the nodal vector $\bfx$ defines a closed surface denoted by $\Gamma_h[\bfx]$. We can then define globally continuous finite element {\it basis functions}
$$
	\phi_i[\bfx]:\Gamma_h[\bfx]\to\R, \qquad i=1,\dotsc,\dof,
$$
which have the property that on every triangle their pullback to the reference triangle is polynomial of degree $k$, and which satisfy at the node $x_j$
\begin{equation*}
	\phi_i[\bfx](x_j) = \delta_{ij} \quad  \text{ for all } i,j = 1,  \dotsc, \dof .
\end{equation*}
These functions span the finite element space on $\Gamma_h[\bfx]$,
\begin{equation*}
	S_h[\bfx] = S_h(\Gamma_h[\bfx])=\spn\big\{ \phi_1[\bfx], \phi_2[\bfx], \dotsc, \phi_\dof[\bfx] \big\} .
\end{equation*}
For a finite element function $u_h\in S_h[\bfx]$, the tangential gradient $\nabla_{\Gamma_h[\bfx]}u_h$ is defined piecewise on each element.

The discrete surface at time $t$ is parametrized by the initial discrete surface via the map $X_h(\cdot,t):\Gamma_h^0\to\Gamma_h[\bfx(t)]$ defined by
$$
	X_h(p_h,t) = \sum_{j=1}^\dof x_j(t) \, \phi_j[\bfx(0)](p_h), \qquad p_h \in \Gamma_h^0,
$$
which has the properties that $X_h(p_j,t)=x_j(t)$ for $j=1,\dots,\dof$, that  $X_h(p_h,0) \linebreak = p_h$ for all $p_h\in\Gamma_h^0$, and
$$
\Gamma_h[\bfx(t)]=\Gamma[X_h(\cdot,t)],
$$
where the right-hand side equals $\{ X_h(p_h,t) \,:\, p_h \in \Ga_h^0 \}$ like in Section~\ref{subsection: basic notions}.

The {\it discrete velocity} $v_h(x,t)\in\R^3$ at a point $x=X_h(p_h,t) \in \Gamma[X_h(\cdot,t)]$ is given by
$$
	\partial_t X_h(p_h,t) = v_h(X_h(p_h,t),t).
$$
In view of the transport property of the basis functions  \cite{DziukElliott_ESFEM},
$$
	\frac\d{\d t} \Bigl( \phi_j[\bfx(t)](X_h(p_h,t)) \Bigr) =0 ,
$$
the discrete velocity equals, for $x \in \Gamma_h[\bfx(t)]$,
$$
	v_h(x,t) = \sum_{j=1}^\dof v_j(t) \, \phi_j[\bfx(t)](x) \qquad \hbox{with } \ v_j(t)=\dot x_j(t),
$$
where the dot denotes the time derivative $\d/\d t$. 
Hence, the discrete velocity $v_h(\cdot,t)$ is in the finite element space $S_h[\bfx(t)]$, with nodal vector $\bfv(t)=\dot\bfx(t)$.
%

The {\it discrete material derivative} of a finite element function $u_h(x,t)$ with nodal values $u_j(t)$ is
$$
\mat_h u_h(x,t) = \frac{\d}{\d t} u_h(X_h(p_h,t)) = \sum_{j=1}^\dof \dot u_j(t)  \phi_j[\bfx(t)](x)  \quad\text{at}\quad x=X_h(p_h,t).
$$

 
\subsection{ESFEM spatial semi-discretization}
\label{subsection:semi-discretization}

The finite element spatial semi-discretization of the weak coupled parabolic system \eqref{weak form} reads as follows: Find the unknown nodal vector $\bfx(t)\in \R^{3\dof}$ and the unknown finite element functions $v_h(\cdot,t)\in S_h[\bfx(t)]^3$ and 
 $\n_h(\cdot,t)\in S_h[\bfx(t)]^3$, $H_h(\cdot,t)\in S_h[\bfx(t)]$ such that
\begin{subequations}
\label{semidiscrete weak form}
	\begin{align}
	&\int_{\Gamma_h[\bfx]}\!\! \nb_{\Ga_h[\bfx]} v_h \cdot \nb_{\Ga_h[\bfx]} \phiv_h +\int_{\Gamma_h[\bfx]}\!\! v_h \cdot \phiv_h 
	\nonumber \\
	& \qquad\qquad =   -\!\int_{\Ga_h[\bfx]} \!\!\nabla_{\Gamma_h[\bfx]}(H_h\n_h) \cdot \nabla_{\Ga_h[\bfx]}\phiv_h -\!\int_{\Gamma_h[\bfx]} \!\! H_h\n_h \cdot \phiv_h
	\\ \label{discrete nu}
	&\int_{\Gamma_h[\bfx]}\!\!\mat_h \n_h \cdot \phin_h +\int_{\Gamma_h[\bfx]}\!\! \nb_{\Ga_h[\bfx]} \n_h \cdot \nb_{\Ga_h[\bfx]} \phin_h =  
	\int_{\Gamma_h[\bfx]}\!\! |\nb_{\Ga_h[\bfx]} \n_h|^2  \,\n_h \cdot \phin_h
	\\
	&\int_{\Gamma_h[\bfx]}\!\!\mat_h H_h \, \phiH_h +\int_{\Gamma_h[\bfx]}\!\! \nb_{\Ga_h[\bfx]} H_h \cdot \nb_{\Ga_h[\bfx]} \phiH_h =  
	\int_{\Gamma_h[\bfx]}\!\!|\nb_{\Ga_h[\bfx]} \n_h|^2  \,H_h \, \phiH_h
	\end{align}
\end{subequations}
for all $\phiv_h\in S_h[\bfx(t)]^3$, $\phin_h\in S_h[\bfx(t)]^3$, and $\phiH_h\in S_h[\bfx(t)]$, with  the surface $\Gamma_h[\bfx(t)]=\Gamma[X_h(\cdot,t)] $ given by the differential equation
\begin{equation}\label{xh}
	\partial_t X_h(p_h,t) = v_h(X_h(p_h,t),t), \qquad p_h\in\Ga_h^0.
\end{equation}
The initial values for the nodal vector $\bfx$ are taken as the positions of the nodes of the triangulation of the given initial surface $\Gamma^0$.
The initial data for $\n_h$ and $H_h$ are determined by Lagrange interpolation of $\n^0$ and $H^0$, respectively. 

In the above approach, the discretization of the evolution equations for $\n$ and $H$ is done in the usual way of evolving surface finite elements. The velocity law \eqref{mcf-v} is enforced by a Ritz projection of $-H_h\n_h$ to the finite element space on $\Gamma_h[\bfx]$. Taking instead just the finite element interpolation of $-H_h\n_h$ does not appear sufficient for a convergence analysis.
Note that the finite element functions $\n_h$ and $H_h$ are {\it not} the normal vector and the mean curvature of the discrete surface $\Gamma_h[\bfx(t)]$.

For comparison, Dziuk's discretization \cite{Dziuk90} uses the differential equation \eqref{xh} together with the ESFEM discretization of \eqref{weak form Dziuk}: for all $\varphi_h\in S_h[\bfx(t)]^3$,
\begin{equation} \label{semidiscrete weak form Dziuk}
\int_{\Ga_h[\bfx]} \!\! v_h \cdot \varphi_h = -  \int_{\Ga_h[\bfx]} \nabla_{\Ga_h[\bfx]} x_{\Ga_h[\bfx]} \cdot \nabla_{\Ga_h[\bfx]}\varphi_h.
\end{equation}

\subsection{Matrix--vector formulation}
\label{subsection:DAE}

We collect the nodal values in column vectors  $\bfv=(v_j) \in \R^{3N}$, $\bfH=(H_j)\in\R^N$ and $\bfn=(\n_j) \in \R^{3N}$, and denote 
$$
\bfu=\left(\begin{array}{c}
\bfn\\ 
\bfH
\end{array}\right) \in \R^{4N}
.$$ 

We define the surface-dependent mass matrix $\bfM(\bfx)$ and stiffness matrix $\bfA(\bfx)$ 
on the surface determined by the nodal vector $\bfx$:
\begin{equation*}
\begin{aligned}
	\bfM(\bfx)\vert_{ij} =&\ \int_{\Ga_h[\bfx]} \! \phi_i[\bfx] \phi_j[\bfx] , \\
	\bfA(\bfx)\vert_{ij} =&\ \int_{\Ga_h[\bfx]} \! \nb_{\Ga_h[\bfx]} \phi_i[\bfx] \cdot \nb_{\Ga_h[\bfx]} \phi_j[\bfx] , 
\end{aligned}
\qquad i,j = 1,  \dotsc,\dof ,
\end{equation*}
with the finite element nodal basis functions $\phi_j[\bfx] \in S_h[\bfx]$.
We define the nonlinear functions $\bff(\bfx,\bfu)\in\R^{4N}$ and $\bfg(\bfx,\bfu)\in\R^{3N}$ by
$$
\bff(\bfx,\bfu)
=
\left(\begin{array}{c}
\bff_1(\bfx,\bfu)  \\ 
\bff_2(\bfx,\bfu)  
\end{array}\right)
$$
with $\bff_1(\bfx,\bfu)\in \R^{3N}$ and $\bff_2(\bfx,\bfu)\in \R^{N}$, given by  \begin{equation*}
	\begin{aligned}
	\bff_1(\bfx,\bfu)\vert_{j+(\ell-1)N} &=\int_{\Gamma_h[\bfx]} \!\! \alpha_h^2 \,(\n_h)_\ell \,  \, \phi_j[\bfx] , 
	\qquad\text{with}\quad  \alpha_h^2 = | \nb_{\Ga_h[\bfx]} \n_h |^2,
	\\
	\bff_2(\bfx,\bfu)\vert_j &=\int_{\Gamma_h[\bfx]}\!\! \alpha_h^2 \,H_h \,  \, \phi_j[\bfx] ,
	\\
	\bfg(\bfx,\bfu)\vert_{j+(\ell-1)N} &= -\int_{\Gamma_h[\bfx]} \!\!\!\!\!   H_h  (\n_h  )_\ell \, \phi_j[\bfx]
	-\int_{\Gamma_h[\bfx]} \!\!\!\!\!  \nb_{\Ga_h[\bfx]} (H_h  (\n_h  )_\ell)\cdot  \nb_{\Ga_h[\bfx]}\phi_j[\bfx] ,
	\end{aligned}
\end{equation*}
for $j = 1, \dotsc, N$, and $\ell=1,2,3$. 

We further let, for $d=3$ or $4$ (with the identity matrices $I_d \in \R^{d \times d}$) 
$$
\bfM^{[d]}(\bfx)= I_d \otimes \bfM(\bfx), \quad
\bfA^{[d]}(\bfx)= I_d \otimes \bfA(\bfx),
 \quad
 \bfK^{[d]}(\bfx) = I_d \otimes \bigl( \bfM(\bfx) + \bfA(\bfx) \bigr).
$$
When no confusion can arise, we often write $\bfM(\bfx)$ for $\bfM^{[d]}(\bfx)$, $\bfA(\bfx)$ for $\bfA^{[d]}(\bfx)$, and $\bfK(\bfx)$ for $\bfK^{[d]}(\bfx)$.

Then, the equations \eqref{semidiscrete weak form} can be written in the following matrix--vector form: 
\begin{subequations}
	\label{eq:matrix--vector form}
	\begin{align}
	\bfK^{[3]}(\bfx)\bfv &= \bfg(\bfx,\bfu) , \label{matrix-form-v}
	\\
		\bfM^{[4]}(\bfx)\dot\bfu+\bfA^{[4]}(\bfx)\bfu &= \bff(\bfx,\bfu) , \label{matrix-form-u} 
	\end{align}
\end{subequations}
and \eqref{xh} is equivalent to 
\begin{align}
\label{matrix-form-X}
\dot\bfx=\bfv.
\end{align}
For comparison we state the matrix--vector formulation of Dziuk's semidiscretization \eqref{xh}--\eqref{semidiscrete weak form Dziuk}, which is  \eqref{matrix-form-X} together with
\begin{equation}
\label{eq:matrix--vector form Dziuk}
\bfM^{[3]}(\bfx)\bfv+\bfA^{[3]}(\bfx)\bfx = \bfzero.
\end{equation}
This is certainly more elegant than \eqref{eq:matrix--vector form} and computationally less expensive. However, since the main computational cost is in the computation of the matrices $\bfA(\bfx)$ and $\bfM(\bfx)$ and their decompositions, the (fully discretized) scheme \eqref{eq:matrix--vector form} is not substantially more expensive than the (fully discretized) scheme \eqref{eq:matrix--vector form Dziuk}. On the other hand, the method presented here provides important extra geometrical information (about the normal vector and curvature) --- and it comes with a proof of optimal-order convergence. The critical part in the convergence analysis is the proof of stability of the discretization, in the sense of bounding errors in terms of defects in the numerical equations in the appropriate norms. Our proof of stability uses the matrix--vector formulation \eqref{eq:matrix--vector form} with \eqref{matrix-form-X}, but it does not work for Dziuk's scheme 
\eqref{eq:matrix--vector form Dziuk} with \eqref{matrix-form-X}.

\subsection{Lifts} \label{subsec:lifts}

We need to compare functions on the {\it exact surface} $\Gamma(t)=\Gamma[X(\cdot,t)]$ with functions on the {\it discrete surface} $\Gamma_h(t)=\Gamma_h[\bfx(t)]$. To this end, we further work with functions on the {\it interpolated surface} $\Gamma_h^*(t)=\Gamma_h[\xs(t)]$, where
$\xs(t)$ denotes the nodal vector collecting the grid points $x_j^*(t)=X(p_j,t)$ on the exact surface.

Any finite element function $w_h:\Gamma_h(t)\to\R^m$ ($m=1$ or 3) on the discrete surface, with nodal values $w_j$, is associated with a finite element function $\widehat w_h$ on the interpolated surface $\Gamma_h^*(t)$ that is defined by 
$$
\widehat w_h  = \sum_{j=1}^N w_j \phi_j[\xs(t)].
$$
This can be further lifted to a function on the exact surface by using the \emph{lift operator} $l$, which was introduced for linear and higher-order surface approximations in \cite{Dziuk88} and \cite{Demlow2009}, respectively. 
The lift operator $l$ maps a function on the interpolated surface $\Gamma_h^*$ to a function on the exact surface $\Gamma$, provided that $\Gamma_h^*$ is sufficiently close to $\Gamma$.
The exact regular surface $\Gamma$ can be represented, \bbk in some neighbourhood of the surface, \ebk by a smooth signed distance function $d : \R^3 \times [0,T] \to \R$, cf.~\cite[Section~2.1]{DziukElliott_ESFEM}, such that $\Gamma$ is the zero level set of~$d$ \bbk (i.e., $x\in \Gamma$ if and only if $d(x)=0$), with negative values of $d$ in the interior. \ebk 
Using this distance function,  the lift of a continuous function $\eta_h \colon \Ga_h^* \to \R^m$ is defined as
$
    \eta_{h}^{l}(y) := \eta_h(x)$ for $x\in\Ga_h^*$,
where for every $x\in \Ga_h^*$ the point $y=y(x)\in\Ga$ is uniquely defined via
$
    y = x - \nu(y) d(x).
$
 
The composed lift $L$ from finite element functions on the discrete surface $\Gamma_h(t)$ to functions on the exact surface $\Gamma(t)$ via the interpolated surface $\Gamma_h^*(t)$ is denoted by 
$$
w_h^L = (\widehat w_h)^l.
$$

\section{Convergence of the semi-discretization}
\label{section: main result}

We are now in the position to formulate the first main result of this paper, which yields optimal-order error bounds for the finite element semi-discretization \eqref{semidiscrete weak form}--\eqref{xh} of the system of mean curvature equations \eqref{weak form} with \eqref{velocity}, for finite elements of polynomial degree $k\ge 2$. We denote by $\Gamma(t)=\Gamma[X(\cdot,t)]$ the exact surface and by $\Gamma_h(t)=\Gamma[X_h(\cdot,t)]=\Gamma_h[\bfx(t)]$ the discrete surface at time $t$. We introduce the notation
$$
x_h^L(x,t) =  X_h^L(p,t) \in \Gamma_h(t) \qquad\hbox{for}\quad x=X(p,t)\in\Gamma(t).
$$

\begin{theorem} \label{MainTHM} 
\label{theorem: coupled error estimate}
Consider the space discretization \eqref{semidiscrete weak form}--\eqref{xh} of the coupled mean curvature flow problem  \eqref{weak form} with \eqref{velocity}, using evolving surface finite elements of polynomial degree~$k\ge 2$. 
Suppose that the mean curvature flow problem admits an exact solution $(X,v,\n,H)$ that is sufficiently smooth on the time interval $t\in[0,T]$, and that the flow map $X(\cdot,t):\Gamma^0\to \Gamma(t)\subset\R^3$ is non-degenerate so that $\Gamma(t)$ is a regular surface on the time interval $t\in[0,T]$. 

Then, there exists a constant $h_0 >0$ such that for all mesh sizes $h \leq h_0$  the following error bounds for the lifts of the discrete position, velocity, normal vector and mean curvature hold over the exact surface $\Ga(t)=\Ga[X(\cdot,t)]$ for $0\le t \le T$:
\begin{align*}
        \|x_h^L(\cdot,t) - \mathrm{id}_{\Gamma(t)}\|_{H^1(\Ga(t))^3} &\leq Ch^k, \\
  \|v_h^L(\cdot,t) - v(\cdot,t)\|_{H^1(\Ga(t))^3} 
& \leq C h^k, \\
 \|\n_h^L(\cdot,t) - \n(\cdot,t)\|_{H^1(\Ga(t))^3} 
& \leq C h^k, \\
  \|H_h^L(\cdot,t) - H(\cdot,t)\|_{H^1(\Ga(t))} 
& \leq C h^k,
\end{align*}
and also
\begin{equation*}
\begin{aligned}
 \|X_h^l(\cdot,t) - X(\cdot,t)\|_{H^1(\Ga_0)^3} &\leq Ch^k ,
\end{aligned}
\end{equation*}
where the constant $C$ is independent of $h$ and $t$, but depends on bounds of the $H^{k+1}$ norms of the solution $(X,v,\n,H)$ of the mean curvature flow and on the length $T$ of the time interval.
\end{theorem}

\bbk
Sufficient regularity assumptions are the following: with bounds that are uniform in $t\in[0,T]$, we assume $X(\cdot,t) \in  H^{k+1}(\Ga^0)$,
$v(\cdot,t) \in H^{k+1}(\Ga(X(\cdot,t)))$, and for $u=(\nu,H)$ we have $\ u(\cdot,t), \mat u(\cdot,t) \in W^{k+1,\infty}(\Ga(X(\cdot,t)))^4$.
%
\ebk

The proof of Theorem~\ref{MainTHM} is presented in Sections~\ref{section:stability} to~\ref{section: proof completed}, which are concerned with the stability and consistency of the numerical scheme, and the combination of the estimates.  The proof of stability uses techniques that were first developed in our paper \cite{KLLP2017} on the finite element discretization of an evolving surface driven by the solution of a parabolic equation on the surface.

Several of the observations regarding the convergence result \bbk Theorem~3.1 \ebk of \cite{KLLP2017} apply also to Theorem~\ref{MainTHM}, and so we restate them here:

A key issue in the proof is to ensure that the $W^{1,\infty}$ norm of the position error of the surfaces remains small. The $H^1$ error bound and an inverse estimate yield an $O(h^{k-1})$ error bound in the $W^{1,\infty}$ norm. This is small only for $k\ge 2$, which is why we impose the condition $k\ge 2$ in the above result.

Since the exact flow map $X(\cdot,t):\Gamma^0\to\Gamma(t)$ is assumed to be smooth and non-degenerate, it is locally close to an invertible linear transformation, and (using compactness) it therefore preserves the admissibility of grids with sufficiently small mesh width $h\le h_0$.  Our assumptions therefore guarantee that the triangulations formed by the nodes $x_j^*(t)=X(p_j,t)$ remain admissible uniformly for $t\in[0,T]$ for sufficiently small $h$ (though the bounds in the admissibility inequalities and the largest possible mesh width may deteriorate with growing time). 
\bcl
Since deformation is the gradient of position, the boundedness in $W^{1,\infty}$ of the position error (or equivalently, boundedness of deformation of the numerical flow map), which is ensured since the $O(h^k)$ error bound in $H^1$ and an inverse inequality yield an $O(h^{k-1})$ error bound in $W^{1,\infty}$,  excludes degeneration of the mesh on the numerical surface under the assumptions of the theorem.
\ecl


The error bound will be proved by clearly separating the issues of consistency and stability.
The consistency error is the  defect that results from inserting a projection of the exact solution (here we use finite element interpolation) into the discretized equation.
The defect bounds involve geometric estimates that were obtained for the time dependent case and for higher order $k\ge 2$ in \bbk \cite[Section~5]{Kovacs2017}, \ebk by combining techniques of Dziuk \& Elliott \cite{DziukElliott_ESFEM,DziukElliott_L2} and Demlow \cite{Demlow2009}. This is done within the ESFEM formulation of Section~\ref{subsection:semi-discretization}.

The main difficulty in the proof of Theorem~\ref{MainTHM} is to prove {\it stability} in the form of an $h$-independent bound of the error in terms of the defect in appropriate norms.  The stability analysis is  done in the matrix--vector formulation of Section~\ref{subsection:DAE}.  It uses energy estimates and transport formulae that relate the mass and stiffness matrices and the coupling terms for different nodal vectors~$\bfx$. No geometric estimates enter into the proof of stability. 
 
 While the basic approach to proving stability is the same as taken in \cite{KLLP2017}, a different kind of energy estimates needs to be used, which allows us to derive $H^1$ error bounds uniformly on the time interval (and from this, uniform $W^{1,\infty}$ bounds via the inverse inequality), and not just $L^2(H^1)$ estimates as derived in \cite{KLLP2017}. These sharper estimates are needed because the nonlinearities in our evolution system \eqref{weak form} are only {\it locally} Lipschitz continuous. (In \cite{KLLP2017}, global Lipschitz continuity of the nonlinearity $f(u,\nabla u)$ with respect to the second argument is assumed. This condition is used explicitly in the proof, although it was not clearly stated as an assumption in \cite{KLLP2017}.)
 
 \bcl As is explained in Remark~\ref{rem:Dziuk} below, our techniques do not yield a proof of stability and convergence of Dziuk's method.
 \ecl

\section{Linearly implicit full discretization}

For the time discretization of the system of ordinary differential equations \eqref{eq:matrix--vector form}--\eqref{matrix-form-X} we use a $q$-step linearly implicit backward difference formula (BDF) with $q \leq 5$. For a step size $\tau>0$, and with $t_n = n \tau \leq T$, we determine the approximations $\bfx^n$ to $\bfx(t_n)$, $\bfv^n$ to $\bfv(t_n)$, and $\bfu^n$ to $\bfu(t_n)$ by the fully discrete system of linear equations
\begin{equation}
\label{BDF}
	\begin{aligned}
		\bfK(\widetilde \bfx^n) \bfv^n &= \bfg(\widetilde \bfx^n,\widetilde \bfu^n) , \\
		\bfM(\widetilde \bfx^n) \frac{1}{\tau} \sum_{j=0}^q \delta_j \bfu^{n-j} + \bfA(\widetilde \bfx^n) \bfu^n &= \bff(\widetilde \bfx^n,\widetilde \bfu^n) ,  \\
	 \frac{1}{\tau} \sum_{j=0}^q \delta_j \bfx^{n-j}	 &=  \bfv^n,
	\end{aligned}
\end{equation}
where  $\widetilde \bfx^n$ and $\widetilde \bfu^n$ are the extrapolated values 
\begin{equation}
\label{eq:extrapolation def}
\widetilde \bfx^n = \sum_{j=0}^{q-1} \gamma_j \bfx^{n-1-j} , \qquad	\widetilde \bfu^n = \sum_{j=0}^{q-1} \gamma_j \bfu^{n-1-j} , \qquad n \geq q .
\end{equation}
The starting values $\bfx^i$ and $\bfu^i$ ($i=0,\dotsc,q-1$) are assumed to be given. They can be precomputed using either a lower order method with smaller step sizes or an implicit Runge--Kutta method.

The method is determined by its coefficients, given by $\delta(\zeta)=\sum_{j=0}^q \delta_j \zeta^j=\sum_{\ell=1}^q \frac{1}{\ell}(1-\zeta)^\ell$ and $\gamma(\zeta) = \sum_{j=0}^{q-1} \gamma_j \zeta^j = (1 - (1-\zeta)^q)/\zeta$. 
The classical BDF method is known to be zero-stable for $q\leq6$ and to have order $q$; see \cite[Chapter~V]{HairerWannerII}.
This order is retained by the linearly implicit variant using the above coefficients $\gamma_j$; 
cf.~\cite{AkrivisLubich_quasilinBDF,AkrivisLiLubich_quasilinBDF}.

We note that in the $n$th time step, the method requires solving two linear systems  with the symmetric positive definite matrices $\bfK(\widetilde \bfx^n)$ and $\frac{\delta_0}\tau \bfM(\widetilde \bfx^n) + \bfA(\widetilde \bfx^n)$.

From the vectors $\bfx^n =(x_j^n)$ and $\bfv^n = (v_j^n)$ we obtain  position and velocity approximations to $X(\cdot,t_n)$, 
${\rm id}_{\Gamma[X(\cdot,t_n)]}$, and $v(\cdot,t_n)$ as
\begin{equation}\label{x-v-approx}
  \begin{split}
	X_h^n(p_h) &= \sum_{j=1}^N x_j^n \, \phi_j[\bfx(0)](p_h) \quad\hbox{ for } p_h \in \Gamma_h^0, 
	\\   x_h^n & = {\rm id}_{\Gamma[X_h^n]} ,
	\\
        v_h^n(x)      &= \sum_{j=1}^N v_j^n \, \phi_j[\bfx^n](x)  \qquad\hbox{ for } x \in \Gamma_h[\bfx^n],
  \end{split}
\end{equation}
and from the vectors $\bfu^n=(u_j^n)$ with $u_j^n=(\n_j^n,H_j^n)\in\R^3\times \R$ we obtain approximations to the normal vector and the mean curvature at time $t_n$ as
\begin{equation}\label{n-H-approx}
  \begin{split}
        \n_h^n(x)      &= \sum_{j=1}^N \n_j^n \, \phi_j[\bfx^n](x)  \qquad\hbox{ for } x \in \Gamma_h[\bfx^n], \\
          H_h^n(x)      &= \sum_{j=1}^N H_j^n \, \phi_j[\bfx^n](x)  \qquad\hbox{ for } x \in \Gamma_h[\bfx^n].
  \end{split}
\end{equation}

\section{Convergence of the full discretization}

We are now in the position to formulate the second main result of this paper, which yields optimal-order error bounds for the combined ESFEM--BDF full discretization \eqref{BDF}--\eqref{n-H-approx} of the system of mean curvature equations \eqref{weak form} with \eqref{velocity}, for finite elements of polynomial degree $k\ge 2$ and BDF methods of order $q\le 5$.

\begin{theorem} \label{MainTHM-full} 
Consider the ESFEM--BDF full discretization \eqref{BDF}--\eqref{n-H-approx} of the coupled mean curvature flow problem  \eqref{weak form} with \eqref{velocity}, using evolving surface finite elements of polynomial degree~$k\ge 2$ and linearly implicit BDF time discretization of order \bcl $q$ with $2\le q\le 5$. \ecl
Suppose that the mean curvature flow problem admits an exact solution $(X,v,\n,H)$ that is sufficiently smooth on the time interval $t\in[0,T]$, and that the flow map $X(\cdot,t):\Gamma^0\to \Gamma(t)\subset\R^3$ is non-degenerate so that $\Gamma(t)$ is a regular surface on the time interval $t\in[0,T]$. 

Then, there exist  $h_0 >0$, $\tau_0>0$, and $c_0>0$ such that for all mesh sizes $h \leq h_0$  and time step sizes $\tau\le\tau_0$ satisfying the step size restriction 
\begin{equation} \label{stepsize-restriction}
{\bcl \tau \le C_0 h \ecl}
\end{equation}
(where $C_0>0$ can be chosen arbitrarily),
the following error bounds for the lifts of the discrete position, velocity, normal vector and mean curvature hold over the exact surface: provided that the starting values  are $O(h^k+\redon \tau^{q+1/2} \redoff)$ accurate in the $H^1$ norm at time $t_i=i\tau$ for $i=0,\dots,q-1$, we have at time $t_n=n\tau\le T$
\begin{align*}
        \|(x_h^n)^L - \mathrm{id}_{\Gamma(t_n)}\|_{H^1(\Ga(t_n))^3} &\leq C(h^k+\tau^q), \\
      \|(v_h^n)^L - v(\cdot,t_n)\|_{H^1(\Ga(t_n))^3} &\leq C(h^k+\tau^q), \\ 
        \|(\n_h^n)^L - \n(\cdot,t_n)\|_{H^1(\Ga(t_n))^3} &\leq C(h^k+\tau^q), \\ 
     \|(H_h^n)^L - H(\cdot,t_n)\|_{H^1(\Ga(t_n))} &\leq C(h^k+\tau^q), 
\end{align*}
and also
\begin{equation*}
\begin{aligned}
 \|(X_h^n)^l - X(\cdot,t_n)\|_{H^1(\Ga_0)^3} &\leq C(h^k+\tau^q),
\end{aligned}
\end{equation*}
where the constant $C$ is independent of $h$, $\tau$ and $n$ with $n\tau\le T$, but depends on bounds of higher derivatives of the solution $(X,v,\n,H)$ of the mean curvature flow and on the length $T$ of the time interval, and on $C_0$.
\end{theorem}

\bbk
Sufficient regularity assumptions are the following: uniformly in $t\in[0,T]$ and for $j=1,\dotsc,q+1$,
\begin{align*}
	&\ X(\cdot,t)  \in  H^{k+1}(\Ga^0), \ \pa_t^{j} X(\cdot,t) \in  H^{1}(\Ga^0)  , \\
	&\ v(\cdot,t)  \in H^{k+1}(\Ga(X(\cdot,t))),\ {\mat}^j v(\cdot,t) \in H^{2}(\Ga(X(\cdot,t)))  , \\
	\text{for } \ u=(\nu,H) , \quad &\ u(\cdot,t) , \mat u(\cdot,t) \in  W^{k+1,\infty}(\Ga(X(\cdot,t)))^4  , \\
	&\ {\mat}^j u(\cdot,t) \in  H^{2}(\Ga(X(\cdot,t)))^4  .
\end{align*}
\ebk

The proof of Theorem~\ref{MainTHM-full} is given in Sections~\ref{section:stability-full} to~\ref{section: proof completed - full}. It uses results by Dahlquist~\cite{Dahlquist} and Nevanlinna \& Odeh \cite{NevanlinnaOdeh} \bcl in combination with the transfer of techniques of the proof of 
Theorem~\ref{MainTHM} to the time-discrete situation.  
%
\ecl

\section{Stability of the semi-discretization}
\label{section:stability}

\subsection{Preparation: Estimates relating different finite element surfaces}
\label{section: aux}

In our previous work \cite[Section~4]{KLLP2017} we proved some technical results relating different finite element surfaces, which we recapitulate here.
We use the following setting.

The  finite element matrices of Section~\ref{subsection:DAE} induce discrete versions of Sobolev norms. Let $\bfx \in \R^{3 N}$ be a nodal vector defining the discrete surface $\Gamma_h[\bfx]$. For any nodal vector $\bfw=(w_j) \in \R^{N}$, with the corresponding finite element function $w_h= \sum_{j=1}^N w_j \phi_j[\bfx] \in S_h[\bfx]$, we define the following norms:
\begin{align} \label{M-L2}
	&  \|\bfw\|_{\bfM(\bfx)}^{2} = \bfw^T \bfM(\bfx) \bfw = \|w_h\|_{L^2(\Ga_h[\bfx])}^2 , \\
	\label{A-H1}
	&  \|\bfw\|_{\bfA(\bfx)}^{2} = \bfw^T \bfA(\bfx) \bfw = \|\nb_{\Ga_h[\bfx]} w_h\|_{L^2(\Ga_h[\bfx])}^2 , \\
	\label{K-H1}
	&  \|\bfw\|_{\bfK(\bfx)}^{2} = \bfw^T \bfK(\bfx) \bfw = \|w_h\|_{H^1(\Ga_h[\bfx])}^2 .
\end{align}
When $\bfw\in \R^{dN}$ so that the corresponding finite element function $w_h$ maps into $\R^d$, we write  in the following $\| w_h \|_{L^2(\Gamma)}$ for $\| w_h \|_{L^2(\Gamma)^d}$ and $\| w_h \|_{H^1(\Gamma)}$ for $\| w_h \|_{H^1(\Gamma)^d}$.

Let now $\bfx,\bfy \in \R^{3 N}$ be two nodal vectors defining discrete surfaces $\Gamma_h[\bfx]$ and $\Gamma_h[\bfy]$, respectively. 
We  denote the difference by $\bfe= (e_j)=\bfx-\bfy \in \R^{  3  N}$. For  $\theta\in[0,1]$, we consider the intermediate surface $\Gamma_h^\theta=\Gamma_h[\bfy+\theta\bfe]$ and the corresponding finite element functions given as
$$
	e_h^\theta=\sum_{j=1}^N e_j \phi_j[\bfy+\theta\bfe]
$$
and in the same way, for any vectors $\bfw,\bfz \in \R^N$,
$$
w_h^\theta=\sum_{j=1}^N w_j \phi_j[\bfy+\theta\bfe] \andquad z_h^\theta=\sum_{j=1}^N z_j \phi_j[\bfy+\theta\bfe] .
$$
\bbk
Figure~\ref{figure:relating different surfaces} illustrates the described construction.
\ebk

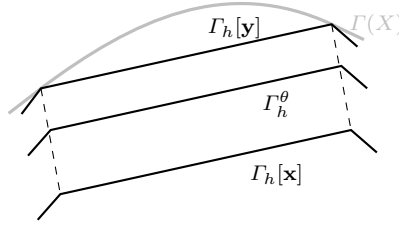
\begin{figure}[htbp]
	\begin{center}
	\begin{tikzpicture}[scale=0.85]
	\draw[very thick,lightgray] (0,4) to [out=37,in=155] (4.5,5);
	\draw[very thick,lightgray] (0,4) to [out=37,in=38] (-0.5,3.61);
	\draw[very thick,lightgray] (4.5,5) to [out=155,in=154] (5,4.76);
	\node[lightgray] at (5.2,5) {$\Ga(X)$};
	\draw[thick] (-0.3,3.6) -- (0,4) -- (4.5,5) -- (4.9,4.65);
	\node at (3.0,4.95) {$\Ga_h[\bfy]$};
	\draw[thick] (-0.2,2.95) -- (0.15,3.35) -- (4.65,4.35) -- (5.05,4.0);
	\node at (3.65,3.75) {$\Ga_h^\theta$};
	\draw[thick] (-0.05,1.95) -- (0.30,2.35) -- (4.80,3.35) -- (5.20,3.0);
	\node at (3.65,2.65) {$\Ga_h[\bfx]$};
	\draw[dashed] (0,4) -- (0.30,2.35);
	\draw[dashed] (4.5,5) -- (4.80,3.35);
	\end{tikzpicture}
	\caption{The construction of the intermediate surfaces $\Gamma_h^\theta$}
	\label{figure:relating different surfaces}
	\end{center}
\end{figure}

\bcl
The following formulae relate the mass and stiffness matrices for the discrete surfaces $\Gamma_h[\bfx]$ and $\Gamma_h[\bfy]$. They result from the Leibniz rule and are given in Lemma 4.1 of \cite{KLLP2017}.
\begin{lemma}
\label{lemma:matrix differences}   In the above setting, the following identities hold true:
	\begin{align}
	\label{eq:matrix difference M}
		\bfw^T (\bfM(\bfx)-\bfM(\bfy)) \bfz =&\ \int_0^1 \int_{\Ga_h^\theta} w_h^\theta (\nabla_{\Ga_h^\theta} \cdot e_h^\theta) z_h^\theta \; \d\theta, \\
	\label{eq:matrix difference A}
		\bfw^T (\bfA(\bfx)-\bfA(\bfy)) \bfz =&\ \int_0^1 \int_{\Ga_h^\theta} \nb_{\Ga_h^\theta} w_h^\theta \cdot (D_{\Ga_h^\theta} e_h^\theta)\nb_{\Ga_h^\theta}  z_h^\theta \; \d\theta ,
	\end{align}
	where
	$D_{\Ga_h^\theta} e_h^\theta =  \textnormal{tr}(E^\theta) I_3 - (E^\theta+(E^\theta)^T)$ with $E^\theta=\nabla_{\Ga_h^\theta} e_h^\theta \in \R^{3\times 3}$.
\end{lemma}	
%
The following lemma combines Lemmas 4.2 and 4.3 of \cite{KLLP2017}.
\begin{lemma} \label{lemma:theta-independence} In the above setting, if 
	$$
	\| \nabla_{\Gamma_h[\bfy]} e_h^0 \|_{L^\infty(\Gamma_h[\bfy])} \le \tfrac12,
	$$
	then, for $0\le\theta\le 1$ and $1\le p \le \infty$, the finite element function \\
	$w_h^\theta=\sum_{j=1}^N w_j \phi_j[\bfy+\theta\bfe]$ on $\Gamma_h^\theta=\Gamma_h[\bfy+\theta\bfe]$ is bounded by
	\begin{align*}
		&\| w_h^\theta \|_{L^p(\Gamma_h^\theta)} \leq c_p \, \|w_h^0 \|_{L^p(\Gamma_h^0)}
	\\
	&\| \nabla_{\Gamma_h^\theta} w_h^\theta \|_{L^p(\Gamma_h^\theta)} \le c_p \, \| \nabla_{\Gamma_h^0} w_h^0 \|_{L^p(\Gamma_h^0)}
	\end{align*}
	where $c_p$ is an absolute constant (in particular, independent of $0\le \theta\le 1$ and~$h$).	Moreover, $c_\infty=2$.
\end{lemma}

The first estimate is not stated explicitly in \cite{KLLP2017}, but follows immediately with the proof of Lemma~4.3 in \cite{KLLP2017}.

If $\| \nabla_{\Gamma_h[\bfy]} e_h^0 \|_{L^\infty(\Gamma_h[\bfy])}\le \frac14$, using the lemma for $w_h^\theta=e_h^\theta$ shows that
\begin{equation}\label{e-theta}
\| \nabla_{\Gamma_h^\theta} e_h^\theta \|_{L^\infty(\Gamma_h^\theta)} \le \tfrac12, \qquad 0\le\theta\le 1,
\end{equation}
and then the lemma with $p=2$ and the definition of the norms \eqref{M-L2} and \eqref{A-H1} (and interchanging the roles of $\bfy$ and $\bfy+\theta\bfe$) show that 
\begin{equation}
\label{norm-equiv}
	\begin{aligned}
		&\text{the norms $\|\cdot\|_{\bfM(\bfy+\theta\bfe)}$ are $h$-uniformly equivalent for $0\le\theta\le 1$,}
		\\
		&\text{and so are the norms $\|\cdot\|_{\bfA(\bfy+\theta\bfe)}$.}
	\end{aligned}
\end{equation}

Under the condition that 
$\eps := \| \nabla_{\Gamma_h[\bfy]} e_h^0 \|_{L^\infty(\Gamma_h[\bfy])}\le \tfrac14$, using \eqref{e-theta} in Lemma~\ref{lemma:matrix differences}  and applying the Cauchy-Schwarz inequality yields the bounds, with $c=c_\infty c_2^2$,
	\begin{equation}\label{eq:matrix difference bounds}
	\begin{aligned}
	\bfw^T (\bfM(\bfx)-\bfM(\bfy)) \bfz \leq&\ c \eps  \,\|\bfw\|_{\bfM(\bfy)}\|\bfz\|_{\bfM(\bfy)} , \\[1mm]
	\bfw^T (\bfA(\bfx)-\bfA(\bfy)) \bfz \leq&\ c \eps\, \|\bfw\|_{\bfA(\bfy)}\|\bfz\|_{\bfA(\bfy)} .
	\end{aligned}
	\end{equation}
We will also use similar bounds where we use the $L^\infty$ norm of $w_h$ or its gradient and the $L^2$ norm of the gradient of $e_h$.
	
Consider now a continuously differentiable function $\bfx:[0,T]\to\R^{3N}$ that defines a finite element surface $\Gamma_h[\bfx(t)]$ for every $t\in[0,T]$, and assume that its time derivative $\bfv(t)=\dot\bfx(t)$ is the nodal vector of a finite element function $v_h(\cdot,t)$ that satisfies
\begin{equation}\label{vh-bound}
\| \nabla_{\Gamma_h[\bfx(t)]}v_h(\cdot,t) \|_{L^{\infty}(\Gamma_h[\bfx(t)])} \le K, \qquad 0\le t \le T.
\end{equation}
With $\bfe=\bfx(t)-\bfx(s)=\int_s^t \bfv(r)\,dr$, 
the bounds \eqref{eq:matrix difference bounds}
then yield the following bounds, which were first shown in Lemma~4.1 of \cite{DziukLubichMansour_rksurf}: 

for $0\le s, t \le T$ with $K|t-s| \le \tfrac14$, 
we have with $C=c K$
	\begin{equation}\label{eq:matrix difference bounds-t}
	\begin{aligned}
	\bfw^T \bigl(\bfM(\bfx(t))  - \bfM(\bfx(s))\bigr)\bfz \leq&\ C \, |t-s| \, \|\bfw\|_{\bfM(\bfx(t))}\|\bfz\|_{\bfM(\bfx(t))} , \\[1mm]
	\bfw^T \bigl(\bfA(\bfx(t))  - \bfA(\bfx(s))\bigr)\bfz \leq&\ C\,  |t-s| \, \|\bfw\|_{\bfA(\bfx(t))}\|\bfz\|_{\bfA(\bfx(t))}.   
	\end{aligned}
	\end{equation}
%
%
%
%
%
Letting $s\to t$, this implies the bounds stated in Lemma~4.6 of~\cite{KLLP2017}:
	\begin{equation}\label{eq:matrix derivatives}
	\begin{aligned}
	\bfw^T \frac\d{\d t}\bfM(\bfx(t))  \bfz \leq&\ C  \,\|\bfw\|_{\bfM(\bfx(t))}\|\bfz\|_{\bfM(\bfx(t))} , \\[1mm]
	\bfw^T \frac\d{\d t}\bfA(\bfx(t))  \bfz \leq&\ C \, \|\bfw\|_{\bfA(\bfx(t))}\|\bfz\|_{\bfA(\bfx(t))} .
	\end{aligned}
	\end{equation}
Moreover, by patching together finitely many intervals over which $K|t-s| \le \tfrac14$, we obtain that
\begin{equation}
\label{norm-equiv-t}
	\begin{aligned}
		&\text{the norms $\|\cdot\|_{\bfM(\bfx(t))}$ are $h$-uniformly equivalent for $0\le t \le T$,}
		\\
		&\text{and so are the norms $\|\cdot\|_{\bfA(\bfx(t))}$.}
	\end{aligned}
\end{equation}\ecl

\subsection{Defects and errors}
\label{subsec:stability - error eqs}

We choose nodal vectors $\xs(t)\in\R^{3N}$, $\vs(t)\in\R^{3N}$ and $\us(t)\in\R^{4N}$ that are related to the exact solution $X$, $v$ and $u=(\n,H)$ as follows: $\xs(t)$ and $\vs(t)$ collect the values at the finite element nodes of $X(\cdot,t)$ and $v(\cdot,t)$, respectively.
The vector $\us(t)$ contains the nodal values of the finite element function $u_h^*(\cdot,t)\in S_h[\xs(t)]^4$ that is defined on the interpolated surface $\Gamma_h[\xs(t)]$ by a Ritz map: omitting the ubiquitous argument $t$,
\bbk \begin{equation}\label{uhs-ritz}
\int_{\Gamma_h[\xs]} \!\! \!\!\!\! \nb_{\Gamma_h[\xs]} u_h^* \cdot \nabla_{\Gamma_h[\xs]} \varphi_h + \int_{\Gamma_h[\xs]} \!\!\!\!  u_h^* \cdot \varphi_h =
\int_{\Gamma[X]} \!\! \!\!\!\! \nabla_{\Gamma[X]} u \cdot \nabla_{\Gamma[X]} \varphi_h^l + \int_{\Gamma[X]} \!\!\!\! u \cdot  \varphi_h^l
\end{equation} \ebk 
for all $\varphi_h \in S_h[\xs]^4$, where again $\varphi_h^l$ denotes the lift to a function on $\Gamma[X]$.

The nodal vectors $\bfx^*(t)$, $\vs(t)$, $\us(t)$ satisfy the  equations  \eqref{eq:matrix--vector form}  up to some defects $\bfd_{\bf v}$ and $\bfd_{\bf u}$ that will be studied in Section~\ref{section:Defect}: 
\begin{subequations}
\label{defect vectors}
	\begin{align}
		\label{dv} 
		\bfK^{[3]}(\xs)\vs &= \bfg(\xs,\us) +\bfM^{[3]} (\xs)\dv , \\
		\label{du}
		\bfM^{[4]}(\xs)\dotus+\bfA^{[4]}(\xs)\us &= \bff(\xs,\us) + \bfM^{[4]}(\xs) \du,
	\end{align}
\end{subequations}
together with the differential equation
\begin{align}
\label{matrix-form-X-star}
	\dotxs=\vs.
\end{align}

\bcl
In the following we omit the dimension superscripts  $[3]$ and $[4]$ on the matrices $\bfM,\bfA$ and $\bfK=\bfM + \bfA$.

 The errors between the nodal values of the numerical solutions and the exact values are denoted by $\ex=\bfx-\xs$, $\ev=\bfv-\vs$, $\eu=\bfu-\us$ and their corresponding finite element functions on the interpolated surface $\Gamma_h[\xs]$ by $e_x$, $e_v$, $e_u$, respectively. We obtain the error equations by subtracting \eqref{defect vectors} from \eqref{eq:matrix--vector form} and  \eqref{matrix-form-X-star} from \eqref{matrix-form-X}:
\begin{subequations} 
\label{eq:error equations}
\begin{align}
\label{eq:error eq - v}
	\bfK(\bfx) \ev 
	=&\ -  \big( \bfK(\bfx)-\bfK(\xs) \big) \vs  \\
	\nonumber 
	&\ + \big(\bfg(\bfx,\bfu) - \bfg(\xs,\us)\big) - \bfM(\xs)\dv , \\
\label{eq:error eq - u}
	\nonumber \bfM(\bfx)\doteu+ \bfA(\bfx)\eu 
	=&\ -  \big( \bfM(\bfx)-\bfM(\xs) \big) \dotus  \\
	&\ - \big( \bfA(\bfx)-\bfA(\xs) \big) \us \\
	\nonumber &\ + \big(\bff(\bfx,\bfu) - \bff(\xs,\us)\big) - \bfM(\xs)\du , \\
\label{eq:error eq - x}
    \dotex =&\ \ev . \qquad \qquad \qquad \qquad \qquad \qquad 
\end{align}
\end{subequations}
\ecl
We note further that $\ex(0)=0$ and $\ev(0)=0$, but in general $\eu(0)\ne 0$.

For estimating the defect $\dv$, we use the norm given by
\begin{equation*}
	\|\dv\|_{\star,\xs}^2 := \dv^T \bfM(\xs)\bfK(\xs)\inv \bfM(\xs) \dv .
\end{equation*}
By \cite[Eq.~(5.5)]{KLLP2017},  this equals
the following dual norm for the corresponding finite element function $d_v\in S_h[\bfx^*]^3$, which has  the vector of nodal values $\dv$,
$$
\|\dv\|_{\star,\xs} = \|d_v\|_{H_h\inv(\Gamma_h[\bfx^*])} := 
\sup_{0\ne\varphi_h\in S_h[\xs]^3} \frac{ \int_{\Gamma_h[\xs]} d_v  \cdot\varphi_h } { \| \varphi_h \|_{H^1(\Gamma_h[\xs])} } \, .
$$
For estimating the defect $\du$, which is the nodal vector of  a finite element function $d_u\in S_h(\bfx^*)^4$, we use the $L^2$ norm
$$
\| \du \|_{\bfM(\xs)}^2 = \du^T \bfM(\xs) \du = \| d_u \|_{L^2(\Gamma_h[\xs])}^2.
$$
The errors will be estimated in the $H^1$ norm:
$$
\normK{\bfe}^2 = \bfe^T  \bfK(\xs) \bfe = \| e \|_{H^1(\Gamma_h[\xs])}^2.
$$

\subsection{Stability estimate}

The following result provides the key stability estimate.

\begin{proposition}
\label{proposition:stability - coupled problem}
    Assume that, for some $\kappa$ with $1<\kappa \le k$, the defects are bounded by 
    \begin{equation}
    \label{eq:assumed defect bounds}
        \begin{aligned}
			\|\dv(t)\|_{\star,\xs(t)}  &\leq c h^\kappa , \\
			\|\du(t)\|_{\bfM(\xs(t))} &\leq c h^\kappa , 
        \end{aligned}
        \qquad \hbox{for }\  0\le t \le T ,
    \end{equation}
    and that also the errors in the initial values satisfy
    $$
     \| \eu(0) \|_{\bfK(\xs(0))} \le c h^\kappa.
    $$
    Then, there exists  $h_0>0$ such that the following stability estimate holds for all $h\leq h_0$ and $0\le t \le T$:
	\begin{equation}
	\label{eq:stability bound}
		\begin{aligned}
			& \| \ex(t) \|_{\bfK(\xs(t))}^2 +  \| \ev(t) \|_{\bfK(\xs(t))}^2 +  \| \eu(t) \|_{\bfK(\xs(t))}^2 \\
			& \leq  C   \| \eu(0) \|_{\bfK(\xs(0))}^2+C \max_{0\le s \le t}\|\dv(s)\|_{\star,\xs(s)}^2  
			+ C \int_0^t\normMs{\du(s)}^2 \d s ,
		\end{aligned}
	\end{equation}
	where $C$ is independent of $h$ and $t$, but depends on the final time $T$.
\end{proposition}

From the interpolation error bound of Proposition 2.7 in \cite{Demlow2009} and the Ritz map error bound of \bbk Theorem~6.3 \ebk in \cite{Kovacs2017} we know that $ \| \eu(0) \|_{\bfK(\xs(0))} = O(h^k)$.
We note that therefore the error functions $e_x(\cdot,t), e_v(\cdot,t), e_u(\cdot,t) \in S_h[\xs\t]$ corresponding to the nodal values $\ex, \ev, \eu$, respectively, are bounded by
\begin{equation*} 
	\begin{aligned}
		\|e_x(\cdot,t)\|_{H^1(\Gamma_h[\xs\t])}
		&\ + \|e_v(\cdot,t)\|_{H^1(\Gamma_h[\xs\t])} \\
		&\ +\|e_u(\cdot,t)\|_{H^1(\Gamma_h[\xs\t])} \leq C h^\kappa .
	\end{aligned}
\end{equation*}
In Section~\ref{section:Defect} we will show that the assumed defect bounds \eqref{eq:assumed defect bounds} are indeed satisfied with $\kappa = k \geq 2$, provided the exact solution is sufficiently smooth.

\begin{proof}
The basic procedure of the proof follows that of \cite{KLLP2017}: it uses energy estimates for the error equations in the matrix--vector formulation and  relies on the technical lemmas of Section~\ref{section: aux}. Since we need uniform-in-time $H^1$-norm error bounds (in order to control the $W^{1,\infty}$ norm of the errors in $u$ via an inverse inequality), we test with the time derivative of the error vector. In contrast, in \cite{KLLP2017} we just tested with the error vector as this allowed us to obtain $L^2(H^1)$-norm estimates for $e_u$, which was sufficient for a nonlinearity $f(u,\nabla_\Gamma u)$ that is globally Lipschitz-bounded in the second variable. This global Lipschitz continuity is, however, not satisfied for the nonlinearity considered here.
	
We test \eqref{eq:error eq - v} with $\ev$, which yields
\begin{equation}
\label{eq:error eq tested - v}
	\begin{aligned}
		\|\ev\|_{\bfK(\bfx)}^2 = \ev^T \bfK(\bfx) \ev 
		=&\ -  \ev^T \big( \bfK(\bfx)-\bfK(\xs) \big) \vs  \\
		&\ + \ev^T \big(\bfg(\bfx,\bfu) - \bfg(\xs,\us)\big) - \ev^T \bfM(\xs)\dv,
	\end{aligned}
\end{equation}
while we test \eqref{eq:error eq - u} with $\doteu$, which yields 
\begin{equation} 
\label{eq:error eq tested - u}
	\begin{aligned}
		\doteu^T\bfM(\bfx)\doteu+ \doteu^T\bfA(\bfx)\eu 
		=&\ -  \doteu^T\big( \bfM(\bfx)-\bfM(\xs) \big) \dot{\bf u}^*  \\
		&\ - \doteu^T\big( \bfA(\bfx)-\bfA(\xs) \big) \us \\
		&\ + \doteu^T\big(\bff(\bfx,\bfu) - \bff(\xs,\us)\big) \\
		&\ - \doteu^T\bfM(\xs)\du .
	\end{aligned}
\end{equation}

Let $t^*\in(0,T]$ be the maximal time such that the following inequalities hold:
\begin{equation}
\label{eq:assumed bounds - 2}
    \begin{aligned}
    	\| e_x(\cdot,t)\|_{W^{1,\infty}(\Ga_h[\xs(t)])} &\leq h^{(\kappa-1)/2} , \\
    	\| e_v(\cdot,t)\|_{W^{1,\infty}(\Ga_h[\xs(t)])} &\leq h^{(\kappa-1)/2} , \\
        \|e_u(\cdot,t) \|_{W^{1,\infty}(\Ga_h[\xs(t)])} &\leq h^{(\kappa-1)/2} , 
    \end{aligned} \qquad \textrm{ for } \quad t\in[0,t^*].
\end{equation}
Note that $t^*>0$ since initially $e_x(\cdot,0)=0, e_v(\cdot,0)=0$ and, by an inverse inequality and the higher-order error bound for the Ritz map \cite{Kovacs2017}, we have $\|e_u(\cdot,0)\|_{W^{1,\infty}(\Ga_h[\xs(0)])}\le ch^{-1}\|e_u(\cdot,0)\|_{H^{1}(\Ga_h[\xs(0)])}\le Ch^{k-1}$. 
We first prove the stated error bounds for $0\leq t \leq t^*$. At the end of the proof we will show that in fact $t^*$ coincides with $T$.

\bcl Since the exact surface functions $X,v,u$ are assumed smooth, the associated finite element functions $x_h^*(\cdot,t),v_h^*(\cdot,t),u_h^*(\cdot,t)$ on the interpolated surface 
$\Gamma_h[\xs(t)]$ have $W^{1,\infty}$ norms that are bounded independently of $h$, for all $t\in[0,T]$. \bbk 
The bounds~\eqref{eq:assumed bounds - 2} together with Lemma~\ref{lemma:theta-independence}
then imply that the $W^{1,\infty}$ norms of the ESFEM functions $x_h(\cdot,t),v_h(\cdot,t),u_h(\cdot,t)$ on the discrete surface $\Gamma_h[\bfx(t)]$ are also bounded independently of $h$ and $t\in[0,t^*]$, and so are their lifts to the interpolated surface $\Gamma_h[\xs(t)]$. In particular,  it will be important that the discrete velocity $v_h$ with nodal vector $\bfv=\dot\bfx$ satisfies \eqref{vh-bound}.

In the following $c$ and $C$ are generic constants that take different values on different occurrences.

(A) {\it Estimates for the surface PDE:}
We estimate the terms of \eqref{eq:error eq tested - u} separately, with \bcl Lemmas~\ref{lemma:matrix differences} and~\ref{lemma:theta-independence} and the ensuing bounds \eqref{norm-equiv}--\eqref{eq:matrix derivatives} as our main tools.  The condition of Lemma~\ref{lemma:theta-independence} follows from the first estimate in \eqref{eq:assumed bounds - 2}, $\| e_x(\cdot,t)\|_{W^{1,\infty}(\Ga_h[\xs(t)])} \leq h^{(\kappa-1)/2} \leq \frac12$ for $h \leq h_0$ sufficiently small. 

(i) By the definition of the $\bfM$-norm we have
\begin{align*}
  \doteu^T \bfM(\bfx)\doteu = \normMnum{\doteu}^2 .
\end{align*}

(ii) 
The symmetry of $\bfA(\bfx)$ and the bound \eqref{eq:matrix derivatives} yield  
\begin{align*}
	\doteu^T \bfA(\bfx)\eu 
	&=- \half \eu^T \diff \big(\bfA(\bfx)\big) \eu 
	+ \half \diff \Big(\eu^T \bfA(\bfx)\eu\Big)  \\
	&\ge -c \normAnum{\eu}^2  
	+ \half \diff \normAnum{\eu}^2 .
\end{align*}

(iii) Using \bcl Lemmas~\ref{lemma:matrix differences} and \ref{lemma:theta-independence}, we now show
for the first term on the right-hand side that
\begin{align} \label{eq:est-first}
	-\doteu^T  \big( \bfM(\bfx)-\bfM(\xs) \big) \dotus 
	&\le C\,  \normM{\doteu}  \normA{\ex}
\end{align}
with a constant $C$ that is independent of $h$.

By Lemma~\ref{lemma:matrix differences} we have with $\Gamma_h^\theta=\Gamma_h[\xs+\theta\ex]$
\begin{align*}
-\doteu^T  \big( \bfM(\bfx)-\bfM(\xs) \big) \dotus  &= 
- \int_0^1 \int_{\Gamma_h^\theta} 
\dot e_u^\theta\, (\nabla_{\Gamma_h^\theta}\cdot e_x^\theta) \,\partial_h^\bullet u_h^{*,\theta}\,\d\theta 
\\
&\le  \int_0^1 \|  \dot e_u^\theta \|_{L^2(\Gamma_h^\theta)}\,
\| \nabla_{\Gamma_h^\theta}\cdot e_x^\theta  \|_{L^2(\Gamma_h^\theta)}\,
\| \partial_h^\bullet u_h^{*,\theta}  \|_{L^\infty(\Gamma_h^\theta)}\, \d\theta.
\end{align*}
Here we note that by Lemma~\ref{lemma:theta-independence} together with the bound \eqref{eq:assumed bounds - 2},
\begin{align*}
\| \nabla_{\Gamma_h^\theta}\cdot e_x^\theta  \|_{L^2(\Gamma_h^\theta)} &\le
\| \nabla_{\Gamma_h^\theta} e_x^\theta  \|_{L^2(\Gamma_h^\theta)} 
\\
&\le
c_2 \,\| \nabla_{\Gamma_h^0} e_x^0  \|_{L^2(\Gamma_h^0)} =
c_2 \,\| \nabla_{\Gamma_h[\xs]} e_x  \|_{L^2(\Gamma_h[\xs])}
\end{align*}
and also
$$
\|  \dot e_u^\theta \|_{L^2(\Gamma_h^\theta)} \le c_2 \,\|  \dot e_u^0 \|_{L^2(\Gamma_h^0)} =
c_2\, \|  \dot e_u \|_{L^2(\Gamma_h[\xs])}
$$
and 
$$
\| \partial_h^\bullet u_h^{*,\theta}  \|_{L^\infty(\Gamma_h^\theta)} \le
c_\infty\, \| \partial_h^\bullet u_h^{*,0}  \|_{L^\infty(\Gamma_h^0)} =
c_\infty\, \| \partial_h^\bullet u_h^{*}  \|_{L^\infty(\Gamma_h[\xs])}.
$$
We now show that the last term is bounded by a constant that depends only on the solution regularity.
For the material derivative of the Ritz map we have
\begin{align*}
&\| \partial_h^\bullet u_h^{*} \|_{L^\infty(\Gamma_h[\xs])} \le 
c \| (\partial_h^\bullet u_h^{*})^l  \|_{L^\infty(\Gamma[X(\cdot,t)])}
 \\
&\le c \| (\partial_h^\bullet u_h^{*})^l  - I_h \partial^\bullet u\|_{L^\infty(\Gamma[X(\cdot,t)])} 
\\&\quad +
c \| I_h  \partial^\bullet u -  \partial^\bullet u \|_{L^\infty(\Gamma[X(\cdot,t)])}
+ c \|   \partial^\bullet u \|_{L^\infty(\Gamma[X(\cdot,t)])} \\
&\le
 \frac c h \| (\partial_h^\bullet u_h^{*})^l  - I_h \partial^\bullet u\|_{L^2(\Gamma[X(\cdot,t)])} 
 \\ &\quad +
c \| I_h  \partial^\bullet u -  \partial^\bullet u \|_{L^\infty(\Gamma[X(\cdot,t)])}
+ c \|    \partial^\bullet u \|_{L^\infty(\Gamma[X(\cdot,t)])}
\\
&\le  \frac c h \| (\partial_h^\bullet u_h^{*})^l  -  \partial^\bullet u\|_{L^2(\Gamma[X(\cdot,t)])} +
\frac c h \|  \partial^\bullet u - I_h \partial^\bullet u\|_{L^2(\Gamma[X(\cdot,t)])} 
\\ &\quad +
c \| I_h  \partial^\bullet u -  \partial^\bullet u \|_{L^\infty(\Gamma[X(\cdot,t)])}
+ c \|   \partial^\bullet u \|_{L^\infty(\Gamma[X(\cdot,t)])}
\\
&\le Ch^{k-1} + Ch^k + Ch^k +  C,
\end{align*}
where we used the norm equivalence for the lift operator (see \cite{Demlow2009}) in the first inequality,
an inverse inequality in the second inequality, and the known error bounds for interpolation (see \cite[Proposition~2]{Demlow2009}) 
and for the Ritz map (see \cite[Theorem~6.3]{Kovacs2017}) in the last inequality.
Combining the estimates above yields 
\begin{align*}
-\doteu^T  \big( \bfM(\bfx)-\bfM(\xs) \big) \dotus &\le C \, \|  \dot e_u \|_{L^2(\Gamma_h[\xs])}\, \| \nabla_{\Gamma_h[\xs]} e_x  \|_{L^2(\Gamma_h[\xs])} 
\\
&= C\,  \normM{\doteu}  \normA{\ex},
\end{align*}
which is \eqref{eq:est-first}.

(iv)  Direct estimation of the second term on the right-hand side of \eqref{eq:error eq tested - u}, like in (iii), would yield a bound with a factor $\normA{\doteu}$, which cannot be controlled as required. The second term is therefore first rewritten, using the product rule of differentiation as 
\begin{align} \label{RA1}
	 - \doteu^T \big( \bfA(\bfx)-\bfA(\xs) \big) \us =& -  \frac{\d}{\d t} \Big(\eu^T \big( \bfA(\bfx) \! - \! \bfA(\xs) \big) \us\Big)
	\\
	\nonumber
	&\  
	+\eu^T \big( \bfA(\bfx) \! - \! \bfA(\xs) \big) \dotus \! + \eu^T \frac{\d}{\d t}\big( \bfA(\bfx) \! - \! \bfA(\xs) \big) \us \! .
\end{align}
Among these terms, the first term on the second line is estimated via \bcl Lemmas~\ref{lemma:matrix differences} and \ref{lemma:theta-independence}, like in (iii), as
$$
\eu^T \big( \bfA(\bfx) \! - \! \bfA(\xs) \big) \dotus \le C\, \normA{\eu} \,\normA{\ex}.
$$
\noindent The last term is rewritten, using Lemma~\ref{lemma:matrix differences} and the Leibniz formula, as \ecl
\begin{equation}
\label{wTdtAz}
	\begin{aligned}
		\bfw^T \Big(\diff \big( \bfA(\bfx)-\bfA(\xs) \big) \Big) \bfz 
		&=\frac{\d}{\d t}\int_0^1 \int_{\Ga_h^\theta} \nb_{\Ga_h^\theta} w_h^\theta \cdot (D_{\Ga_h^\theta} e_x^\theta) \nb_{\Ga_h^\theta}  z_h^\theta \ \d\theta  \\
		&= 
		\int_0^1 \int_{\Ga_h^\theta} \mat_{\Gamma_h^\theta} \big(\nb_{\Ga_h^\theta} w_h^\theta \big) \cdot (D_{\Ga_h^\theta} e_x^\theta)  \nb_{\Ga_h^\theta}  z_h^\theta \ \d\theta  \\
		&\quad +
		\int_0^1 \int_{\Ga_h^\theta}  \nb_{\Ga_h^\theta} w_h^\theta \cdot \mat_{\Gamma_h^\theta} \big( D_{\Ga_h^\theta} e_x^\theta \big) \nb_{\Ga_h^\theta}  z_h^\theta \ \d\theta  \\
		&\quad 
		+ \int_0^1 \int_{\Ga_h^\theta}  \nb_{\Ga_h^\theta} w_h^\theta \cdot (D_{\Ga_h^\theta} e_x^\theta) \mat_{\Gamma_h^\theta} \big( \nb_{\Ga_h^\theta}  z_h^\theta \big) \ \d\theta  \\
		&\quad 
		+ \int_0^1 \int_{\Ga_h^\theta}  \nb_{\Ga_h^\theta} w_h^\theta \cdot (D_{\Ga_h^\theta} e_x^\theta) \nb_{\Ga_h^\theta}  z_h^\theta \, (\nb_{\Gamma_h^\theta} \cdot \blueon v_{\Ga_h^\theta}\blueoff) \d\theta \\
		&=: \int_0^1 (J_1^\theta+J_2^\theta +J_3^\theta+J_4^\theta)\,\d\theta .
	\end{aligned}
\end{equation}
Here, \bcl $v_{\Ga_h^\theta}(\cdot,t)$ \ecl is the velocity of $\Gamma_h^\theta(t)$ (as a function of $t$), which is the finite element function in $S_h[\xs(t)+\theta \ex(t)]$ with nodal vector $\dotxs(t) + \theta \dot\bfe_\bfx(t) = \vs(t) + \theta \ev(t)$. Related to this velocity, $\mat_{\Gamma_h^\theta}$ denotes the material derivative on ${\Gamma_h^\theta}$.  We denote by $w_h^\theta(\cdot,t)$ and $z_h^\theta(\cdot,t)$ the finite element functions on $\Gamma_h^\theta(t)$ with the time-independent nodal vectors $\bfw$ and $\bfz$, respectively.  We thus have
\begin{equation}
\label{vhtheta}
	\bcl v_{\Ga_h^\theta} \ecl = v_h^{*,\theta} + \theta e_v^\theta .
\end{equation} 
By using the identity (see \cite[Lemma~2.6]{DziukKronerMuller}) 
\begin{equation}
\label{eq:mat-grad formula}
	\begin{aligned} 
		\mat_{\Gamma_h^\theta} (\nabla_{\Gamma_h^\theta} w_h^\theta ) 
		&= \nabla_{\Gamma_h^\theta} \mat_{\Gamma_h^\theta}w_h^\theta 
		-  
		\bigl(\nabla_{\Gamma_h^\theta} v_{\Ga_h^\theta} -	\nu_h^\theta(\nu_h^\theta)^T (\nabla_{\Gamma_h^\theta } v_{\Ga_h^\theta} )^T \bigr) \nabla_{\Gamma_h^\theta} w_h^\theta \\
		&= 
		- 
		\bigl(\nabla_{\Gamma_h^\theta} v_{\Ga_h^\theta} -	\nu_h^\theta(\nu_h^\theta)^T(\nabla_{\Gamma_h^\theta } v_{\Ga_h^\theta} )^T \bigr) \nabla_{\Gamma_h^\theta} w_h^\theta , 
	\end{aligned}
\end{equation}
where we have used the property $\mat_{\Gamma_h^\theta}w_h^\theta=0$ in the last equality, we obtain 
\begin{align*} 
	|J_1^\theta| &\leq
	\int_0^1 \int_{\Ga_h^\theta} \Big| \Big( \bigl(\nabla_{\Gamma_h^\theta} v_{\Ga_h^\theta} -	\nu_h^\theta(\nu_h^\theta)^T(\nabla_{\Gamma_h^\theta } v_{\Ga_h^\theta} )^T \bigr) \nabla_{\Gamma_h^\theta} w_h^\theta \Big) \cdot (D_{\Ga_h^\theta} e_x^\theta)  \nb_{\Ga_h^\theta}  z_h^\theta \Big| \d\theta  \\
	&\leq c \int_0^1 \! \|\nabla_{\Gamma_h^\theta} w_h^\theta \|_{L^2(\Gamma_h^\theta)}
	\|\nabla_{\Gamma_h^\theta} v_h^{*,\theta} \|_{L^\infty(\Gamma_h^\theta)}
	\|D_{\Gamma_h^\theta} e_x^\theta \|_{L^2(\Gamma_h^\theta)}
	\|\nabla_{\Gamma_h^\theta} z_h^\theta \|_{L^\infty(\Gamma_h^\theta)}  \d\theta \\
	& + c \int_0^1 \!\! \|\nabla_{\Gamma_h^\theta} w_h^\theta \|_{L^2(\Gamma_h^\theta)}
	\|\nabla_{\Gamma_h^\theta} e_v^\theta \|_{L^2(\Gamma_h^\theta)}
	\|D_{\Gamma_h^\theta} e_x^\theta \|_{L^\infty(\Gamma_h^\theta)}
	\|\nabla_{\Gamma_h^\theta} z_h^\theta \|_{L^\infty(\Gamma_h^\theta)} \d\theta \\
	&\leq c \|\nabla_{\Gamma_h[\xs]} w_h \|_{L^2(\Gamma_h[\xs])}
	\|D_{\Gamma_h[\xs]} e_h \|_{L^2(\Gamma_h[\xs])}
	\|\nabla_{\Gamma_h[\xs]} z_h \|_{L^\infty(\Gamma_h[\xs])} \\
	&\ + c \|\nabla_{\Gamma_h[\xs]} w_h \|_{L^2(\Gamma_h[\xs])}
	\|\nabla_{\Gamma_h[\xs]} e_v \|_{L^2(\Gamma_h[\xs])}
	\|\nabla_{\Gamma_h[\xs]} z_h \|_{L^\infty(\Gamma_h[\xs])} \\
	&\leq c \normA{\bfw} \big( \normA{\ex} + \normA{\ev} \big)
	\|z_h\|_{W^{1,\infty}(\Gamma_h[\xs])} ,
\end{align*}
where $z_h=z_h^0$ and where we have used the decomposition \eqref{vhtheta}, Lemma~\ref{lemma:theta-independence}, and the bounds \eqref{eq:assumed bounds - 2}. Similarly, we obtain 
\begin{align*} 
	|J_2^\theta| + |J_3^\theta| + |J_4^\theta| &\leq c \normA{\bfw} \big( \normA{\ex} + \normA{\ev} \big)
	\|z_h\|_{W^{1,\infty}(\Gamma_h[\xs])} , 
\end{align*}
where for the estimate for $J_2^\theta$ we used the equality analogous to \eqref{eq:mat-grad formula}:
\begin{equation}
\label{eq:mat-D formula}
	\begin{aligned}
		\mat_{\Gamma_h^\theta} (D_{\Ga_h^\theta} e_x^\theta) &= \mat_{\Gamma_h^\theta} \Big( \textnormal{tr}(\nb_{\Ga_h^\theta} e_x^\theta) - \big( \nb_{\Ga_h^\theta} e_x^\theta+(\nb_{\Ga_h^\theta} e_x^\theta)^T \big) \Big) \\
		=&\ D_{\Ga_h^\theta} (\mat_{\Gamma_h^\theta} e_x^\theta) 
		+ \textnormal{tr}(\bar E^\theta) - (\bar E^\theta+(\bar E^\theta)^T) ,
	\end{aligned}
\end{equation}
with $\bar E^\theta= - \bigl(\nabla_{\Gamma_h^\theta} v_{\Ga_h^\theta} -	\nu_h^\theta(\nu_h^\theta)^T(\nabla_{\Gamma_h^\theta } v_{\Ga_h^\theta} )^T \bigr) \nabla_{\Gamma_h^\theta} \bbk e_x^\theta\ebk$, as follows from \cite[Lemma~2.6]{DziukKronerMuller} and the definition of the first order linear differential operator $D_{\Gamma_h^\theta}$. 

\noindent Further using $\mat_{\Gamma_h^\theta} e_x^\theta = e_v^\theta$ (since for the nodal vectors we have $\dotex=\ev$), we altogether obtain the bound
\begin{align*}
	&\ - \doteu^T \big( \bfA(\bfx)-\bfA(\xs) \big) \us \\
	&\ \leq  -\frac{\d}{\d t} \Big(\eu^T \big( \bfA(\bfx)-\bfA(\xs) \big) \us\Big) + c \normA{\eu} \big( \normK{\ev} + \normK{\ex} \big)
	 .
\end{align*}
%
%
(v) 
The term containing the nonlinearity $\bff$ can be written as
\begin{align*}
	\doteu^T \big(\bff(\bfx,\bfu) - \bff(\xs,\us)\big) &= \int_{\Ga_h^1} \!  f(u_h,\nabla_{\Ga_h^1} u_h) \cdot \dot e_u^1- \int_{\Ga_h^0} \!  f(u_h^\ast,\nabla_{\Ga_h^0} u_h^\ast) \cdot \dot e_u^0 ,
\end{align*}
where the function $f$ is given by the right-hand side  of \eqref{weak form} for $u=(\n,H)$; $f:\R^4 \times \R^{3\times 4} \to \R^4$ is smooth and therefore locally Lipschitz continuous.

We estimate this term in the following way. Using the abbreviation
\bcl
\begin{equation}
\label{uhtheta}
	u_{\Ga_h^\theta} :=  u_h^{*,\theta}+\theta e_u^\theta = \sum_{j=1}^N (u_j^*+\theta (\eu)_j) \, \phi_j[\xs+\theta \ex]  ,
\end{equation}
\ecl
we have
\begin{align*}
	\doteu^T \big(\bff(\bfx,\bfu) - \bff(\xs,\us)\big) &=
	\int_0^1 \frac\d{\d\theta} \int_{\Gamma_h^\theta} f(u_{\Ga_h^\theta}, \nabla_{\Gamma_h^\theta}u_{\Ga_h^\theta}) \cdot \dot e_u^\theta\, \d\theta.
\end{align*}
Here we apply the Leibniz formula, noting that $e_x^\theta$ is the velocity of the surface $\Gamma_h^\theta$ considered as a function of $\theta$ and that we have the vanishing material derivative
$$
\mat_\theta \dot e_u^\theta = \sum_{j=1}^N \frac{\d}{\d\theta}(\doteu)_j \, \phi_j[\xs+\theta \ex] =0.
$$
So we obtain
\begin{align*}
	& \doteu^T \big(\bff(\bfx,\bfu) - \bff(\xs,\us)\big) \\
	& = \
	\int_0^1\int_{\Gamma_h^\theta} \Bigl( \pa_\theta^\bullet \big( f(u_{\Ga_h^\theta}, \nabla_{\Gamma_h^\theta}u_{\Ga_h^\theta}) \big) \cdot \dot e_u^\theta + f(u_{\Ga_h^\theta}, \nabla_{\Gamma_h^\theta}u_{\Ga_h^\theta}) \cdot \dot e_u^\theta \, (\nabla_{\Gamma_h^\theta}\cdot e_x^\theta) \Bigr)\d\theta.
\end{align*}
Here we proceed further using the chain rule
\begin{align*}
	\partial_\theta^\bullet f(u_{\Ga_h^\theta}, \nabla_{\Gamma_h^\theta}u_{\Ga_h^\theta}) 
	&= \partial_1 f(u_{\Ga_h^\theta}, \nabla_{\Gamma_h^\theta}u_{\Ga_h^\theta}) \,  \partial_\theta^\bullet u_{\Ga_h^\theta} \\
	&\ + \partial_2 f(u_{\Ga_h^\theta}, \nabla_{\Gamma_h^\theta}u_{\Ga_h^\theta}) \,  \partial_\theta^\bullet (\nabla_{\Gamma_h^\theta}u_{\Ga_h^\theta})
\end{align*}
and observe the following:  by the $W^{1,\infty}$ bound for the exact solution $u^*$, and for $e_u$ in \eqref{eq:assumed bounds - 2} (and hence for $e_u^\theta$ by 
Lemma~\ref{lemma:theta-independence}), 
we have on recalling \eqref{uhtheta} that $u_{\Ga_h^\theta}$ and its gradient take values in a bounded set. Since $f$ is smooth, we therefore have
$$
	\|  \partial_i f(u_{\Ga_h^\theta}, \nabla_{\Gamma_h^\theta}u_{\Ga_h^\theta}) \|_{L^\infty(\Gamma_h^\theta)} \le C, \qquad i=1,2.
	$$
From $u_{\Ga_h^\theta}=u_h^* + \theta e_u^\theta$ and the identity $\mat_\theta e_u^\theta = 0$ (obtained by the same argument as for $\dot e_u^\theta$ above) we note
$$
\partial_\theta^\bullet u_{\Ga_h^\theta} = e_u^\theta.
$$
We further use the relation, see \cite[Lemma~2.6]{DziukKronerMuller},
\begin{equation}
\label{eq:mat grad interchange}
	\partial_\theta^\bullet \big( \nabla_{\Gamma_h^\theta}u_{\Ga_h^\theta} \big) =
	\nabla_{\Gamma_h^\theta} \big( \partial_\theta^\bullet u_{\Ga_h^\theta} \big) 
	- \Big( \nabla_{\Gamma_h^\theta} e_x^\theta - \nu_h^\theta (\nu_h^\theta)^T (\nabla_{\Gamma_h^\theta} e_x^\theta)^T \Big) \nabla_{\Gamma_h^\theta}u_{\Ga_h^\theta}.
\end{equation}
We then have, on inserting \eqref{uhtheta} and using once again
Lemma~\ref{lemma:theta-independence} and the bounds in \eqref{eq:assumed bounds - 2},
\begin{align*}
	& \doteu^T \big(\bff(\bfx,\bfu) - \bff(\xs,\us)\big)
	\\
	&= \int_0^1\int_{\Gamma_h^\theta} \dot e_u^\theta \Bigl( \bbk f(u_{\Ga_h^\theta}, \nabla_{\Gamma_h^\theta}u_{\Ga_h^\theta}) \cdot  (\nabla_{\Gamma_h^\theta}\cdot e_x^\theta) \ebk + \partial_1 f(u_{\Ga_h^\theta}, \nabla_{\Gamma_h^\theta}u_{\Ga_h^\theta}) e_u^\theta 
	\\
	 &\ +  \partial_2 f(u_{\Ga_h^\theta}, \nabla_{\Gamma_h^\theta}u_{\Ga_h^\theta})
	\bigl( \nabla_{\Gamma_h^\theta} e_u^\theta  
	- \big( \nabla_{\Gamma_h^\theta} e_x^\theta - \nu_h^\theta (\nu_h^\theta)^T (\nabla_{\Gamma_h^\theta} e_x^\theta)^T \big) \nabla_{\Gamma_h^\theta}u_{\Ga_h^\theta} \bigr)\Bigr) \d\theta
	\\
	\leq  &\ c \| \dot e_u \|_{L^2(\Gamma_h[\xs])} \Bigl( \| e_u \|_{L^2(\Gamma_h[\xs])} 
	\\
	 &\ + \| \nabla_{\Gamma_h[\xs])} e_u \|_{L^2(\Gamma_h[\xs])} + 
	\| \nabla_{\Gamma_h[\xs])} e_x \|_{L^2(\Gamma_h[\xs])}\,
	\| \nabla_{\Gamma_h[\xs])} u_h^* \|_{L^\infty(\Gamma_h[\xs])} \Bigr)
	\\
	\leq  &\ c  \normM{\doteu} \Bigl( \normK{\eu} + \normA{\ex} \Bigr).
\end{align*}

(vi) The defect term is bounded using the Cauchy--Schwarz inequality:
$$
	-\doteu^T\bfM(\xs)\du \le \normM{\doteu} \normM{\du}.
$$

\noindent The combination of the above inequalities \blueon and the norm equivalence \eqref{norm-equiv} \blueoff yield the bound
\begin{align} \label{err-est-eu-1}
	\begin{aligned}
		&\normMnum{\doteu}^2 + \half \diff \normAnum{\eu}^2 \\
		& \le 
		c \normA{\eu}^2
		+c  \normM{\doteu}  \normK{\ex}
		\\
		&\quad + c \normA{\eu} \big( \normK{\ev} + \normK{\ex} \big)   \\
		&\quad - \diff \Big(\eu^T \big( \bfA(\bfx)-\bfA(\xs) \big) \us\Big)
		 \\
		&\quad 
		+c \normM{\doteu} \big( \normK{\eu} + \normK{\ex} \big) \\
		&\quad +
		\normM{\doteu} \normM{\du}.	\end{aligned}
\end{align}
Here we note for the first term on the left-hand side that, by \eqref{eq:matrix derivatives},
\begin{align*}
	\half \diff \normMnum{\eu}^2 &= \eu^T \bfM(\bfx) \doteu \! + \! \half \eu^T \diff \bfM(\bfx) \eu 
	\leq  \half \normMnum{\doteu}^2 + c \normMnum{\eu}^2.
\end{align*}Estimating further, using the norm equivalence \eqref{norm-equiv} and Young's inequality and absorptions (with $h\leq h_0$ for a sufficiently small $h_0$) into $\normM{\doteu}^2$, we obtain the following estimate:
\begin{align} 
	\begin{aligned}
		&\half \diff \normMnum{\eu}^2 + \half \diff \normAnum{\eu}^2 \\
		& \leq c \normK{\ex}^2 + c \normK{\ev}^2 + c \normK{\eu}^2 \\
		&\quad - \diff \Big(\eu^T \big( \bfA(\bfx)-\bfA(\xs) \big) \us\Big)
	\\
		& \quad 
		+ c\normM{\du}^2 .
	\end{aligned}
\end{align}

After integration in time, using the above estimates together with Young's inequality, assumption \eqref{eq:assumed bounds - 2} and the norm equivalence  \eqref{norm-equiv}, and recalling that $\bfM(\xs)+\bfA(\xs)=\bfK(\xs)$, we obtain
\begin{align} \label{err-est-eu-2}
	\begin{aligned}
		&\half \normKt{\eu(t)}^2  
	        \leq \half \normKo{\eu(0)}^2  \\
		& \quad + c \int_0^t \big( \normKs{\ev(s)}^2 + \normKs{\eu(s)}^2 + \normKs{\ex(s)}^2 \big) \d s \\ 
		&\quad -\eu(t)^T \big( \bfA(\bfx(t))-\bfA(\xs(t)) \big) \us(t) 
	\\
		&\quad 
		+ c \int_0^t \normMs{\du(s)}^2 \d s \\
		%
		%
		& \leq  \half \normKo{\eu(0)}^2\\
		&\quad + c \int_0^t \big( \normKs{\ex(s)}^2 + \normKs{\ev(s)}^2+ \normKs{\eu(s)}^2 \big) \d s \\ 
		&\quad + \frac{1}{6}  \normAt{\eu(t)}^2 + c \normKt{\ex(t)}^2 \\ 
		&\quad  + c \int_0^t \normMs{\du(s)}^2 \d s ,
	\end{aligned}
\end{align}
\bbk where the term involving the stiffness matrices was estimated analogously as (iii), and with the help of the Young inequality. \ebk 

Absorption of the term containing $\eu(t)$  (in the case of a sufficiently small $h$) finally yields
\begin{align} \label{eq:final estimate PDE}
	\begin{aligned}
		  \normKt{\eu(t)}^2 & \le  
		c \int_0^t \big( \normKs{\ex(s)}^2 + \normKs{\ev(s)}^2 + \normKs{\eu(s)}^2 \big) \d s \\
		& \hspace{-2mm}  + c \,\normKt{\ex(t)}^2 + c\,\normKo{\eu(0)}^2+  c \int_0^t \normMs{\du(s)}^2 \d s  .
	\end{aligned}
\end{align}

(B) {\it Estimates for the velocity equation:} 
Testing \eqref{eq:error eq tested - v} with $\ev$  and using Lemma~\ref{lemma:matrix differences},
the Cauchy--Schwarz and Young inequalities, and the norm equivalence \eqref{norm-equiv} yields the bound 
\begin{align*}
	\tfrac12\normK{\ev}^2 \leq c \normK{\ex}^2 + \ev^T \big(\bfg(\bfx,\bfu) - \bfg(\xs,\us)\big) + c\|\dv\|_{\star,\xs}^2 .
\end{align*}
The term with the nonlinearity $\bfg$ is first rewritten as
\begin{align*}
	&\hspace{-5mm} \ev^T \big(\bfg(\bfx,\bfu) - \bfg(\xs,\us)\big) \\
	&= \int_{\Ga_h^1} g(u_h) e_v^1 - \int_{\Ga_h^0} g(u_h^*) e_v^0 \\
	&\ + \int_{\Ga_h^1} \fg(u_h,\nb_{\Ga_h^1} u_h) \cdot \nb_{\Ga_h^1} e_v^1 - \int_{\Ga_h^0} \fg(u_h^*,\nb_{\Ga_h^0} u_h^*) \cdot  \nb_{\Ga_h^0} e_v^0 ,
\end{align*}
where the functions $g:\R^4\to \R^3$ and $\fg:\R^4\times \R^{3\times 4}\to \R^{3\times 3}$ 
are smooth and hence locally Lipschitz-continuous functions as given by \eqref{weak form}.

We now use the same argument as in (v) for the nonlinearity in the equation for the dynamic variables $u$. Using formula \eqref{uhtheta} and the Leibniz formula, recalling that $e_x^\theta$ is the velocity of $\Gamma_h^\theta$ \bbk as a function of $\theta$, \ebk and also using that $\pa_\theta^\bullet e_v^\theta=0$, we again obtain
\begin{align*}
	&\ \ev^T \big(\bfg(\bfx,\bfu) - \bfg(\xs,\us)\big) \\
	=&\ \int_0^1 \int_{\Gamma_h^\theta} \pa_\theta^\bullet \big( g(u_{\Ga_h^\theta}) \big) \cdot e_v^\theta \d\theta 
	+ \int_0^1 \int_{\Gamma_h^\theta} \pa_\theta^\bullet \big( \fg(u_{\Ga_h^\theta}, \nabla_{\Gamma_h^\theta}u_{\Ga_h^\theta}) \big) \cdot \nabla_{\Gamma_h^\theta} e_v^\theta \d\theta \\
	&\ + \int_0^1 \int_{\Gamma_h^\theta} \fg(u_{\Ga_h^\theta}, \nabla_{\Gamma_h^\theta}u_{\Ga_h^\theta}) \cdot \pa_\theta^\bullet \big( \nabla_{\Gamma_h^\theta} e_v^\theta \big) \d\theta \\
	&\ + \int_0^1 \int_{\Gamma_h^\theta} \Big(  g(u_{\Ga_h^\theta}) \cdot e_v^\theta + \fg(u_{\Ga_h^\theta}, \nabla_{\Gamma_h^\theta}u_{\Ga_h^\theta}) \cdot  \nabla_{\Gamma_h^\theta} e_v^\theta \Big) (\nb_{\Gamma_h^\theta} \cdot e_x^\theta) \d\theta. \\
\end{align*}
Here we use the chain rule
\begin{align*}
	\bbk \partial_\theta^\bullet g(u_{\Ga_h^\theta}) = \ebk &\
	\bbk \partial_1 g(u_{\Ga_h^\theta}) \, \partial_\theta^\bullet u_{\Ga_h^\theta} \ebk , \\
	\partial_\theta^\bullet \fg(u_{\Ga_h^\theta}, \nabla_{\Gamma_h^\theta}u_{\Ga_h^\theta}) &=
	\partial_1 \fg(u_{\Ga_h^\theta}, \nabla_{\Gamma_h^\theta}u_{\Ga_h^\theta}) \,  \partial_\theta^\bullet u_{\Ga_h^\theta}
	\\ & +
	\partial_2 \fg(u_{\Ga_h^\theta}, \nabla_{\Gamma_h^\theta}u_{\Ga_h^\theta}) \,  \partial_\theta^\bullet (\nabla_{\Gamma_h^\theta}u_{\Ga_h^\theta}) .
\end{align*}
By the $W^{1,\infty}$ bound for $u_h^*$ and $e_u$ (by~\eqref{eq:assumed bounds - 2}), and hence for $e_u^\theta$ by Lemma~\ref{lemma:theta-independence}, the arguments of $g$ and $\fg$ take values in a bounded set. By the smoothness of both functions, we therefore have 
$$
	\|\partial_1 g(u_{\Ga_h^\theta}) \|_{L^\infty(\Gamma_h^\theta)} \le C  \andquad \|  \partial_i \fg(u_{\Ga_h^\theta}\nabla_{\Gamma_h^\theta}u_{\Ga_h^\theta}) \|_{L^\infty(\Gamma_h^\theta)} \le C, \qquad i=1,2.
$$
Using $\partial_\theta^\bullet u_{\Ga_h^\theta} = e_u^\theta$ and $\pa_\theta^\bullet e_v^\theta=0$, the relations \eqref{eq:mat grad interchange}, \eqref{eq:mat-D formula}, \eqref{uhtheta}, and using once again Lemma~\ref{lemma:theta-independence} and the bounds \eqref{eq:assumed bounds - 2}, we obtain
\begin{align*}
	& \ev^T \big(\bfg(\bfx,\bfu) - \bfg(\xs,\us)\big) \\
	&\ = \int_0^1 \int_{\Gamma_h^\theta} \partial_1 g(u_{\Ga_h^\theta}) \ e_u^\theta \cdot e_v^\theta \ \d\theta + \int_0^1 \int_{\Gamma_h^\theta} \Big( \partial_1 \fg(u_{\Ga_h^\theta}, \nabla_{\Gamma_h^\theta}u_{\Ga_h^\theta}) \, e_u^\theta \\
	&\ + \partial_2 \fg(u_{\Ga_h^\theta}, \nabla_{\Gamma_h^\theta}u_{\Ga_h^\theta}) \Big( \nabla_{\Gamma_h^\theta} e_u^\theta \! - \! \big( \nabla_{\Gamma_h^\theta} e_x^\theta \! - \! \nu_h^\theta (\nu_h^\theta)^T (\nabla_{\Gamma_h^\theta} e_x^\theta)^T \big) \nabla_{\Gamma_h^\theta}u_{\Ga_h^\theta} \Big) \! \cdot \! \nabla_{\Gamma_h^\theta} e_v^\theta \d\theta \\
	&\quad + \int_0^1 \int_{\Gamma_h^\theta} \fg(u_{\Ga_h^\theta}, \nabla_{\Gamma_h^\theta}u_{\Ga_h^\theta}) \cdot \Big( 
	- \big( \nabla_{\Gamma_h^\theta} e_x^\theta - \nu_h^\theta (\nu_h^\theta)^T (\nabla_{\Gamma_h^\theta} e_x^\theta)^T \big) \nabla_{\Gamma_h^\theta} e_v^\theta \Big) \d\theta \\
	&\quad + \int_0^1 \int_{\Gamma_h^\theta} \Big(  g(u_{\Ga_h^\theta}) \cdot e_v^\theta + \fg(u_{\Ga_h^\theta}, \nabla_{\Gamma_h^\theta}u_{\Ga_h^\theta}) \cdot  \nabla_{\Gamma_h^\theta} e_v^\theta \Big) (\nb_{\Gamma_h^\theta} \cdot e_x^\theta) \d\theta 
\end{align*}
so that
\begin{align*}
	& \ev^T \big(\bfg(\bfx,\bfu) - \bfg(\xs,\us)\big) \\
	&\ \leq c \|e_u\|_{L^2(\Gamma_h[\xs])} \ \|e_v\|_{L^2(\Gamma_h[\xs])} \\
	&\quad + c  \|\nb_{\Gamma_h[\xs]} e_v\|_{L^2(\Gamma_h[\xs])} \ 
	\Big( \|e_u\|_{L^2(\Gamma_h[\xs])} + \|\nb_{\Gamma_h[\xs]} e_u\|_{L^2(\Gamma_h[\xs])} \\
	&\quad + \|\nb_{\Gamma_h[\xs]} e_x\|_{L^2(\Gamma_h[\xs])} \  \|\nb_{\Gamma_h[\xs]}u_h^*\|_{L^\infty(\Gamma_h[\xs])} + \|\nb_{\Gamma_h[\xs]} e_x\|_{L^2(\Gamma_h[\xs])} \Big) \\
	&\quad + c \|e_v\|_{H^1(\Gamma_h[\xs])} \ \|\nb_{\Gamma_h[\xs]} e_x\|_{L^2(\Gamma_h[\xs])}\\
	&\ \leq c \normK{\ev} \big( \normK{\eu} + \normK{\ex} \big).
\end{align*}

Altogether, after an absorption, we then have
\begin{equation}
\label{eq:final estimate surface}
	\normK{\ev}^2 \leq c \normK{\ex}^2 + c \normK{\eu}^2 + c \|\dv\|_{\star,\xs}^2 .
\end{equation}

(C) {\it Combination:} 
We use the equation $\dotex = \ev$ \bbk and the bound~\eqref{eq:matrix derivatives} \ebk to show the bound
\begin{align*}
	\normK{\ex(t)}^2 &= \int_0^t \frac{\d}{\d s} \normK{\ex(s)}^2 \d s \\
	 \bbk=\ebk &\ \bbk \int_0^t  \Bigl(2\ex(s)^T \bfK(\xs) \dotex(s) + \ex(s)^T \frac{\d}{\d s}\big(\bfK(\xs)\big) \ex(s) \Bigr)\, \d s \\ \ebk
	&\leq c \int_0^t \normK{\ev(s)}^2 \d s + c \int_0^t \normK{\ex(s)}^2 \d s  ,
\end{align*}
which is substituted into the estimates \eqref{eq:final estimate surface} and \eqref{eq:final estimate PDE}, the latter again plugged into \eqref{eq:final estimate surface}. Then we take the linear combination of the three inequalities, and obtain
\begin{equation*} 
\begin{aligned}
	& \normKt{\ex(t)}^2 + \normKt{\ev(t)}^2 + \normKt{\eu(t)}^2 \\
	& \leq c \int_0^t \big( \normKs{\ex(s)}^2 + \normKs{\ev(s)}^2 + \normKs{\eu(s)}^2 \big) \d s \\
	&\quad + c\,\normKo{\eu(0)}^2 + c \|\dv(t)\|_{\star,\xs(t)}^2 + c \int_0^t \normMs{\du(s)}^2 \d s.
	\end{aligned}
\end{equation*}
Finally, by Gronwall's inequality we obtain the stability bound \eqref{eq:stability bound} for $t \in [0,t^*]$.

Now it only remains to show that $t^* = T$ for $h$ sufficiently small. 
To this end we use the assumed defect bounds to obtain the error estimates of order~$\kappa$:
\begin{align} 
\begin{aligned}
	\normK{\ev(t)} + \normK{\eu(t)} + \normK{\ex(t)}  
	\le C h^{\kappa}.
\end{aligned}
\end{align}
Then, by the inverse inequality, we have for $t\in[0,t^*]$ 
\begin{equation*} 
	\begin{aligned}
		&\|e_v(\cdot,t)\|_{W^{1,\infty}(\Gamma_h[\xs\t])}
		+\|e_u(\cdot,t)\|_{W^{1,\infty}(\Gamma_h[\xs\t])}
		+\|e_x(\cdot,t)\|_{W^{1,\infty}(\Gamma_h[\xs\t])} \\
		&\le ch^{-1}
		\big(\normK{\ev\t} +  \normK{\eu\t} +  \normK{\ex\t} \big) \\ 
		&\le 
		c C h^{\kappa-1} \leq \tfrac12 h^{(\kappa-1)/2} ,
	\end{aligned}
\end{equation*}
for sufficiently small $h$.
Hence we can extend the bounds \eqref{eq:assumed bounds - 2} beyond $t^*$, which contradicts the maximality of $t^*$ unless $t^*=T$. Therefore we have the stability bound \eqref{eq:stability bound} for $t\in[0,T]$. 
\qed
\end{proof}

\bcl

\begin{remark} \label{rem:Dziuk} 

One might be tempted, as we originally were, to use the above arguments also for proving the stability of Dziuk's method. Unfortunately, this does not work out. Let us explain what difficulties arise. 

The error equation for Dziuk's method \eqref{eq:matrix--vector form Dziuk} looks like \eqref{eq:error eq - u}--\eqref{eq:error eq - x}, with $\ex$ in place of $\eu$ and $\mathbf{f}\equiv 0$. This then leads to an estimate like \eqref{err-est-eu-1}, again with $\ex$ in place of $\eu$ and with the last but one line dropped, 
i.e., 
\begin{align*} 
	\begin{aligned}
		&\normMnum{\dotex}^2 + \half \diff \normAnum{\ex}^2 \\
		& \le 
		c \normA{\ex}^2
		+c  \normM{\dotex}  \normK{\ex}
		+c h^{(\kappa-1)/2}\normM{\dotex}^2 \\
		&\quad + c \normA{\ex} \big( \normK{\ev} + \normK{\ex} \big)   \\
		&\quad - \diff \Big(\ex^T \big( \bfA(\bfx)-\bfA(\xs) \big) \xs\Big)
		\\
		&\quad +
		\normM{\dotex} \normM{\du}.	\end{aligned}
\end{align*}
The term $\normK{\ev}$ in the inequality above cannot be bounded by the left-hand side, which contains only $\normM{\dotex}^2$ instead of $\normK{\dotex}^2$ (recall that $\ev = \dotex$ by~\eqref{eq:error eq - x}). 
By considering the enlarged system, we manage to bound the term $\normK{\ev}$ in \eqref{err-est-eu-1} by using \eqref{eq:final estimate surface} derived from testing \eqref{eq:error eq - v} with~$\ev$. But \eqref{eq:error eq - v} does not exist in the error equations for Dziuk's method \eqref{eq:matrix--vector form Dziuk}. 


On the other hand, the above arguments work well for the dynamic velocity law $\partial^\bullet v = \Delta_{\Gamma[X]} v$ instead of the velocity law $v = \Delta_{\Gamma[X]} x_{\Gamma[X]}$ of mean curvature flow; see \cite{KLLP2017} and \cite{KL2018} where the evolving finite element semi-discretization and full discretization, respectively, of such a dynamic velocity law were studied. 
The dynamic velocity law $\partial^\bullet v = \Delta_{\Gamma[X]} v$ would yield an error equation similar to \eqref{eq:error eq - v}, which can help to bound the term $\normK{\ev}$. 
However, taking the time derivative in the velocity law $v = \Delta_{\Gamma[X]} x_{\Gamma[X]}$ of mean curvature flow (or in its weak formulation \eqref{weak form Dziuk} or in its finite element discretization \eqref{semidiscrete weak form Dziuk}) does not lead to a dynamic velocity law with a coercive operator on the right-hand side, but instead the differentiation of \eqref{semidiscrete weak form Dziuk} yields
\begin{align*}
\frac{\d}{\d t}  \int_{\Ga_h}  v_h \cdot \varphi_h 
& +  \int_{\Ga_h} \Bigl(\nabla_{\Ga_h} v_h \cdot \nabla_{\Ga_h}  \varphi_h + 
  \nabla_{\Ga_h} x_h \cdot (D_{\Ga_h} v_h) \nabla_{\Ga_h} \varphi_h \Bigr) =0
\end{align*}
for all $\varphi_h \in S_h[\bfx]^3$, where the operator $D_{\Ga_h}$ is defined in Lemma~\ref{lemma:matrix differences} and where we note that $ \nabla_{\Ga_h} x_h = I- \nu_{\Ga_h}\nu_{\Ga_h}^T$. The right-hand integral is not given by a $H^1$-elliptic bilinear form, in contrast to its first term: 
a calculation shows that for $ \varphi_h=v_h$ the last term of the integrand is everywhere non-positive, and
the complete integrand in the second integral becomes negative in points where $\nabla_{\Ga_h} v_h$  has the normal vector $\nu_{\Ga_h}$ in its null-space and has two non-zero eigenvalues of different sign (one eigenvalue is always zero, since the columns of  $\nabla_{\Ga_h} v_h$ are orthogonal to $\nu_{\Ga_h}$). 

So we must concede that our techniques do not yield a stability proof for Dziuk's method. It remains an open problem whether
a stability result similar to Proposition~\ref{proposition:stability - coupled problem} exists at all for Dziuk's method.
\end{remark}
\ecl

\section{Defect bounds for the semi-discretization}
\label{section:Defect}

We now turn to estimating the consistency errors defined by \eqref{defect vectors}.
\begin{lemma}
\label{lemma: semidiscrete residual}
    Let the  surface $X$ evolving under mean curvature flow be sufficiently regular on the time interval $[0,T]$. Then, there exist $h_0>0$ and $c>0$ such that for all $h\leq h_0$ and for all $t\in[0,T]$, the  defects $d_v(t)\in S_h(\Gamma_h[\bfx^*])^3$ and $d_u(t)\in S_h(\Gamma_h[\bfx^*])^4$ of the $k$th-degree finite elements,  as defined by their nodal vectors $\dv(t)$ and $\du(t)$ in \eqref{defect vectors},
    are bounded as
    \begin{align*}
        \|\dv(t)\|_{\star,\xs(t)}=&\ \|d_v\|_{H_h\inv(\Ga_h[\xs(t)])} \leq c h^k , \\
        \|\du(t)\|_{\bfM(\xs(t))}=&\ \|d_u\|_{L^2(\Ga_h[\xs(t)])} \leq c h^k .
    \end{align*}
 The constant $c$ is independent of $h$ and $t\in[0,T]$.
\end{lemma}

\begin{proof} The proof is very similar to the proof of the defect bounds in \cite{KLLP2017}. The main difference is that we insert the Ritz projection $u_h^*$ instead of the nodal interpolation into the scheme. This has the benefit that a critical term in estimating the $L^2$ norm cancels, as we will see below. On the other hand, the approximation quality of the Ritz projection is the same as that of interpolation: as is shown in \bbk \cite[Theorem~6.3]{Kovacs2017}, \ebk we have
\begin{equation}\label{ritz-bound}
\begin{aligned}
&\| (u_h^*)^l (t) - u(t) \|_{H^1(\Gamma[X(\cdot,t)])} \| \le C h^k, \\
&\| (\partial^\bullet_h u_h^*)^l (t) - \partial^\bullet u(t) \|_{H^1(\Gamma[X(\cdot,t)])} \| \le C h^k.
\end{aligned}
\end{equation}
We begin by rewriting the equations \eqref{defect vectors} defining the defects in weak form. With the functions of $u=(\n,H)$ defined by $g(u)=-H\n$ and $f(u,\nabla_\Gamma u)=|A|^2 u$ for $A=\nabla_\Gamma \nu$ we have
\begin{align*}
&\int_{\Gamma_h[\xs]} \nabla_{\Gamma_h[\xs]} v_h^* \cdot \nabla_{\Gamma_h[\xs]} \psi_h +
\int_{\Gamma_h[\xs]} v_h^* \cdot  \psi_h 
\\
& = \int_{\Gamma_h[\xs]} \nabla_{\Gamma_h[\xs]} \bigl(g(u_h^*)\bigr) \cdot \nabla_{\Gamma_h[\xs]} \psi_h +
\int_{\Gamma_h[\xs]} g(u_h^*) \cdot  \psi_h 
\\
& \quad 
+ \int_{\Gamma_h[\xs]} d_v \cdot  \psi_h 
\end{align*}
for all $\psi_h\in S_h[\xs]^3$, 
and
\begin{align*}
&\int_{\Gamma_h[\xs]} \!\partial^\bullet_h u_h^* \cdot \varphi_h
+
\int_{\Gamma_h[\xs]} \!\nabla_{\Gamma_h[\xs]} u_h^* \cdot \nabla_{\Gamma_h[\xs]} \varphi_h
= 
\int_{\Gamma_h[\xs]} \! f(u_h^*, \nabla_{\Gamma_h[\xs]}u_h^*) \cdot  \varphi_h
\\
& \quad + \int_{\Gamma_h[\xs]} d_u \cdot  \varphi_h 
\end{align*}
for all $\varphi_h \in S_h[\xs]^4$. Subtracting the weak formulation \eqref{weak form} of the equations for the exact solution, we thus have
\begin{align*}
 &\int_{\Gamma_h[\xs]} d_v \cdot  \psi_h = \biggl(\int_{\Gamma_h[\xs]} \nabla_{\Gamma_h[\xs]} v_h^* \cdot \nabla_{\Gamma_h[\xs]} \psi_h
 - \int_{\Gamma[X]} \nabla_{\Gamma[X]} v \cdot  \nabla_{\Gamma[X]} \psi_h^l \biggr)
 \\
 & \quad + \biggl( \int_{\Gamma_h[\xs]} v_h^* \cdot  \psi_h - \int_{\Gamma[X]}  v \cdot   \psi_h^l \biggr)
 \\
 & \quad - \biggl( \int_{\Gamma_h[\xs]} \nabla_{\Gamma_h[\xs]} \bigl(g(u_h^*)\bigr) \cdot \nabla_{\Gamma_h[\xs]} \psi_h 
 -  \int_{\Gamma[X]} \nabla_{\Gamma[X]} \bigl(g(u)\bigr) \cdot  \nabla_{\Gamma[X]} \psi_h^l \biggr)
 \\
 &\quad -  \biggl( \int_{\Gamma_h[\xs]} g(u_h^*) \cdot  \psi_h  - \int_{\Gamma[X]}  g(u) \cdot   \psi_h^l \biggr)
\end{align*}
for all $\psi_h\in S_h[\xs]^3$, 
and 
\begin{align*}
& \int_{\Gamma_h[\xs]} d_u \cdot  \varphi_h  = \biggl( \int_{\Gamma_h[\xs]} \!\partial^\bullet_h u_h^* \cdot \varphi_h -
 \int_{\Gamma[X]}  \partial^\bullet u \cdot   \varphi_h^l \biggr)
\\
& \quad + \biggl[ \int_{\Gamma_h[\xs]} \!\nabla_{\Gamma_h[\xs]} u_h^* \cdot \nabla_{\Gamma_h[\xs]} \varphi_h 
- \int_{\Gamma[X]} \nabla_{\Gamma[X]} u \cdot  \nabla_{\Gamma[X]} \varphi_h^l \biggr]
\\ 
&- \biggl( \int_{\Gamma_h[\xs]} \! f(u_h^*, \nabla_{\Gamma_h[\xs]}u_h^*) \cdot  \varphi_h - 
 \int_{\Gamma[X]}  f(u,\nabla_{\Gamma[X]}u )\cdot   \varphi_h^l \biggr)
\end{align*}
for all $\varphi_h \in S_h[\xs]^4$. Here we note that the critical term in big square brackets equals
\bcl
$$
-\biggl[ \int_{\Gamma_h[\xs]} \! u_h^* \cdot  \varphi_h 
- \int_{\Gamma[X]} u \cdot  \varphi_h^l \biggr]
$$
\ecl
thanks to the definition of $u_h^*$ in \eqref{uhs-ritz}. \bbk The remaining \ebk terms do not contain the gradient of the test function. This allows us obtain estimates of $d_u$ in the $L^2$ norm.

From here on, all the terms in brackets can be bounded by the same arguments as for the corresponding terms in \cite{KLLP2017}, using the bound \eqref{ritz-bound} instead of the bound for the interpolation error where required. These arguments are based on geometric estimates that were previously proved in \cite{Demlow2009,Dziuk88,DziukElliott_ESFEM,DziukElliott_L2,Kovacs2017}. We refer to \cite{KLLP2017} for the details.
\qed
\end{proof}

\section{Proof of Theorem~\ref{theorem: coupled error estimate}}
\label{section: proof completed}
The errors are decomposed using interpolations for $X$ and $v$ and the Ritz map \eqref{uhs-ritz} for $u$ and using the definition of the composed lift $L$ from Section~\ref{subsec:lifts}: 
\begin{align*}
  X_h^L  - X  =&\ \big( \widehat X_h  - X_h^*  \big)^{l} +  \big((X_h^*)^l  - X \big), \\
    v_h^{L}  - v  =&\ \big(\widehat v_h  - v_h^*)^{l} + \big((v_h^*)^l  - v \big) , \\
 u_h^{L}  - u  =&\ \big(\widehat u_h  - u_h^* \big)^{l} + \big((u_h^*)^l  - u \big) .
\end{align*}

The last terms  in these formulas can be bounded in the $H^1(\Gamma)$ norm by $Ch^k$, using  the interpolation error bounds from \cite[Proposition 2.7]{Demlow2009} for $X$ and~$v$, and the Ritz-map error bound \eqref{ritz-bound}  for $u$.

To bound the first terms on the right-hand sides, we use  the stability estimate of Proposition~\ref{proposition:stability - coupled problem} together with the defect bounds of Lemma~\ref{lemma: semidiscrete residual} to obtain
$$
\normK{\ex} + \normK{\ev} + \normK{\eu} \le Ch ^k.
$$
By the equivalence of norms shown in \cite[Lemma~3]{Dziuk88} and by \eqref{K-H1} we have (omitting the argument $t$)
$$
    \| \big(\widehat u_h - u_h^*\big)^{l} \|_{H^1(\Ga[X])} \le
    c \| \widehat u_h - u_h^* \|_{H^1(\Ga_h[\xs])}
    = c \normK{\eu},
$$
and similarly for $\widehat v_h - v_h^*$ and $\widehat X_h - X_h^*$. This proves the result.
%

\section{Stability of the  full discretization}
\label{section:stability-full}

\subsection{Auxiliary results by Dahlquist and Nevanlinna \& Odeh}

We recall two important results that enable us to use energy estimates for BDF methods up to order 5: the first result is from Dahlquist's $G$-stability theory, and the second one from the multiplier technique of Nevanlinna and Odeh.

\begin{lemma}[Dahlquist \cite{Dahlquist}]
\label{lemma: Dahlquist}
    Let $\delta(\zeta) = \sum_{j=0}^q \delta_j \zeta^j$ and $\mu(\zeta) = \sum_{j=0}^q \mu_j \zeta^j$ be polynomials of degree at most $q$ (at least one of them of degree $q$) that have no common divisor. Let $\la \, \cdot , \cdot \, \ra$ denote an inner product on $\R^N$. 
    If
    \begin{equation*}
        \textnormal{Re} \frac{\delta(\zeta)}{\mu(\zeta)} > 0 \qquad \textrm{for} \quad |\zeta|<1,
    \end{equation*}
    then there exists a symmetric positive definite matrix $G = (g_{ij}) \in \R^{q\times q}$ 
    such that for all $\bfw_0,\dotsc,\bfw_q\in\R^N$
    \begin{equation*}
        \Big\la \sum_{i=0}^q \delta_i \bfw_{q-i} , \sum_{i=0}^q \mu_i \bfw_{q-i}  \Big\ra \ge \sum_{i,j=1}^q g_{ij} \la \bfw_i , \bfw_j \ra - \sum_{i,j=1}^q g_{ij} \la \bfw_{i-1} , \bfw_{j-1} \ra .
      \end{equation*}
\end{lemma}

In view of the following result, the choice $\mu(\zeta)=1-\eta\zeta$ together with the polynomial $\delta(\zeta)$ of the BDF methods will play an important role later on.
\begin{lemma}[Nevanlinna \& Odeh \cite{NevanlinnaOdeh}]
\label{lemma: NevanlinnaOdeh multiplier}
    If $q\leq5$, then there exists $0\leq\eta<1$ \st\ for $\delta(\zeta)=\sum_{\ell=1}^q \frac{1}{\ell}(1-\zeta)^\ell$,
    \begin{equation*}
        \textnormal{Re} \,\frac{\delta(\zeta)}{1-\eta\zeta} > 0 \qquad \textrm{for} \quad |\zeta|<1.
    \end{equation*}
    The smallest possible values of $\eta$ are found to be 
    $\eta= 0,  0,  0.0836, 0.2878,  0.8160$ 
    for $q=1,\dotsc,5$, respectively.
\end{lemma}

 These results have previously been applied in the error analysis of BDF methods for various parabolic problems in
\cite{AkrivisLiLubich_quasilinBDF,AkrivisLubich_quasilinBDF,KL2018,KovacsPower_quasilinear,LubichMansourVenkataraman_bdsurf}, where they were used when testing the error equation with the error. In contrast to these papers, here these results are used for testing the error equation with the discretized time derivative of the error.

\subsection{Defects and errors}

We choose nodal vectors $\xs(t)\in\R^{3N}$, $\vs(t)\in\R^{3N}$ and $\us(t)\in\R^{4N}$ determined by the exact solution $X$, $v$ and $u=(\n,H)$ as
in Section~\ref{subsec:stability - error eqs}, and we abbreviate $\bfx_*^n=\xs(t_n)$, $\bfv_*^n=\vs(t_n)$, and $\bfu_*^n=\us(t_n)$. When we insert these reference values into the numerical scheme, we obtain defects $\dv^n$, $\du^n$, $\dx^n$ that will be studied in the next section: for $n\ge q$,
\begin{equation}
\label{BDF-defects}
	\begin{aligned}
		\bfK(\widetilde \bfx_*^n) \bfv_*^n &= \bfg(\widetilde \bfx_*^n,\widetilde \bfu_*^n) + \redon \bfM(\wtxls^n) \redoff \dv^n, \\
		\bfM(\widetilde \bfx_*^n)\dot \bfu_*^{n} + \bfA(\widetilde \bfx_*^n) \bfu_*^n &= \bff(\widetilde \bfx_*^n,\widetilde \bfu_*^n) + \redon \bfM(\wtxls^n) \redoff \du^n,  \\
		\dot \bfx_*^{n} &= \bfv_*^n + \dx^n,
	\end{aligned} 
\end{equation}
where
\begin{align*}
&\widetilde \bfx_*^n = \sum_{j=0}^{q-1} \gamma_j \bfx_*^{n-1-j}, \qquad 
\widetilde \bfu_*^n = \sum_{j=0}^{q-1} \gamma_j \bfu_*^{n-1-j}, 
\\
&\dot \bfx_*^{n} = \frac{1}{\tau} \sum_{j=0}^q \delta_j \bfx_*^{n-j}, \qquad\
\dot \bfu_*^{n} = \frac{1}{\tau} \sum_{j=0}^q \delta_j \bfu_*^{n-j}.
\end{align*}
The errors of the numerical solution values $\bfx^n$, $\bfv^n$ and $\bfu^n$ are denoted by
\begin{align*}
	 \ex^n = \bfx^n - \xls^n , \qquad \ev^n = \bfv^n - \vls^n, \qquad \eu^n = \bfu^n - \uls^n,
\end{align*}
and we abbreviate
\begin{equation}\label{dotexun}
\dotex^n = \frac{1}{\tau} \sum_{j=0}^q \delta_j \ex^{n-j}, \qquad 
\doteu^n = \frac{1}{\tau} \sum_{j=0}^q \delta_j \eu^{n-j} .
\end{equation}
Subtracting \eqref{BDF-defects} from \eqref{BDF},  we obtain the following error equations:
\bcl
\begin{subequations} 
\label{eq:error equations - full}
\begin{align}
\label{eq:error eq - v - full}
	\bfK(\widetilde \bfx^n)  \ev^n &= \bfr_\bfv^n ,
	\\
\label{eq:error eq - u - full}
	 \bfM(\widetilde \bfx^n)\doteu^n+ \bfA(\widetilde \bfx^n)\eu^n &= \bfr_\bfu^n ,
	 \\
\label{eq:error eq - x - full}
    \dotex^n &= \ev^n - \dx^n , 
\end{align}
\end{subequations}
where
\begin{subequations} 
\label{eq:error terms - full}
\begin{align}
\label{eq:error term - v - full}
        \bfr_\bfv^n
	=&\ -  \big( \bfK(\widetilde \bfx^n)-\bfK(\widetilde \bfx_*^n) \big) \vls^n 
	\\
	\nonumber 
	&\ + \big(\bfg(\widetilde \bfx^n,\widetilde \bfu^n) - \bfg(\widetilde \bfx_*^n,\widetilde \bfu_*^n)\big) - \redon \bfM(\wtxls^n) \redoff \dv^n , \\
\label{eq:error term - u - full}
	 \bfr_\bfu^n
	=&\ -  \big( \bfM(\widetilde \bfx^n)-\bfM(\widetilde \bfx_*^n) \big) \dot\bfu_*^n  
	 \\
	\nonumber
	&\ - \big( \bfA(\widetilde \bfx^n)-\bfA(\widetilde \bfx_*^n) \big) \bfu_*^n 
	 \\
	\nonumber &\ + \big(\bff(\widetilde \bfx^n,\widetilde \bfu^n) - \bff(\widetilde \bfx_*^n,\widetilde \bfu_*^n)\big) - \redon \bfM(\wtxls^n) \redoff \du^n .
\end{align}
\end{subequations}

%
%
%

\subsection{Stability estimate}

The following stability result is the analogue of  Proposition~\ref{proposition:stability - coupled problem} for the full discretization.

\begin{proposition}
\label{proposition:stability - coupled problem - full} Consider the linearly implicit BDF time discretization of order $q$ with $2\le q \le 5$.
 Assume that, for step sizes restricted by $\tau\le C_0 h$, there exists $\kappa$ with $1 < \kappa \le k$ such that the defects are bounded by 
 		\begin{equation}
	\label{eq:assume small defects}
		\|\dx^n\|_{\bfK(\xls^n)} \leq c h^\kappa , \quad
		\|\dv^n\|_{\star,\xls^n} \leq c h^\kappa , \quad 
		\|\du^n\|_{\bfM(\xls^n)} \leq c h^\kappa ,
	\end{equation} 
	 for $q\tau \leq n\tau \leq T$, and that also the errors of the starting values are bounded by
		\begin{equation}
	\label{eq:assume small initial values}
		\|\ex^i\|_{\bfK(\xls^n)} \leq c h^\kappa , \quad
		\|\ev^i\|_{\star,\xls^n} \leq c h^\kappa , \quad 
		\|\eu^i\|_{\bfM(\xls^n)} \leq c h^\kappa ,
	\end{equation} 
for $i=0,\dots,q-1$.
Then, there exist  $h_0>0$ and $\tau_0>0$ such that the following stability estimate holds for all $h\leq h_0$, $\tau\le\tau_0$, and $n$ with $n\tau \le T$, satisfying $\tau \leq C_0 h$ (where $C_0>0$ can be chosen arbitrarily),
	\begin{equation}
	\label{eq:stability bound - full}
		\begin{aligned}
			& \| \ex^n \|_{\bfK(\xls^n)}^2 +  \| \ev^n \|_{\bfK(\xls^n)}^2 +  \| \eu^n \|_{\bfK(\xls^n)}^2 \\
			& \leq \ C  \sum_{i=0}^{q-1}\Bigl( \| \ex^i \|_{\bfK(\xls^i)}^2 + \| \ev^i \|_{\bfK(\xls^i)}^2 + \| \eu^i \|_{\bfK(\xls^i)}^2\bigr) \\
			& \quad +C \max_{0\le j \le n}\|\dv^j\|_{\star,\xls^j}^2  
			+ C \tau\sum_{j=q}^{n} \|\du^j\|_{\bfM(\xls^j)}^2 + C  \tau\sum_{j=q}^{n} \|\dx^j\|_{\bfK(\xls^j)}^2,
		\end{aligned}
	\end{equation}
	where $C$ is independent of $h$, $\tau$ and $n$ with $n\tau\le T$, but depends on the final time $T$.
\end{proposition}

In Section~\ref{section:defect-full} we will show that under sufficient smoothness assumptions on the solution, the defects satisfy the bounds 
$$
\|\dx^n\|_{\bfK(\xls^n)} \le C\tau^q,\quad\ 
\|\dv^n\|_{\star,\xls^n}  \le C(h^k+\tau^q), \quad\ 
\|\du^n\|_{\bfK(\xls^n)} \le C(h^k+\tau^q).
$$
Hence, condition \eqref{eq:assume small defects} is satisfied under the step-size restriction $\tau\le C_0 h$ if $q\ge 2$.
We note that the error functions $e_x^n,e_v^n \in S_h[\xls^n]^3$  and $e_u^n \in S_h[\xls^n]^4$ with nodal vectors $\ex^{n}$, $\ev^{n}$ and $\eu^n$, respectively, are then bounded by
\begin{equation*}
  \begin{aligned}
	&\| e_x^n \|_{H^1(\Ga_h(\xls^n))} \leq C(h^k+\tau^q), \\ 
	&\| e_v^n \|_{H^1(\Ga_h(\xls^n))} \leq C(h^k+\tau^q) , \\
	&\| e_u^n \|_{H^1(\Ga_h(\xls^n))} \leq C(h^k+\tau^q) ,
  \end{aligned}
	\qquad \textnormal{ for }  n \tau \leq T,
\end{equation*}
provided the starting values are sufficiently accurate.

\begin{proof} The organisation of the proof, including the numbering of its parts, will be the same as that of Proposition~\ref{proposition:stability - coupled problem}.

Let $t^*\in(0,T]$ (which {\it a priori} might depend on  $\tau$ and $h$) be the maximal time such that the following inequalities hold: 
\begin{equation}
\label{eq:assumed bounds - n}
    \begin{aligned}
    	\| e_x^n\|_{W^{1,\infty}(\Ga_h[\xls^n])} &\leq  h^{(\kappa-1)/2} , \\
    	\| e_v^n\|_{W^{1,\infty}(\Ga_h[\xls^n])} &\leq  h^{(\kappa-1)/2} , \\
        \|e_u^n \|_{W^{1,\infty}(\Ga_h[\xls^n])} &\leq  h^{(\kappa-1)/2} , 
    \end{aligned} \qquad \textrm{ for } \quad n\tau \le t^*.
\end{equation}
In the following we assume that $n\tau \le t^*$, so that the above bounds are valid.
At the end of the proof we will show that in fact $t^*$ coincides with $T$.

\medskip
(A) \emph{Estimates for the surface PDE:} In the time-continuous case, we tested the error equation for $\eu$ with the time derivative $\doteu$. Now, in the time-discrete case, we form the difference of equation \eqref{eq:error eq - u - full} for $n$ with $\eta$ times this equation for $n-1$, for $\eta\in[0,1)$ of 
Lemma~\ref{lemma: NevanlinnaOdeh multiplier}, and then we test this difference
with the discrete time derivative $\doteu^n$ defined by \eqref{dotexun}. This yields, for $n\ge q+1$,
\begin{equation} 
\label{eq:error eq tested - u - full}
	\begin{aligned}
	 &(\doteu^n)^T\bfM(\wtx^n)\doteu^n - \eta (\doteu^n)^T\bfM(\wtx^{n-1})\doteu^{n-1} + 
	 (\doteu^n)^T\bfA(\wtx^n)(\eu^n - \eta \eu^{n-1}) 
	 \\
	 & = 
	 -  \eta (\doteu^n)^T\big( \bfA(\wtx^n) -  \bfA(\wtx^{n-1}) \big)  \eu^{n-1} 
	 +  (\doteu^n)^T (\ru^n-\eta\ru^{n-1}).
	\end{aligned}
\end{equation}

To be able to deal with the second term in the first line and the first term in the second line (and further terms later in the proof), we need to estimate $\wtx^n-\wtx^{n-1}$. Let
$$
\widetilde \bfV^n = \be\, \wtx^n = \frac1\tau( \wtx^n-\wtx^{n-1} )
$$
and denote by $\widetilde V_h^n \in S_h[\wtx^n]$ the finite element function on $  \Gamma_h[\wtx^n]$ with nodal vector $\widetilde \bfV^n$.
We will show that under condition \eqref{eq:assumed bounds - n},
\begin{equation} \label{bar-vhn}
\| \widetilde V_h^n \|_{W^{1,\infty}(\Gamma_h[\wtx^n])} \le K
\end{equation}
with a constant $K$ independent of $h$ and $\tau$ and $n$ with $n\tau\le t^*$. 
This condition will play the same crucial role as \eqref{vh-bound} in the time-continuous case.
We proceed by induction and assume that this bound holds up to $n-1$. In the same way as \eqref{vh-bound} led to the norm equivalence~\eqref{norm-equiv-t}, this induction hypothesis implies that for $m=n-1$,
\begin{align}
\nonumber
&\text{the norms $\|\cdot\|_{\bfM(\wtx^j)}$ are $h$- and $\tau$-uniformly equivalent for $q\le j \le m$,}
\\
&\text{and so are the norms $\|\cdot\|_{\bfA(\wtx^j)}$.}
\label{norm-equiv-n}
\end{align}
Next we note that the finite element function $\widetilde v_h^n$ on
$\Gamma_h[\wtx^{n-1}]$ with nodal vector
$$
\widetilde \bfv^n = \sum_{j=0}^{q-1} \gamma_j \bfv^{n-j-1} = \frac1\tau \sum_{j=0}^q \delta_j \wtx^{n-j}
$$
satisfies, by the norm equivalence \eqref{norm-equiv-n}, by the bounds~\eqref{eq:assumed bounds - n}, and by the fact that finite element interpolation is a bounded operation on $W^{1,\infty}$,
\begin{align*}
&\| \widetilde v_h^n \|_{W^{1,\infty}(\Gamma_h[\wtx^{n-1}])} 
\le \sum_{j=0}^{q-1} |\gamma_j| \,\|v_h^{n-j-1}\| _{W^{1,\infty}(\Gamma_h[\wtx^{n-1}])} 
\\
&\le \sum_{j=0}^{q-1} |\gamma_j| \,\|v_h^{n-j-1}\| _{W^{1,\infty}(\Gamma_h[\wtx^{n-j-1}])} 
\le  \sum_{j=0}^{q-1} |\gamma_j|\, \|v_h^{n-j-1}\| _{W^{1,\infty}(\Gamma_h[\xls^{n-j-1}])} 
\\
&\le  \sum_{j=0}^{q-1} |\gamma_j| \Bigl( \|v_{h,*}^{n-j-1}\| _{W^{1,\infty}(\Gamma_h[\xls^{n-j-1}])} 
+ \|e_v^{n-j-1}\| _{W^{1,\infty}(\Gamma_h[\xls^{n-j-1}])} \Bigr)
\\
&\le  \sum_{j=0}^{q-1} |\gamma_j| \Bigl( c\|v(\cdot,t_{n-j-1})\| _{W^{1,\infty}(\Gamma_h[\xls^{n-j-1}])} 
+ \|e_v^{n-j-1}\| _{W^{1,\infty}(\Gamma_h[\xls^{n-j-1}])} \Bigr) \le C.
\end{align*}
We factorize the generating polynomial $\delta(\zeta)=\sum_{j=0}^q \delta_j\zeta^j = \sum_{\ell=1}^q \frac1\ell(1-\zeta)^\ell$ as
\begin{equation}\label{delta-factor}
\delta(\zeta)=(1-\zeta)\sigma(\zeta).
\end{equation}
The polynomial \blueon $\sigma(\zeta)=\sum_{j=0}^{q-1}\sigma_j \zeta^j$ \blueoff has degree $q-1$ and has no zeros in the closed unit disk (this property expresses the zero-stability of the BDF method, which holds true for $q\le 6$; see \cite[Section~III.3]{HNW}). Therefore, there exists $\varrho<1$ such that
\begin{equation}
\label{eq:inverse of sigma}
	\frac{1}{\sigma(\zeta)} =: \chi(\zeta) = \sum_{j=0}^\infty \chi_j \zeta^j \quad\ \textnormal{ with } |\chi_j| \leq c\varrho^j.
\end{equation}
We rewrite, for $n\ge q$,
$$
\widetilde \bfv^n = \sum_{j=0}^{q-1}  \sigma_j \widetilde \bfV^{n-j} = \sum_{j=q}^n \sigma_{n-j} \widetilde \bfV^{j} + \bfs^n,
$$
with $\sigma_j=0$ for $j\ge q$ and $\bfs^n = \sum_{i=1}^{q-1} \sigma_{n-i} \be \widetilde \bfx^i$, for which we note that $\bfs^n=0$ for $n\ge 2q-2$.
We then obtain 
\begin{equation}\label{inverse convolution}
\widetilde \bfV^n = \sum_{j=q}^n \chi_{n-j} \bigl( \widetilde \bfv^j - \bfs^j),
\end{equation}
because we have (using the associativity of convolution) 
\begin{align*}
	  \sum_{j=q}^n \chi_{n-j}  (\widetilde\bfv^j-\bfs^j) & 
	=   \sum_{j=q}^n \chi_{n-j} \sum_{k=q}^{j} \sigma_{j-k}  \widetilde\bfV^{k} 
        = \sum_{k=q}^{n} \Bigl( \sum_{m=0}^{n-k} \chi_{n-k-m}\sigma_m \Bigr)  \widetilde\bfV^{k} =  \widetilde\bfV^n.
\end{align*}
With the geometric decay of the coefficients $ \chi_{j}$ and the bound for $ \widetilde \bfv^j$ and the assumed accuracy of the starting values, we conclude from \eqref{inverse convolution} to the bound
$$
\| \widetilde V_h^n \|_{W^{1,\infty}(\Gamma_h[\wtx^{n-1}])} \le C
$$
by the same arguments as used for bounding $\widetilde v_h^n$ above. Here, $C$ depends only on a $W^{1,\infty}$ bound of the exact solution on the interval $[0,T]$. Since $\wtx^n = \wtx^{n-1} + \tau \widetilde \bfV^n$, using this bound for $\widetilde V_h^n$ in Lemma~\ref{lemma:theta-independence} yields that the norm equivalence \eqref{norm-equiv-n} extends up to $m=n$ (and not just $n-1$), and so we obtain \eqref{bar-vhn} with $K=2C$.


\medskip
After these preparations, the terms in  \eqref{eq:error eq tested - u - full} are now estimated separately.

(i) On the left-hand side of \eqref{eq:error eq tested - u - full}, the first term is 
$$
(\doteu^n)^T\bfM(\wtx^n)\doteu^n =  \|\doteu^n\|_{\bfM(\wtx^{n})}^2 .
$$
%
The second term is bounded by
\begin{align*}
	&\ (\doteu^n)^T\bfM(\wtx^{n-1})\doteu^{n-1} \\
	&\leq  \|\doteu^n\|_{\bfM(\wtx^{n-1})} \|\doteu^{n-1}\|_{\bfM(\wtx^{n-1})} \\
	&\leq  \tfrac12  \|\doteu^n\|_{\bfM(\wtx^{n-1})}^2 +  \tfrac12  \|\doteu^{n-1}\|_{\bfM(\wtx^{n-1})}^2 \\
	&\leq  \tfrac12 (1+c\tau) \|\doteu^n\|_{\bfM(\wtx^{n})}^2 +  \tfrac12  \|\doteu^{n-1}\|_{\bfM(\wtx^{n-1})}^2,
\end{align*} 
where we used the bound \eqref{bar-vhn} in \eqref{eq:matrix difference bounds} to raise the superscript from $n-1$ to $n$ in the first term of the last line.
This yields
\begin{equation}
\label{M-lower bound}
	\begin{aligned}
		& (\doteu^n)^T\bfM(\wtx^n)\doteu^n - \eta (\doteu^n)^T\bfM(\wtx^{n-1})\doteu^{n-1} 
		\\
		& \qquad\geq  \big( 1 - \tfrac12\eta (1+ c\tau) \big)  \|\doteu^n\|_{\bfM(\wtx^{n})}^2 - \tfrac12\eta \|\doteu^{n-1}\|_{\bfM(\wtx^{n-1})}^2.
	\end{aligned} 
\end{equation}

(ii) We now bound the critical third term on the left-hand side of \eqref{eq:error eq tested - u - full} from below. Thanks to Lemmas~\ref{lemma: Dahlquist} and~\ref{lemma: NevanlinnaOdeh multiplier} we have
\begin{equation}
\label{DNO}
	\begin{aligned}
		& (\doteu^{n})^T \bfA(\wtx^n)(\eu^{n} - \eta \eu^{n-1}) = \Bigl(\frac{1}{\tau} \sum_{i=0}^q \delta_i \eu^{n-i}\Bigr)^T\bfA(\wtx^n)(\eu^{n} - \eta \eu^{n-1})
		\\
		&\geq \frac{1}{\tau} \!\! \sum_{i,j=1}^q g_{ij} (\eu^{n-q+i})^T \bfA(\wtx^n) \eu^{n-q+j} - 
		\frac{1}{\tau} \!\! \sum_{i,j=1}^q g_{ij} (\eu^{n-q+i-1})^T \bfA(\wtx^n) \eu^{n-q+j-1}.
	\end{aligned}
\end{equation}

The first term on the right-hand side of \eqref{eq:error eq tested - u - full} is bounded as follows: recalling that
$\wtx^n-\wtx^{n-1}=\tau\widetilde \bfV^n$ with the bound \eqref{bar-vhn} and using the bound \eqref{eq:matrix difference bounds} and the Young inequality yields
\begin{equation*} 
	\begin{aligned}
		-(\doteu^n)^T\bigl( \bfA(\wtx^n) -  \bfA(\wtx^{n-1}) \bigr)  \eu^{n-1} & \leq c\tau  \|\doteu^n\|_{\bfA(\wtx^n)}\, \|\eu^{n-1}\|_{\bfA(\wtx^{n})}
		\\
		& \leq	 
		\tfrac12 c \rho \tau^2 \|\doteu^n\|_{\bfA(\wtx^n)}^2 + \tfrac12 c\rho^{-1} \|\eu^{n-1}\|_{\bfA(\wtx^n)}^2,
	\end{aligned}
\end{equation*}
where $\rho>0$ will be chosen small, but independent of $h$ and $\tau$. By an inverse inequality we  have
$$
	\|\doteu^n\|_{\bfA(\wtx^n)} \le \frac ch \|\doteu^n\|_{\bfM(\wtx^n)},
$$
and under the step size restriction $\tau\le C_0h$,
we thus obtain
\begin{equation}
\label{A-diff}
	- (\doteu^n)^T\bigl( \bfA(\wtx^n) -  \bfA(\wtx^{n-1}) \bigr)  \eu^{n-1} \le
	c \rho \|\doteu^n\|_{\bfM(\wtx^n)}^2 +  c \rho^{-1}\|\eu^{n-1}\|_{\bfA(\wtx^{n-1})}^2,
\end{equation}
where we  used again the bound \eqref{bar-vhn} in \eqref{eq:matrix difference bounds} to reduce the superscript from $n$ to $n-1$ in the last term.

\medskip
The last expression on the right-hand side of \eqref{eq:error eq tested - u - full} is
$(\doteu^n)^T (\ru^n-\eta\ru^{n-1})$, where $\ru^n$ is defined in~\eqref{eq:error term - u - full}.
The terms that  contain the stiffness matrix $\bfA$ will require new arguments, whereas the other terms are bounded similarly to the corresponding terms in the time-continuous case.

(iii) Using \eqref{eq:assumed bounds - n} in justifying the bound~\eqref{eq:matrix difference bounds} and using the Young inequality,
the terms involving the differences of mass matrices are bounded as,  using the notation $\wtex^n = \widetilde \bfx^n - \wtxls^n = \sum_{i=0}^{q-1} \gamma_i \ex^{n-1-i}$,  
\begin{equation}
\label{eq:rhs estimate - dot e vs solution}
	\begin{aligned}
		& - (\doteu^n)^T \big( \bfM(\wtx^n)-\bfM(\wtxls^n) \big) \bfu_*^n 
		+ \eta (\doteu^n)^T \big( \bfM(\wtx^{n-1})-\bfM(\wtxls^{n-1}) \big) \bfu_*^{n-1} \\
		&\leq  c \| \doteu^n\|_{\bfM(\wtx^n)} \|\wtex^n\|_{\bfK(\wtx^n)}
		+ c \| \doteu^n\|_{\bfM(\wtx^{n-1})} \|\wtex^{n-1}\|_{\bfK(\wtx^{n-1})} \\
		&\leq  \rho \| \doteu^n\|_{\bfM(\wtx^n)}^2  + c  \rho^{-1}   \|\wtex^n\|_{\bfK(\wtx^n)}^2
		+ c  \rho^{-1} \|\wtex^{n-1}\|_{\bfK(\wtx^{n-1})}^2 
	\end{aligned}
\end{equation}
with a small $\rho > 0$, which will be chosen later independently of $h,\tau$ and $n$ with $n\tau\le t^*$.
%

(iv) We next bound the two terms in $(\doteu^n)^T (\ru^n-\eta\ru^{n-1})$ that involve the stiffness matrix $\bfA$:
\begin{align*}
	R_{\bfA}^n &= R_{\bfA}^{1,n} + R_{\bfA}^{2,n} \\
	&= - (\doteu^n)^T \big( \bfA(\wtx^n)-\bfA(\wtxls^n) \big) \bfu_*^n + \eta (\doteu^n)^T \big( \bfA(\wtx^{n-1})-\bfA(\wtxls^{n-1}) \big) \bfu_*^{n-1}.
\end{align*}
These terms will be bounded using a discrete analogue of the arguments in part (iv) of the proof of Proposition~\ref{proposition:stability - coupled problem}. We give the detailed argument for the first term, the second term is then bounded in the same way.

Part (iv) of the proof of Proposition~\ref{proposition:stability - coupled problem} starts from the product rule of differentiation, which is not directly available in the time-discrete case. So we need to come up with a substitute. From the definition \eqref{dotexun} of $\doteu^n$ and the factorization \eqref{delta-factor}  
we obtain the relation
\begin{equation}
\label{eq:connecting discrete derivatives}
	\doteu^n = \frac1\tau\sum_{j=0}^{q} \delta_j \, \eu^{n-j}  = 	\sum_{j=0}^{q-1} \sigma_j \, \frac1\tau\bigl( \eu^{n-j} - \eu^{n-j-1} \bigr) =
	\sum_{j=0}^{q-1} \sigma_j \, \be \eu^{n-j}	
	 .
\end{equation}
%
%
We rewrite the term $R_{\bfA}^{1,n}$ using  \eqref{eq:connecting discrete derivatives}  and a discrete product rule formula, and obtain the following discrete substitute for \eqref{RA1}:
\begin{equation}
\label{eq:discrete product rule formula}
	\begin{aligned}
		R_{\bfA}^{1,n}
		&= - \sum_{j=0}^{q-1} \sigma_j (\be \eu^{n-j})^T \big( \bfA(\wtx^n)-\bfA(\wtxls^n) \big) \bfu_*^n \\
		&= - \sum_{j=0}^{q-1} \sigma_j \be \Big( (\eu^{n-j})^T \big( \bfA(\wtx^n)-\bfA(\wtxls^n) \big) \bfu_*^n \Big) \\
		&\ + \sum_{j=0}^{q-1} \sigma_j  (\eu^{n-j-1})^T \be \big( \bfA(\wtx^n)-\bfA(\wtxls^n) \big) \bfu_*^n \\
		&\ + \sum_{j=0}^{q-1} \sigma_j  (\eu^{n-j-1})^T \big( \bfA(\wtx^{n-1})-\bfA(\wtxls^{n-1}) \big) \be \bfu_*^n .
	\end{aligned}
\end{equation}

We rewrite the summands of the term that contains $ \be \big( \bfA(\wtx^n)-\bfA(\wtxls^n) \big)$. Let us define the intermediate surfaces, for a fixed $n$,  
$$\widetilde\Ga_h^{\xi}=\Ga_h[\xi \wtx^n+(1-\xi) \wtx^{n-1}]
\quad\text{and}\quad \widetilde\Ga_{*,h}^{\xi}=\Ga_h[\xi \wtxls^n+(1-\xi) \wtxls^{n-1}] \quad\text{for $\xi\in[0,1]$}, 
$$
with velocities $V_{\widetilde\Ga_h^{\xi}}$ and $V_{\widetilde\Ga_{*,h}^{\xi}}$ that are independent of $\xi$ and correspond to the nodal vectors $\widetilde \bfV^n = \be\, \wtx^n$ and $\widetilde \bfV_*^n = \be\, \wtxls^n$, respectively. We denote by $e_u^\xi$ and $u_*^\xi$ the finite element functions with nodal vectors $\eu$ and $\bfu^*$, respectively.
We then have,  momentarily dropping the superscripts of $\eu$ and $\bfu_*$,
\begin{align*}
	&\ \eu^T \be \big( \bfA(\wtx^n)-\bfA(\wtxls^n) \big) \bfu_* \\
	&= \frac{1}{\tau} \bigg( \eu^T \big( \bfA(\wtx^n)-\bfA(\wtx^{n-1}) \big) \bfu_* - \eu^T \big( \bfA(\wtxls^n)-\bfA(\wtxls^{n-1}) \big) \bfu_*\bigg) \\
	&= \int_0^1 \int_{\widetilde\Ga_h^{\xi}} \nb_{\widetilde\Ga_{h}^{\xi}} e_u^\xi \cdot \big(D_{\widetilde\Ga_h^{\xi}} V_{\widetilde\Ga_h^{\xi}} \big) \nb_{\widetilde\Ga_{h}^{\xi}} u_*^\xi \d \xi \\
	&\ \qquad - \int_0^1 \int_{\widetilde\Ga_{*,h}^{\xi}} \nb_{\widetilde\Ga_{*,h}^{\xi}}e_u^\xi \cdot \big(D_{\widetilde\Ga_{*,h}^{\xi}} V_{\widetilde\Ga_{*,h}^{\xi}} \big) \nb_{\widetilde\Ga_{*,h}^{\xi}} u_*^\xi \d \xi \\
	&= \int_0^1 \frac{\d}{\d \theta} \int_0^1 \int_{\widetilde\Ga_h^{\xi,\theta}} \nb_{\widetilde\Ga_{h}^{\xi,\theta}} e_u^{\xi,\theta} \cdot \big(D_{\widetilde\Ga_h^{\xi,\theta}} V_{\widetilde\Ga_{h}^{\xi,\theta}}\big) \nb_{\widetilde\Ga_{h}^{\xi,\theta}}  u_*^{\xi,\theta}  \d \xi \d \theta ,
\end{align*}
where $\widetilde\Ga_h^{\xi,\theta} = \Ga_h[\wtxls^\xi + \theta \wtex^\xi]$, with $\wtxls^\xi=\xi \wtxls^n+(1-\xi) \wtxls^{n-1}$ and analogously for~$\wtex^\xi$, \blueon and where $D_{\widetilde\Ga_h^{\xi,\theta}} V_{\widetilde\Ga_{h}^{\xi,\theta}} =  \textnormal{tr}(G_{h}^{\xi,\theta}) I_3 - (G_{h}^{\xi,\theta}+(G_{h}^{\xi,\theta})^T)$ with $G_{h}^{\xi,\theta}=\nabla_{\widetilde\Ga_h^{\xi,\theta}} V_{\widetilde\Ga_{h}^{\xi,\theta}} \in \R^{3\times 3}$. \blueoff Furthermore, we let $E_V^{\xi,\theta} =V_{\widetilde\Ga_{h}^{\xi}} - V_{\widetilde\Ga_{*,h}^{\xi}}$ (which corresponds to the nodal values $\bfE_{\widetilde \bfV}^n = \widetilde \bfV^n - \widetilde \bfV_*^n$) and $V_{\widetilde\Ga_{h}^{\xi,\theta}} = V_{\widetilde\Ga_{*,h}^{\xi}} + \theta E_V^{\xi,\theta}$, and we have the relations $0 = \mat_\theta e_u^{\xi,\theta} = \mat_\theta u_*^{\xi,\theta}$ (since both functions are independent of $\theta$) and $\mat_\theta V_{\widetilde\Ga_{h}^{\xi,\theta}} \blueon = \mat_\theta (V_{\widetilde\Ga_{*,h}^{\xi}} + \theta E_V^{\xi,\theta}) \blueoff = E_V^{\xi,\theta}$. After computing the derivative with respect to $\theta$, and using the formula \eqref{eq:mat-grad formula}, we obtain 
\begin{align*}
	&\ \eu^T \be \big( \bfA(\wtx^n)-\bfA(\wtxls^n) \big) \bfu_* \\
	%
	%
	&= 
	\int_0^1 \int_0^1 \int_{\widetilde\Ga_h^{\xi,\theta}} \nb_{\widetilde\Ga_h^{\xi,\theta}} e_u^{\xi,\theta} \cdot \mat_\theta (D_{\widetilde\Ga_h^{\xi,\theta}} V_{\widetilde\Ga_{h}^{\xi,\theta}}) \nb_{\widetilde\Ga_h^{\xi,\theta}} u_*^{\xi,\theta}  \d \xi \d \theta \\
	&\ + \int_0^1 \int_0^1 \int_{\widetilde\Ga_h^{\xi,\theta}} \nb_{\widetilde\Ga_h^{\xi,\theta}} e_u^{\xi,\theta} \cdot (D_{\widetilde\Ga_h^{\xi,\theta}} e_{\widetilde x}^{\xi,\theta}) (D_{\widetilde\Ga_h^{\xi,\theta}} V_{\widetilde\Ga_{h}^{\xi,\theta}}) \nb_{\widetilde\Ga_h^{\xi,\theta}} u_*^{\xi,\theta}   \d \xi \d \theta .
\end{align*} 
\bbk %
%
%
%
Using the formula \eqref{eq:mat-D formula}, and that the nodal values of $E_V^{\xi,\theta}$ are $\bfE_{\widetilde \bfV}^n = \widetilde \bfV^n - \widetilde \bfV_*^n = \be\, (\wtx^n - \wtxls^n) = \be\, \wtex^n$,  and $\wtex^{n} = \sum_{j=0}^{q-1} \gamma_j \ex^{n-1-j}$, we obtain from the lemmas and bounds of Section~\ref{section: aux} 
\begin{equation}
\label{eq:R - middle term estimate}
	\begin{aligned}
		&\ (\eu^n)^T \be \big( \bfA(\wtx^n)-\bfA(\wtxls^n) \big) \bfu_*^n \\
		&\leq c \|u_*^n\|_{W^{1,\infty}(\Ga_h[\xls^n])} \|\eu^n\|_{\bfA(\wtx^n)} \big(\|\be \wtex^n\|_{\bfK(\wtx^n)} \! + \! \|\wtex^n\|_{\bfK(\wtx^n)} \! + \! \|\wtex^{n-1}\|_{\bfK(\wtx^{n-1})}\big) \\
		&\leq c \|\eu^n\|_{\bfA(\wtx^n)} \big(\|\be \wtex^n\|_{\bfK(\wtx^n)} \! + \! \|\wtex^n\|_{\bfK(\wtx^n)} \! + \! \|\wtex^{n-1}\|_{\bfK(\wtx^{n-1})}\big) \\
		&\leq c \|\eu^n\|_{\bfA(\wtx^n)} \Big( \sum_{j=n-q}^{n-1} \| \be \ex^j\|_{\bfK(\wtx^j)} + \sum_{j=n-q-1}^{n-1} \|\ex^j\|_{\bfK(\wtx^j)} \Big) \\
		&\leq c \|\eu^n\|_{\bfA(\wtx^n)}^2 + c \Big( \sum_{j=n-q}^{n-1} \| \be \ex^j\|_{\bfK(\wtx^j)}^2 + \sum_{j=n-q-1}^{n-1} \|\ex^j\|_{\bfK(\wtx^j)}^2 \Big) ,
	\end{aligned}
\end{equation}
\blueon 
where the term $D_{\widetilde\Ga_h^{\xi,\theta}} V_{\widetilde\Ga_{h}^{\xi,\theta}}$ is bounded using \eqref{bar-vhn} and the $W^{1,\infty}$-boundedness of the exact solution. 
\blueoff 

The last term in \eqref{eq:discrete product rule formula} is bounded, like the corresponding term in the proof of Proposition~\ref{proposition:stability - coupled problem}, by
\begin{equation}
\label{eq:R - third term estimate}
	\begin{aligned}
		&\ \sum_{j=0}^{q-1} \sigma_j  (\eu^{n-j-1})^T \big( \bfA(\wtx^{n-1})-\bfA(\wtxls^{n-1}) \big) \be \bfu_*^n \\
		&\leq c \sum_{j=0}^{q-1} \|\eu^{n-j-1}\|_{\bfA(\wtx^{n-1})} \|\wtex^{n-1}\|_{\bfK(\wtx^{n-1})} \\
		&\leq c \sum_{j=n-q}^{n-1} \|\eu^{j}\|_{\bfA(\wtx^{j})}^2 + c \sum_{j=n-q}^{n-1} \|\ex^{j}\|_{\bfK(\wtx^{j})}^2 .
	\end{aligned}
\end{equation}
In summary, we have shown that
\begin{equation}\label{RA1n}
	R_{\bfA}^{1,n}
		=  - \sum_{j=0}^{q-1} \sigma_j \be \Big( (\eu^{n-j})^T \big( \bfA(\wtx^n)-\bfA(\wtxls^n) \big) \bfu_*^n \Big) + r_{\bfA}^{1,n} ,
\end{equation}
where
\begin{equation}\label{rA1n}
r_{\bfA}^{1,n}  \leq  c  \sum_{j=n-q}^{n} \|\eu^{j}\|_{\bfA(\wtx^{j})}^2  + c \Big( \sum_{j=n-q}^{n-1} \| \be \ex^j\|_{\bfK(\wtx^j)}^2 + \sum_{j=n-q-1}^{n-1} \|\ex^j\|_{\bfK(\wtx^j)}^2 \Big) .
\end{equation}
In the same way we obtain
\begin{equation}\label{RA2n}
	R_{\bfA}^{2,n}
		=  \eta \sum_{j=0}^{q-1} \sigma_j \be \Big( (\eu^{n-j})^T \big( \bfA(\wtx^{n-1})-\bfA(\wtxls^{n-1}) \big) \bfu_*^{n-1} \Big) + \,r_{\bfA}^{2,n} 
\end{equation}
with
\begin{equation}\label{rA2n}
r_{\bfA}^{2,n}  \leq  \eta c  \sum_{j=n-q}^{n} \|\eu^{j}\|_{\bfA(\wtx^{j})}^2 + \eta c \Big( \sum_{j=n-q-1}^{n-2} \| \be \ex^j\|_{\bfK(\wtx^j)}^2 + \sum_{j=n-q-2}^{n-2} \|\ex^j\|_{\bfK(\wtx^j)}^2 \Big) .
\end{equation}
 
(v) The terms with the nonlinearity are bounded, with a small $\rho > 0$, by
\begin{equation}
\label{eq:rhs estimate - nonlinearities}
	\begin{aligned}
		& (\doteu^n)^T \big(\bff(\wtx^n,\widetilde \bfu^n) - \bff(\wtxls^n,\widetilde \bfu_*^n)\big)  
		- \eta (\doteu^n)^T \big(\bff(\wtx^{n-1},\widetilde \bfu^{n-1}) - \bff(\wtxls^{n-1},\widetilde \bfu_*^{n-1})\big) \\
		&\leq  c \|\doteu^n\|_{\bfM(\wtx^n)} \big( \|\wteu^n\|_{\bfK(\wtx^n)} + \|\wtex^n\|_{\bfK(\wtx^n)} \big) \\
		&\quad + c \eta \|\doteu^n\|_{\bfM(\wtx^{n-1})} \big( \|\wteu^{n-1}\|_{\bfK(\wtx^{n-1})} + \|\wtex^{n-1}\|_{\bfK(\wtx^{n-1})} \big) \\
		&\leq  \rho \|\doteu^n\|_{\bfM(\wtx^n)}^2 + c \|\wteu^n\|_{\bfK(\wtx^n)}^2 + c \|\wteu^{n-1}\|_{\bfK(\wtx^{n-1})}^2 \\
		&\quad + c \|\wtex^n\|_{\bfK(\wtx^n)}^2 + c\|\wtex^{n-1}\|_{\bfK(\wtx^{n-1})}^2 .
	\end{aligned}
\end{equation}

(vi) The terms with the defects are bounded, with a small $\rho > 0$, by
\redon 
\begin{equation}
\label{eq:rhs estimate - defects}
	\begin{aligned}
		& - (\doteu^n)^T \bfM(\wtxls^n) \du^n 
		+ \eta (\doteu^n)^T \bfM(\wtxls^{n-1}) \du^{n-1} \\
		&\leq  c \|\doteu^n\|_{\bfM(\wtx^n)} \|\du^n\|_{\bfM(\wtxls^n)}
		+ c \eta \|\doteu^n\|_{\bfM(\wtx^{n-1})} \|\du^{n-1}\|_{\bfM(\wtxls^{n-1})} \\
		&\leq  \rho \|\doteu^n\|_{\bfM(\wtx^n)}^2 + c \|\du^n\|_{\bfM(\wtxls^n) }^2 + c \|\du^{n-1}\|_{\bfM(\wtxls^{n-1} )}^2 ,
	\end{aligned}
\end{equation}
where we used the norm equivalences \eqref{norm-equiv} between $\wtx^n$ and $\wtxls^{n}$ (for $n-1$ as well), and $\wtxls^n$ and $\wtxls^{n-1}$, (for the former, using the bounds \eqref{eq:assumed bounds - n}). \redoff

%

 Inserting the estimates from (i)--(vi) into  \eqref{eq:error eq tested - u - full}
altogether yields the formi\-dable inequality, for $n\ge q+1$ with $n\tau\le t^*$,
\begin{equation}
\label{eq:stability estimate - pre sum}
	\begin{aligned}
		& \big( (1-c\tau) - \tfrac12\eta(1+ c\tau) - c\rho \big) \|\doteu^n\|_{\bfM(\wtx^{n})}^2 \! - \! \tfrac12\eta  \|\doteu^{n-1}\|_{\bfM(\wtx^{n-1})}^2 \\
		&\quad + \frac{1}{\tau} \! \!  \sum_{i,j=1}^q g_{ij} (\eu^{n-q+i})^T \bfA(\wtx^n) \eu^{n-q+j} 
		- \frac{1}{\tau} \! \! \sum_{i,j=1}^q g_{ij} (\eu^{n-q+i-1})^T \bfA(\wtx^n) \eu^{n-q+j-1} \\
		&\leq 
		- \sum_{j=0}^{q-1} \sigma_j \be \Big( (\eu^{n-j})^T \big( \bfA(\wtx^n)-\bfA(\wtxls^n) \big) \bfu_*^n \Big)
		\\
		&\quad	+ \eta \sum_{j=0}^{q-1} \sigma_j \be \Big( (\eu^{n-j})^T \big( \bfA(\wtx^{n-1})-\bfA(\wtxls^{n-1}) \big) \bfu_*^{n-1} \Big)
		+ c\eps_n
	\end{aligned}
\end{equation}
with 
\begin{equation} \label{eps-n}
  \begin{aligned}
  \eps_n &=  \sum_{j=0}^{q} \|\eu^{n-j}\|_{\bfA(\wtx^{n-j})}^2  + \sum_{j=0}^{q} \|\ex^{n-1-j}\|_{\bfK(\wtx^{n-1-j})}^2 
   +   \sum_{j=0}^{q-1} \| \be \ex^{n-1-j}\|_{\bfK(\wtx^{n-1-j})}^2
		\\ &\quad 
		+   \redon \|\du^n\|_{\bfM(\wtxls^n)}^2 \redoff + \redon \|\du^{n-1}\|_{\bfM(\wtxls^{n-1})}^2 \redoff .
  \end{aligned}
\end{equation}

%
%

We now sum up the inequality \eqref{eq:stability estimate - pre sum} from $n=q+1$ to $\bar n$ with $\bar n \tau \leq t^*$ and multiply with the step size $\tau$. We estimate the arising terms separately.

(a) If both $\tau$ and $\rho$ are so small that $(1-c\tau)-\eta(1+c\tau) - c\rho \geq (1-\eta)/2$ (recall that $\eta<1$), then
\begin{align*}
&  \tau \!\!\! \sum_{n=q+1}^{\bar n} \bigg( \big( (1-c\tau) - \tfrac12\eta(1+ c\tau) - c\rho \big) \|\doteu^n\|_{\bfM(\wtx^{n})}^2 \! - \! \tfrac12\eta  \|\doteu^{n-1}\|_{\bfM(\wtx^{n-1})}^2 \bigg) \\
&\geq  \tfrac{1}{2}(1-\eta) \tau \! \sum_{n=q+1}^{\bar n} \! \|\doteu^n\|_{\bfM(\wtx^{n})}^2 \blueon -  \tfrac12\eta  \tau\|\doteu^{q}\|_{\bfM(\wtx^{q})}^2 \blueoff .
\end{align*}

(b) For the second line in \eqref{eq:stability estimate - pre sum} we have, with $\gamma>0$ denoting the smallest eigenvalue of the positive definite matrix $(g_{ij})$, and using \eqref{eq:assumed bounds - n} with the bound~\eqref{eq:matrix difference bounds},
\begin{align*}
& \tau \!\!\! \sum_{n=q+1}^{\bar n} \Bigl( \frac{1}{\tau} \! \!  \sum_{i,j=1}^q g_{ij} (\eu^{n-q+i})^T \bfA(\wtx^n) \eu^{n-q+j} 
\\
&\qquad\quad		- \frac{1}{\tau} \! \! \sum_{i,j=1}^q g_{ij} (\eu^{n-q+i-1})^T \bfA(\wtx^n) \eu^{n-q+j-1}  \Bigr)
\\
		&= \sum_{i,j=1}^q g_{ij} (\eu^{\bar n-q+i})^T \bfA(\wtx^{\bar n}) \eu^{\bar n-q+j} -
		\sum_{i,j=1}^q g_{ij} (\eu^{i})^T \bfA(\wtx^{q+1}) \eu^{j} 
		\\
		&\quad - \sum_{n=q+1}^{\bar n-1} \sum_{i,j=1}^q g_{ij} (\eu^{n-q+i})^T \bigl(\bfA(\wtx^{n+1}) - \bfA(\wtx^{n}) \bigr) \eu^{n-q+j}
\\
& \ge \gamma \sum_{j=0}^{q-1} \|\eu^{\bar n - j}\|_{\bfA(\wtx^{\bar n})}^2 - c \sum_{i=1}^{q} \|\eu^i\|_{\bfA(\wtx^{q+1})}^2 
- c \tau \sum_{n=q+1}^{\bar n-1} \|\eu^n\|_{\bfA(\wtx^n)}^2
\\
&\ge \tfrac12 \gamma \sum_{j=0}^{q-1} \|\eu^{\bar n - j}\|_{\bfA(\wtx^{\bar n -j})}^2 - c \sum_{i=1}^{q} \|\eu^i\|_{\bfA(\wtx^{i})}^2 
- c \tau \sum_{n=q+1}^{\bar n-1} \|\eu^n\|_{\bfA(\wtx^n)}^2.
\end{align*}

(c) For the first term on the right-hand side of \eqref{eq:stability estimate - pre sum} we have, after exchanging the sums,  a telescoping sum
\begin{align*}
&-  \tau \!\!\! \sum_{n=q+1}^{\bar n}\sum_{j=0}^{q-1} \sigma_j \be \Big( (\eu^{n-j})^T \big( \bfA(\wtx^n)-\bfA(\wtxls^n) \big) \bfu_*^n \Big)
\\
&= - \sum_{j=0}^{q-1} \sigma_j \Bigl( (\eu^{{\bar n}-j})^T \big( \bfA(\wtx^{\bar n})-\bfA(\wtxls^{\bar n}) \big) \bfu_*^{\bar n} 
- (\eu^{q-j})^T \big( \bfA(\wtx^{q})-\bfA(\wtxls^{q}) \big) \bfu_*^{q} \Bigr)
\\
&\le c \sum_{j=0}^{q-1} \Bigl( \| \eu^{{\bar n}-j} \|_{\bfA(\wtx^{\bar n})}   \| \wtex^{{\bar n}} \|_{\bfA(\wtx^{\bar n})}
+   \| \eu^{{q}-j} \|_{\bfA(\wtx^{q})}   \| \wtex^{{q}} \|_{\bfA(\wtx^{q})} \Bigr)
\\
& \blueon \leq c \sum_{j=0}^{q-1} \| \eu^{{\bar n}-j} \|_{\bfA(\wtx^{\bar n})} \sum_{j=0}^{q-1} \| \ex^{{\bar n-1-j}} \|_{\bfA(\wtx^{\bar n-1-j})}
+ c \sum_{j=0}^{q-1} \| \eu^{{q}-j} \|_{\bfA(\wtx^{q})}  \| \wtex^{{q}} \|_{\bfA(\wtx^{q})} \blueoff \\
&\le \tfrac14\gamma \sum_{j=0}^{q-1}  \|\eu^{\bar n - j}\|_{\bfA(\wtx^{\bar n - j})}^2 + 
c\sum_{j=0}^{q-1}  \| \ex^{\bar n - 1- j} \|_{\bfA(\wtx^{\bar n - 1- j})}^2 
\\
&\quad +
c \sum_{i=0}^{q-1} \Bigl( \|\eu^{i}\|_{\bfA(\bfx^{i})}^2 +  \|\ex^{i}\|_{\bfA(\bfx^{i})}^2 \Bigr) +
c  \| \eu^{{q}} \|_{\bfA(\wtx^{q})}^2 ,
\end{align*}
\blueon 
where for the second to last estimate we used that $\wtex^{n} = \sum_{j=0}^{q-1} \gamma_j \ex^{n-1-j}$ and the norm equivalence \eqref{norm-equiv-n}.
\blueoff 

A bound of the same type holds for the second term on the right-hand side.

(d)  To  bound $\eu^q$, we test \eqref{eq:error eq - u - full} for $n=q$ with $\doteu^q$, multiply with $\tau$, and use  that $\delta_0>0$ and
the techniques presented above in (iii)--(vi) to conclude  that
$$
\tau \|\doteu^q\|_{\bfM(\wtx^q)}^2  + \delta_0	\|\eu^q\|_{\bfA(\wtx^q)}^2 
	\leq c \sum_{i=0}^{q-1}\Bigl( \|\eu^i\|_{\bfA(\bfx^{i})}^2 + \|\ex^i\|_{\bfK(\bfx^{i})}^2 \Bigr)+ c \tau \|\du^q\|_{\redon \bfM(\wtxls^q)\redoff }^2 .
$$

Combining the estimates of (a)--(d), we thus obtain
\begin{equation}
\label{eq:stability estimate-abc}
	\begin{aligned}
		& \tfrac{1}{2}(1-\eta) \tau \! \sum_{n=q+1}^{\bar n} \! \|\doteu^n\|_{\bfM(\wtx^{n})}^2 
		+ \tfrac14 \gamma \sum_{j=0}^{q-1} \|\eu^{\bar n - j}\|_{\bfA(\wtx^{\bar n - j})}^2 \\
                 &\le c\sum_{j=0}^{q-1}  \| \ex^{\bar n - 1- j} \|_{\bfA(\wtx^{\bar n - 1- j})}^2 
		+ c\tau\sum_{n=q+1}^{\bar n} \eps_n
		\\
&\quad +
c \sum_{i=0}^{q-1} \Bigl( \|\eu^{i}\|_{\bfA(\bfx^{i})}^2 +  \|\ex^{i}\|_{\bfA(\bfx^{i})}^2 \Bigr) +
c \tau \|\du^q\|_{\redon \bfM(\wtxls^q)\redoff }^2 .
	\end{aligned}
\end{equation}
The term $\tau\sum_{n=q+1}^{\bar n} \eps_n$, which was defined in \eqref{eps-n}, contains a sum of squared norms of $\partial^\tau \ex^n = (\ex^n-\ex^{n-1})/\tau$, for which we show 
\begin{equation}\label{sum-eps-n}
		\tau \sum_{n=q}^{\bar n} \| \be \ex^n\|_{\bfK(\wtx^n)}^2 \le c\tau \sum_{n=q}^{\bar n} \|  \dotex^n\|_{\bfK(\wtx^n)}^2 
		+ \redon \frac c \tau \redoff \sum_{i=0}^{q-1} \| \ex^i \|_{\bfK(\bfx^{i})}^2.
		\end{equation}
To prove this inequality, we start from the equation (cf.~\eqref{eq:connecting discrete derivatives})
$$
	\dotex^n = \frac1\tau\sum_{j=0}^{q} \delta_j \, \ex^{n-j}  = 	\sum_{j=0}^{q-1} \sigma_j \, \frac1\tau\bigl( \ex^{n-j} - \ex^{n-j-1} \bigr) =
	\sum_{j=0}^{q-1} \sigma_j \, \be \ex^{n-j}	,
$$
which we rewrite, for $n\ge q$, as
$$
	\dotex^n = 
	\sum_{j=q}^n \sigma_{n-j} \be \ex^{j} + \bfs^n,
$$
with $\sigma_j=0$ for $j\ge q$ and $\bfs^n = \sum_{i=1}^{q-1} \sigma_{n-i} \be \ex^i$, for which we note that $\bfs^n=0$ for $n\ge 2q-1$.
We then obtain, as in \eqref{inverse convolution} with \eqref{eq:inverse of sigma},
$$
\be \ex^{n} = \sum_{j=q}^n \chi_{n-j} \bigl(  \dotex^j - \bfs^j).
$$
Since the convolution with the absolutely summable sequence $(\chi_j)$ is a bounded operation in $\ell^2$ and we have the norm equivalence~\eqref{norm-equiv-n} (with $m=\bar n$), we obtain \blueon 
\begin{equation*}
	\tau \sum_{n=q}^{\bar n} \| \be \ex^n\|_{\bfK(\wtx^n)}^2 \leq c\tau \sum_{n=q}^{\bar n} \|  \dotex^n\|_{\bfK(\wtx^n)}^2 + c \tau \sum_{n=q}^{2q-2} \| \bfs^n \|_{\bfK(\bfx^{i})}^2 , 
\end{equation*}
\redon 
where the last term is further bounded by
$$
\tau \sum_{n=q}^{2q-2} \| \bfs^n \|_{\bfK(\bfx^{i})}^2  \leq c \tau \sum_{i=1}^{q-1} \| \be \ex^i \|_{\bfK(\bfx^{i})}^2 \leq
\frac{c}{\tau} \sum_{i=0}^{q-1} \| \ex^i \|_{\bfK(\bfx^{i})}^2.
$$
 This yields \blueoff \eqref{sum-eps-n}.
 \redoff

For the right-hand side of \eqref{sum-eps-n} we note that by \eqref{eq:error eq - x - full},
\begin{equation} \label{dotex-ev}
\tau \sum_{n=q}^{\bar n} \|  \dotex^n\|_{\bfK(\wtx^n)}^2 \le 2 \tau \sum_{n=q}^{\bar n} \|  \ev^n\|_{\bfK(\wtx^n)}^2  +
2 \tau \sum_{n=q}^{\bar n} \|\dx^n\|_{\redon \bfK(\wtxls^n)\redoff }^2 .
\end{equation}

Next we show that for $\bar n\ge q$ with $\bar n\tau\le T$,
\begin{equation}
\label{eu-doteu}
	\|\eu^{\bar n}\|_{\bfM(\xls^{\bar n})}^2 \leq  c\tau \sum_{n=q}^{\bar n}  \|\doteu^n\|_{\bfM(\xls^{n})}^2 +  c \sum_{i=0}^{q-1} \|\eu^i\|_{\bfM(\xls^{i})}^2 .
\end{equation}
The proof is similar to that of \eqref{sum-eps-n}.
We rewrite
$$
\doteu^n = \frac{1}{\tau} \sum_{j=q}^n \delta_{n-j} \eu^{j}  + \bfs^n \qquad\text{with}\qquad \bfs^n =  \frac{1}{\tau} \sum_{i=0}^{q-1} \delta_{n-i} \eu^{i}
$$
and $\delta_j=0$ for $j>q$, and we
note that $\bfs^n = 0$ for $n\ge2q$. 
With the coefficients of the power series
$$
	\omega(\zeta) = \sum_{n=0}^\infty \omega_n \zeta^n = \frac 1{\delta(\zeta)}, \quad\text{ for }\quad \delta(\zeta)=\sum_{j=0}^q \delta_j\zeta^j ,
$$
the same argument as in \eqref{inverse convolution} gives us, for $n\ge q$,
$$
	\eu^n =\tau \sum_{j=q}^n \omega_{n-j}  (\doteu^j-\bfs^j) .
$$
%
By the zero-stability of the BDF method of order $q\le 6$ (which states that all zeros of $\delta(\zeta)$ are outside the unit circle with the exception of the simple zero at $\zeta=1$), the coefficients $\omega_n$ are bounded: $|\omega_n| \le c$ for all $n$.
We take the $\mathbf{M}(\wtx^n)$ norm on both sides and recall that by \eqref{norm-equiv-n} all these norms are uniformly equivalent for $0\le n\tau\le t^*$. With the Cauchy--Schwarz inequality we then obtain \eqref{eu-doteu}.

Recalling the definition \eqref{eps-n} of $\eps_n$ and combining the bounds \eqref{eq:stability estimate-abc}--\eqref{eu-doteu}, we finally obtain (writing here $n$ instead of $\bar n$) for $n\ge q$ with $n\tau\le t^*$
\begin{equation}
\label{eq:stability estimate-eu}
	\begin{aligned}
		  \|\eu^{ n }\|_{\bfK(\wtx^{ n })}^2 
                 &\le c\sum_{i=0}^{q-1}  \| \ex^{  n - 1- j} \|_{\bfK(\wtx^{  n - 1- j})}^2 
  \\
  &\quad +  c\tau\sum_{j=q}^{  n} \Bigl( \|\ex^j\|_{\bfK(\wtx^{j})}^2 + \|\ev^j\|_{\bfK(\wtx^{j})}^2 + \|\eu^j\|_{\bfK(\wtx^{j})}^2 \Bigr)
		\\
&\quad +
c \sum_{i=0}^{q-1} \Bigl( \|\eu^{i}\|_{\bfK(\bfx^{i})}^2 +  \|\ex^{i}\|_{\bfK(\bfx^{i})}^2 \Bigr) 
\\
&\quad +
c \tau \sum_{j=q}^{n} \Bigl( \|\dx^j\|_{\redon \bfK(\wtxls^j)\redoff }^2 + \|\du^j\|_{\redon \bfM(\wtxls^{j})\redoff }^2 \Bigr).
	\end{aligned}
\end{equation}
%
%

From \eqref{eq:error eq - x - full} we obtain by the same argument as in the proof of \eqref{eu-doteu} a bound for $\ex^n$,
\begin{equation}
\label{ex-ev}
 \|\ex^{n}\|_{\bfK(\wtx^{n})}^2 \le  c\tau \sum_{j=q}^{n}  \|\ev^j\|_{\bfK(\wtx^{j})}^2 +  c \sum_{i=0}^{q-1} \|\ex^i\|_{\bfK(\bfx^{i})}^2 
 + c\tau \sum_{j=q}^{n} \|\dx^j\|_{\redon \bfK(\wtxls^{j})\redoff }^2 .
\end{equation}

(B) \emph{Estimates for the velocity equation:} We test \eqref{eq:error eq - v - full} with $\ev^n$, which yields
$$
\|\ev^n\|_{\bfK(\wtx^n)}^2 
		= (\ev^n)^T \bfK(\wtxls^n) \ev^n = (\ev^n)^T \rv^n,
$$
where the right-hand side is estimated in the same way as in part (B) of the proof of Proposition~\ref{proposition:stability - coupled problem} to obtain a bound like \eqref{eq:final estimate surface} for the discrete error in~$v$:
$$
	\|\ev^n\|_{\bfK( \wtx^n)}^2  \leq 
	c  \|\wtex^{n}\|_{\bfK( \wtx^n)}^2 + c  \|\wteu^{n}\|_{\bfK( \wtx^n)}^2 + c \|\dv^n\|_{\star, \redon \wtxls^n\redoff }^2 ,
$$
where $\wtex^{n} = \widetilde \bfx^n - \wtxls^n = \sum_{i=0}^{q-1} \gamma_i \ex^{n-1-i}$ and 
$\wteu^{n} = \widetilde \bfu^n - \wtuls^n = \sum_{i=0}^{q-1} \gamma_i \eu^{n-1-i}$.
This yields
\begin{equation}
\label{eq:final estimate surface - full}
	\|\ev^n\|_{\bfK( \wtx^n)}^2  \leq 
	c \sum_{i=0}^{q-1}\Big( \|\ex^{n-1-i}\|_{\bfK( \wtx^{n-1-i})}^2 +  \|\eu^{n-1-i}\|_{\bfK( \wtx^{n-1-i})}^2\Big) + c \|\dv^n\|_{\star, \redon \wtxls^n\redoff }^2 .
\end{equation}

(C) \emph{Combination:} Combining the bounds \eqref{eq:stability estimate-eu}--\eqref{eq:final estimate surface - full} and using a discrete Gronwall inequality finally yields the stability estimate \eqref{eq:stability bound - full} in the norms that correspond to $\wtx^n$ instead of $\xls^n$. However, these norms are equivalent uniformly in $h$, $\tau$ and $n$ with $n\tau\le t^*$ by \eqref{norm-equiv}, since the required smallness condition in the $W^{1,\infty}$ norm is satisfied because of the obtained error bound for $e_x$ in the $H^1$-norm and an inverse inequality, noting that
$$
\wtx^n - \xls^n = \sum_{j=0}^{q-1} \gamma_j \ex^{n-1-j}  +( \wtxls^n -  \xls^n ),
$$
where the last term is bounded by $c\tau^q$ in the $L^\infty$ norm.

Finally, similarly as before, we also obtain $t^*=T$, since our error estimates show that the required bounds \eqref{eq:assumed bounds - n} remain satisfied from one step to the next provided that $\tau$ and $h$ are sufficiently small and are related by $\tau\le C_0h$.
\qed
\end{proof}

\ebk

\section{Defect bounds for the full discretization}
\label{section:defect-full}
The consistency errors defined by \eqref{BDF-defects} are bounded as follows.
\begin{lemma}
\label{lemma: semidiscrete residual - full}
Let the  surface $X$ evolving under mean curvature flow be sufficiently regular on the time interval $[0,T]$. Then, there exist $h_0>0$ and $\tau_0>0$ such that for all $h\leq h_0$, $\tau\le\tau_0$ and $n\geq$ with $n\tau\le T$, the  defects 
$d_x^n\in S_h(\Gamma_h[\bfx^*])^3$, $d_v^n\in S_h(\Gamma_h[\bfx^*])^3$ and $d_u^n\in S_h(\Gamma_h[\bfx^*])^4$ of the $k$th-degree finite elements and the $q$-step backward difference formula,  as defined by their nodal vectors $\dx^n$, $\dv^n$ and $\du^n$ in \eqref{BDF-defects}, are bounded as
    \begin{align*}
        \|\dx^n\|_{\bfK(\xls^n)} =&\ \|d_x^n\|_{H^1(\Ga_h[\xs(t_n)])}\leq\ c\tau^q, \\
        \|\dv^n\|_{\star,\xls^n}=&\ \|d_v^n\|_{H_h\inv(\Ga_h[\xs(t_n)])} \leq c (h^k+\tau^q) , \\
        \|\du^n\|_{\bfM(\xls^n)}=&\ \|d_u^n\|_{L^2(\Ga_h[\xs(t_n)])} \leq c (h^k +\tau^q) .
    \end{align*}
 The constant $c$ is independent of $h$, $\tau$ and $n$ with $n\tau\le T$.
\end{lemma}

\begin{proof}
Since we have
\begin{equation*}
	\dx^n = \frac{1}{\tau} \sum_{j=0}^q \delta_j \xs(t_{n-j}) - \dotxs(t_n) ,
\end{equation*}
the bound $\|\dx^n\|_{\bfK(\xls^n)} \leq c \tau^p $ just follows by Taylor expansion.

Comparing the defects $\dv(t_n)$ of the semi-discretization, see \eqref{dv},  and $\dv^n$ of the full discretization, see \eqref{BDF-defects}, we find that they are related by
$$
	\redon \bfM(\wtxls^n) \redoff \dv^n = \bfM(\xls^n)\dv(t_n) + \bigl( \bfK(\widetilde\bfx_*^n) - \bfK(\xls^n) \bigr)\bfv_*^n
 - \bigl(\bfg(\widetilde\bfx_*^n,\widetilde\bfu_*^n) - \bfg(\xls^n,\uls^n) \bigr).
$$
Using Lemma~\ref{lemma: semidiscrete residual} to bound $\dv(t_n)$ and the same arguments as in the proof of \bbk Lemma~6.2 \ebk in \cite{KL2018} to bound the remaining terms, and $\|\widetilde \bfx_*^n - \bfx_*^n\|_{\bfK(\bfx_*^n)} \leq c \tau^p$ (see~\cite[Lemma~6.2]{KL2018}), we obtain the stated bound for $\dv^n$.

Similarly, the defects $\du(t_n)$ of the semi-discretization, see \eqref{du},  and $\du^n$ of the full discretization, see \eqref{BDF-defects},  are related by
\begin{align*}
	\redon \bfM(\wtxls^n) \redoff \du^n = &\ \bfM(\bfx_*^n)\du(t_n) \\
	&\ \blueon + \bfM(\widetilde \bfx_*^n) \bigl(\dot \bfu_*^{n} -\dot \bfu_*(t_{n}) \bigr) 
	+ \bigl(\bfM(\widetilde \bfx_*^n)  - \bfM(\bfx_*^n)\bigr)\dot \bfu_*(t_{n}) 
	\blueoff \\
	&\ + \bigl(\bfA(\widetilde \bfx_*^n) - \bfA(\bfx_*^n)\bigr) \bfu_*^{n}
	 - \bigl(\bff(\widetilde\bfx_*^n,\widetilde\bfu_*^n) - \bff(\xls^n,\uls^n) \bigr).
\end{align*}
The term $\du(t_n)$ is estimated in Lemma~\ref{lemma: semidiscrete residual}, \blueon while the difference of the discrete and continuous time derivatives is estimated via a Peano kernel representation, see \cite[Section~3.2.6]{Gautschi}, and $\|\widetilde \bfx_*^n - \bfx_*^n\|_{\bfK(\bfx_*^n)} \leq c \tau^p$, from~\cite[Lemma~6.2]{KL2018}. \blueoff The further terms are estimated like the analogous terms in the equation for $\dv^n$ above. This yields the stated bounds.
\qed
\end{proof}

\section{Proof of Theorem~\ref{MainTHM-full}}
\label{section: proof completed - full}
With the stability estimate of Proposition~\ref{proposition:stability - coupled problem - full} and the defect bounds of Lemma \ref{lemma: semidiscrete residual - full} at hand, the proof is now essentially the same as in Section~\ref{section: proof completed}.
The errors are decomposed using finite element interpolations of $X$ and $v$ and the Ritz map \eqref{uhs-ritz} for $u$ and using the  composed lift $L$ from Section~\ref{subsec:lifts}: 
\begin{align*}
(X_h^n)^L  - X(\cdot,t_n)  =&\ \big( \widehat X_h^n  - X_h^*(\cdot,t_n)  \big)^{l} +  \big((X_h^*(\cdot,t_n))^l  - X(\cdot,t_n) \big), \\
(v_h^n)^L  - v(\cdot,t_n)  =&\ \big( \widehat v_h^n  - v_h^*(\cdot,t_n)  \big)^{l} +  \big((v_h^*(\cdot,t_n))^l  - v(\cdot,t_n) \big), \\
(u_h^n)^L  - u(\cdot,t_n)  =&\ \big( \widehat u_h^n  - u_h^*(\cdot,t_n)  \big)^{l} +  \big((u_h^*(\cdot,t_n))^l  - u(\cdot,t_n) \big).
\end{align*}
The last terms  in these formulas can be bounded in the $H^1(\Gamma)$ norm by $Ch^k$, using the interpolation and Ritz map error bounds of \cite{Kovacs2017}.

To bound the first terms on the right-hand sides, we use  the stability estimate of Proposition~\ref{proposition:stability - coupled problem - full} together with the defect bounds of Lemma~\ref{lemma: semidiscrete residual - full} to obtain, for $t_n=n\tau\le T$ and under the required step-size restriction,
$$
\| \ex^n \|_{\bfK(\xs(t_n))} + \| \ev^n \|_{\bfK(\xs(t_n))} + \| \eu^n \|_{\bfK(\xs(t_n))}
 \le C(h ^k + \tau^q).
$$
By the equivalence of norms shown in \cite[Lemma~3]{Dziuk88} and by \eqref{K-H1} we have 
\begin{align*}
   & \| \big(\widehat u_h^n - u_h^*(t_n)\big)^{l} \|_{H^1(\Ga[X(t_n)])} \le
     c \| \widehat u_h^n - u_h^*(t_n) \|_{H^1(\Ga[X_h^*(t_n)])} 
     \\
   & = c \| \widehat u_h^n - u_h^*(t_n) \|_{H^1(\Ga_h[\xs(t_n)])}
    = c \| \eu^n \|_{\bfK(\xs(t_n))},
\end{align*}
and similarly for $\widehat v_h^n - v_h^*(t_n)$ and $\widehat X_h^n - X_h^*(t_n)$. This proves the result.

\section{Numerical experiments}

We performed the following numerical experiments for mean curvature flow: 
\begin{itemize}
 \item
 for the sphere, where the exact evolution is known; see, e.g.~\cite{Huisken1984}; 
 \item 
 for a dumbbell-shaped surface, where the evolution is known to develop a pinch singularity in finite time; see, e.g.~\cite{Dziuk90}. 
Here we compare Dziuk's algorithm with the algorithm proposed in this paper and with a variant where the computed normal vector is rescaled to unit norm.
\end{itemize}
The initial meshes for both examples were generated using DistMesh \cite{distmesh} without taking advantage of the symmetry of the surface, as opposed to \cite[Figure~2]{Dziuk90} and \cite[Figure~14]{ElliottFritz_DT}.

\subsection{Mean curvature flow for the sphere}
The radius $R(t)$ of a two-dimensional sphere $\Ga\t$ evolving according to mean curvature flow satisfies the differential equation, cf.~\cite{Mantegazza}:
\begin{align*}
	R'\t = &\ - \frac{2}{R\t}, \qquad \textnormal{for} \quad t \in [0,T_{\max}]\\
	R(0) =&\ R_0,
\end{align*}
with solution $R\t = \sqrt{R_0^2-4t}$, which is zero at time $T_{\max} = R_0^2/4$. The mean curvature of $\Ga\t$ is $H(x,t) = \bbk 2 / R\t \ebk$, while the normal vector at $x \in \Ga\t$ is $\nu(x) = x/|x|$. The surface shrinks to a point at time $T_{\max}$.

\begin{figure}[t]
	\centering
	\includegraphics[width=\textwidth]{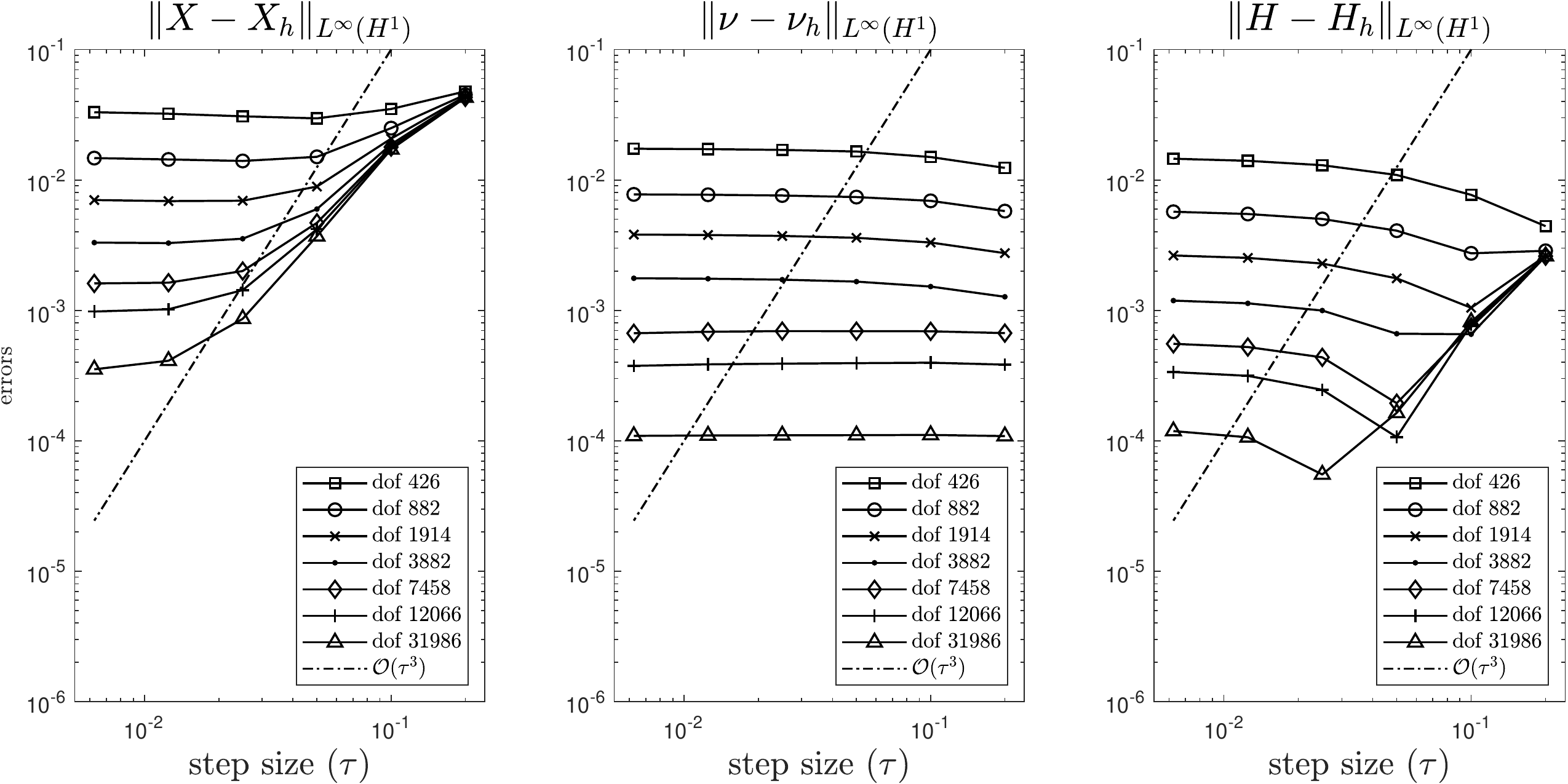}
	\caption{Temporal convergence of the BDF3 / quadratic ESFEM discretisation for MCF for a sphere}
	\label{fig:conv_time}
\end{figure}
\begin{figure}[t]
	\centering
	\includegraphics[width=\textwidth]{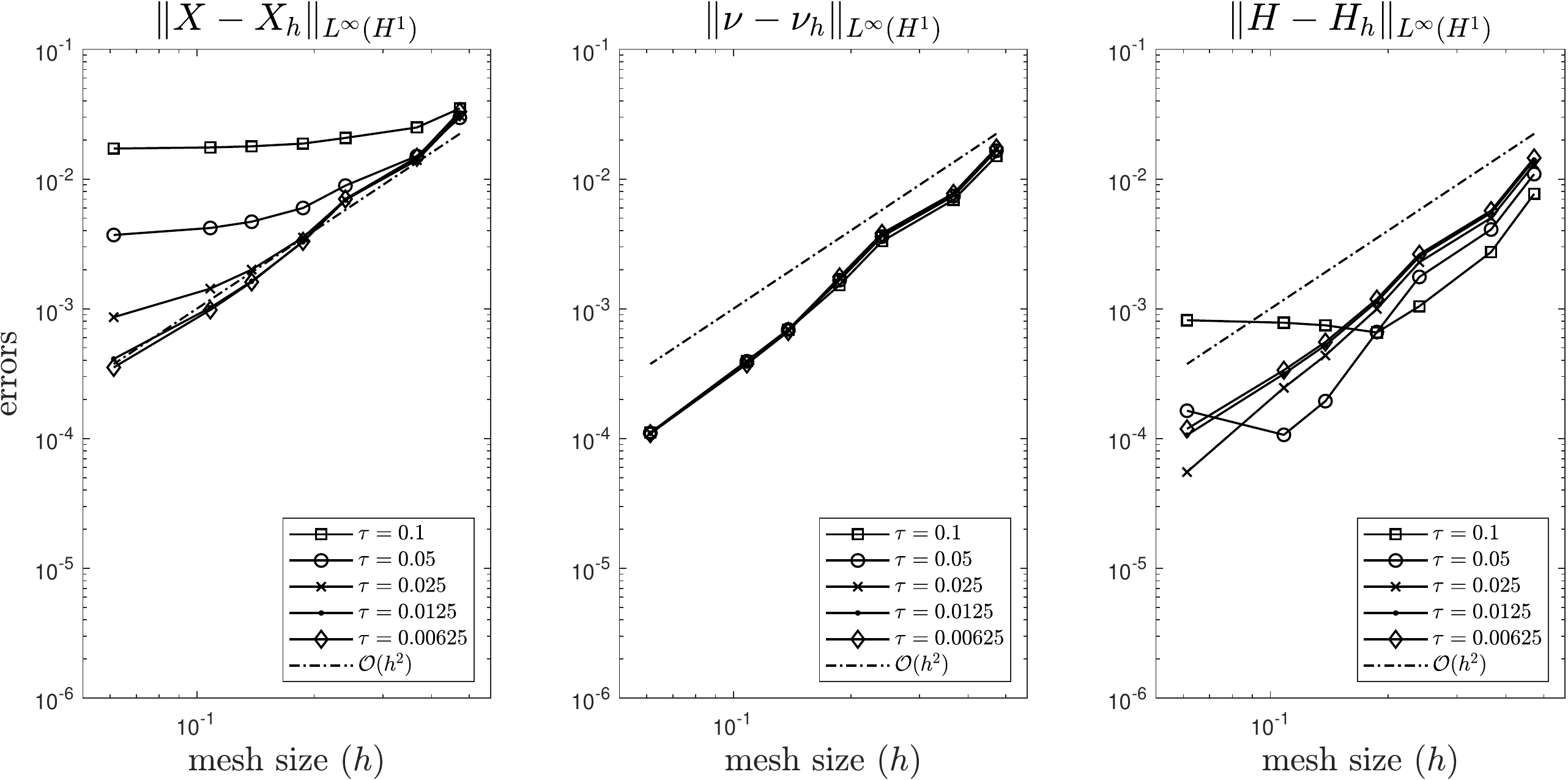}
	\caption{Spatial convergence of the BDF3 / quadratic ESFEM discretisation for MCF for a sphere}
	\label{fig:conv_space}
\end{figure}

Using the algorithm in \eqref{BDF} (combining the $3$-step BDF method and quadratic evolving surface finite elements) we computed approximations to mean curvature flow for the sphere of radius $R_0=2$, over the time interval $[0,T_{\max}]=[0,0.6]$. For our computations we used a sequence of time step sizes $\tau_k=\tau_{k-1}/2$ with $\tau_0 = 0.2$, and a sequence of initial meshes of mesh widths $h_k \approx 2^{-1/2} h_{k-1}$ with $h_0 \approx 0.5$.

In Figure~\ref{fig:conv_time} and \ref{fig:conv_space} we report the errors between the numerical and exact solution (both for the surface error and the error of the dynamic variables $\nu$ and $H$) until time $T=0.6$. 
The logarithmic plots show the \bbk $L^\infty(H^1)$ norm \ebk errors against the time step size $\tau$ in Figure~\ref{fig:conv_time}, and against the mesh width $h$ in Figure~\ref{fig:conv_space}.
The lines marked with different symbols correspond to different mesh refinements and to different time step sizes in Figure~\ref{fig:conv_time} and \ref{fig:conv_space}, respectively.

In Figure~\ref{fig:conv_time} we can observe two regions: a region where the temporal discretisation error dominates, matching to the $O(\tau^3)$  order of convergence of our theoretical results, and a region, with small time step size, where the spatial discretization error dominates (the error curves flatten out). For Figure~\ref{fig:conv_space}, the same description applies, but in the presented cases only the first region can be observed.

The convergence in time and in space as shown by Figures~\ref{fig:conv_time} and \ref{fig:conv_space}, respectively, is in agreement with the theoretical convergence results (note the reference lines).

\subsection{Singular mean curvature flow}

We consider the mean curvature flow starting from a dumbbell shaped surface~$\Ga^0$, given by the distance function
\begin{equation*}
	d(x) = x_1^2 + x_2^2 + G(x_3^2) - 0.04,
\end{equation*}
with $G(s)=2s(s-199/200)$. The evolution of such a surface is known to be singular, as it has often been reported in the literature before; see \cite[Figure~3]{White2002}, \cite[Figure~3.2]{Ecker2012}, or \cite[Figure~2]{Dziuk90}, \cite[Section~7]{ElliottFritz_DT} for numerical experiments.

In Figure~\ref{fig:singularMCF} we compare the algorithm of Dziuk \cite{Dziuk90} (see \eqref{eq:matrix--vector form Dziuk}) and Algorithm \eqref{BDF} (with the two-step BDF method; starting values at $t_1$ are obtained by a linearly implicit backward Euler step) on a mesh with 
$10522$ nodes and time step size $\tau \approx 3 \cdot 10^{-3}$. The figures show snapshots of the surface evolution at different times (from top to bottom). The left column shows plots from Dziuk's algorithm, the middle column from \eqref{BDF}, while on the right we present the results of a normalized version of  \eqref{BDF}, where the normal vector $\n_h$ computed from the discretized evolution equation is rescaled to unit norm at each node. Figure~\ref{fig:curvature} shows the evolution and blow-up of the computed mean curvature. 

\begin{figure}[htbp]
	\centering%
	\includegraphics[width=\textwidth,height=0.24\textheight]{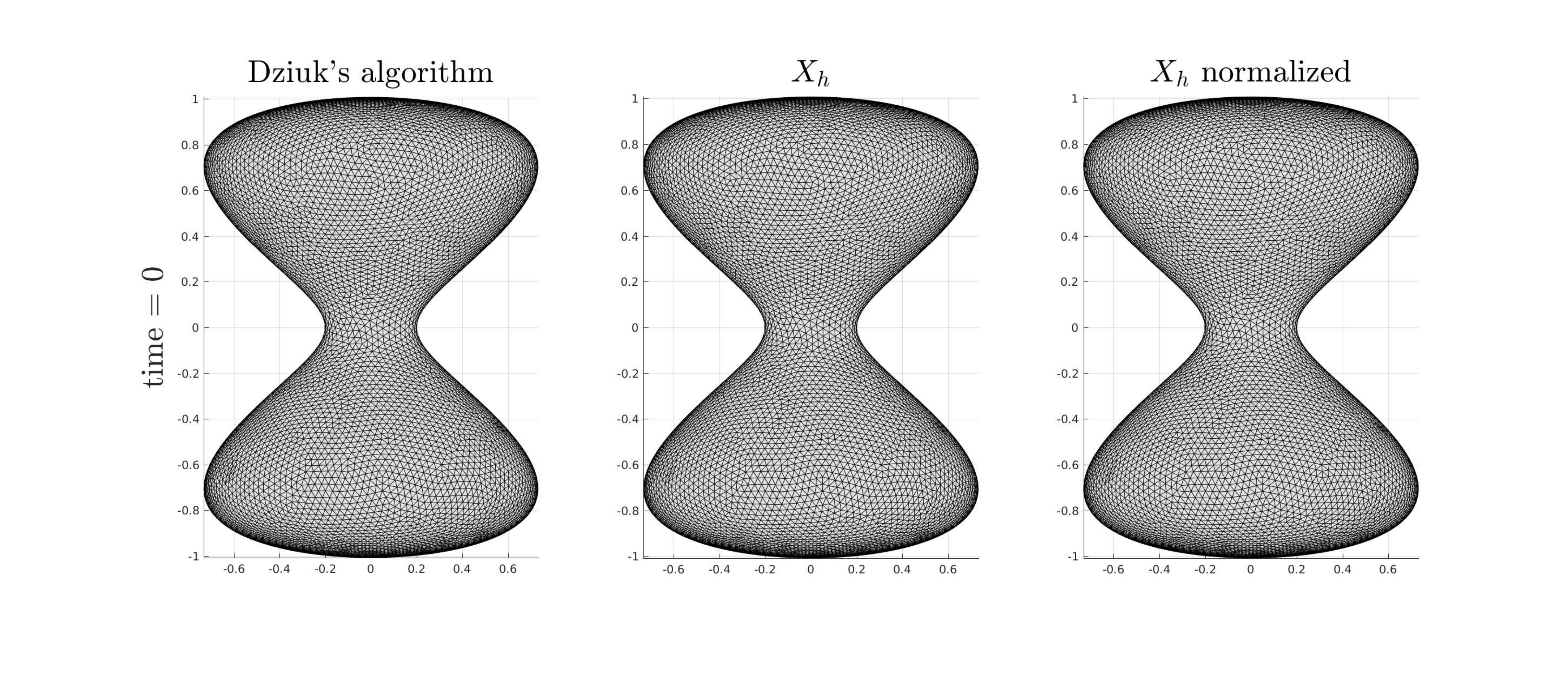}
	\includegraphics[width=\textwidth,height=0.24\textheight]{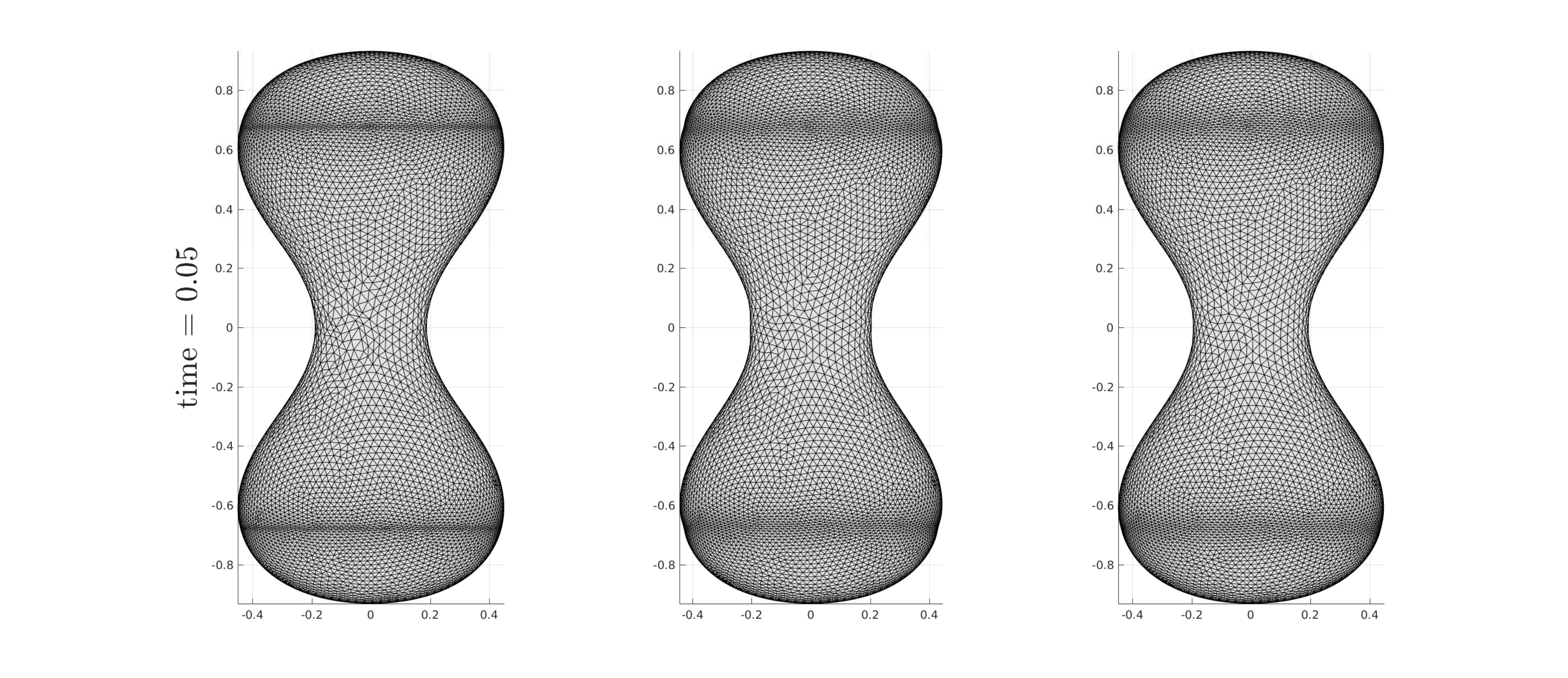}
	\includegraphics[width=\textwidth,height=0.24\textheight]{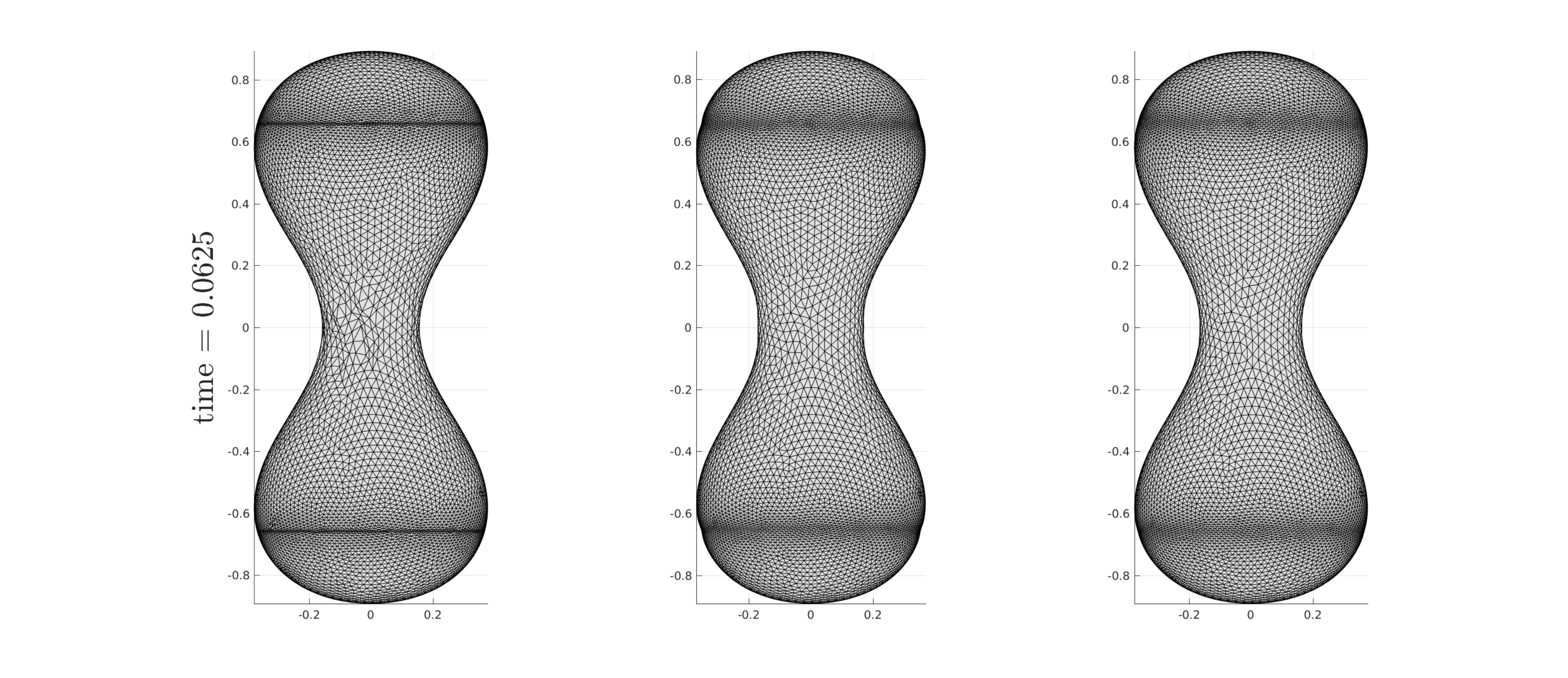}
	\includegraphics[width=\textwidth,height=0.24\textheight]{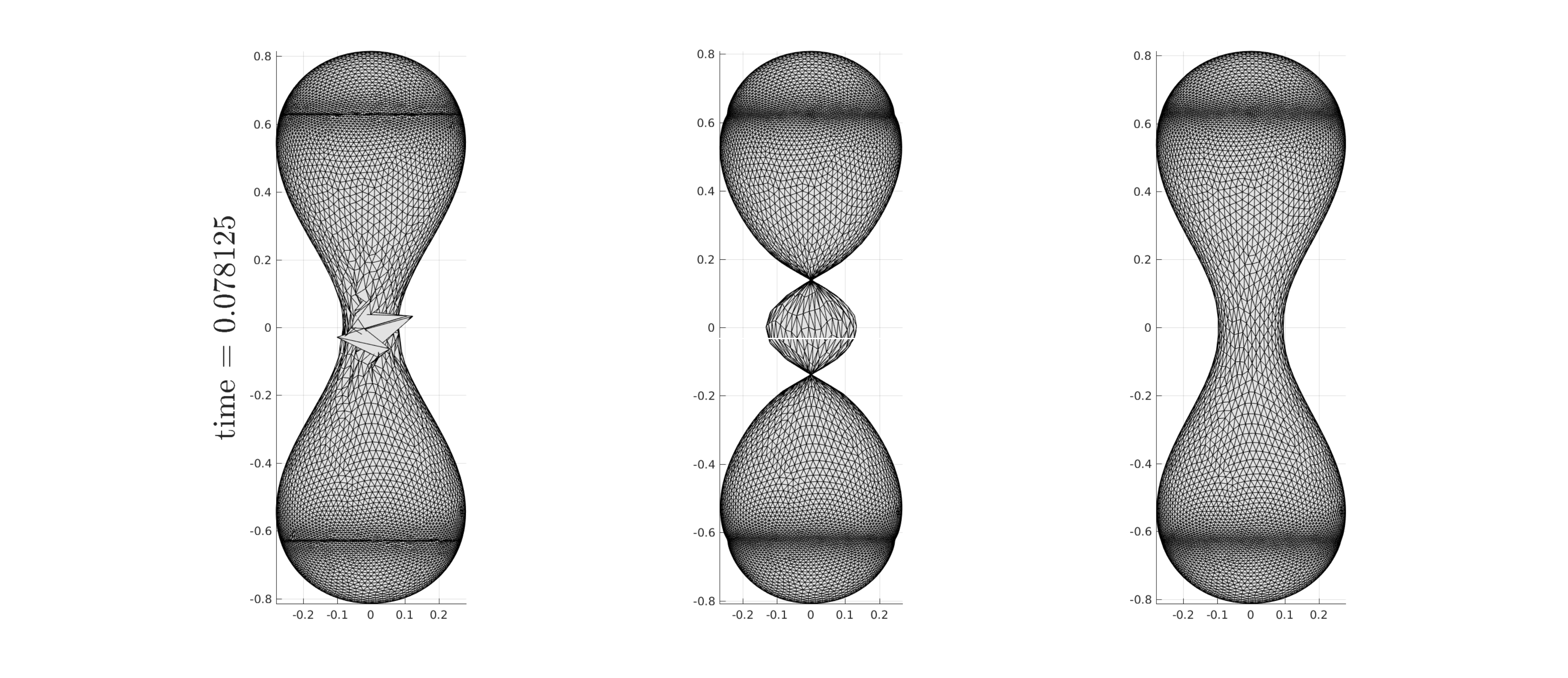}
	\caption{Comparison of Dziuk's algorithm, of Algorithm \eqref{BDF},  and of its version with normalized approximate normal vector $\nu_h$ on a  flow developing a pinch singularity.}
	\label{fig:singularMCF}
\end{figure}
\begin{figure}[htbp]
	\centering%
	\includegraphics[width=\textwidth,height=0.24\textheight]{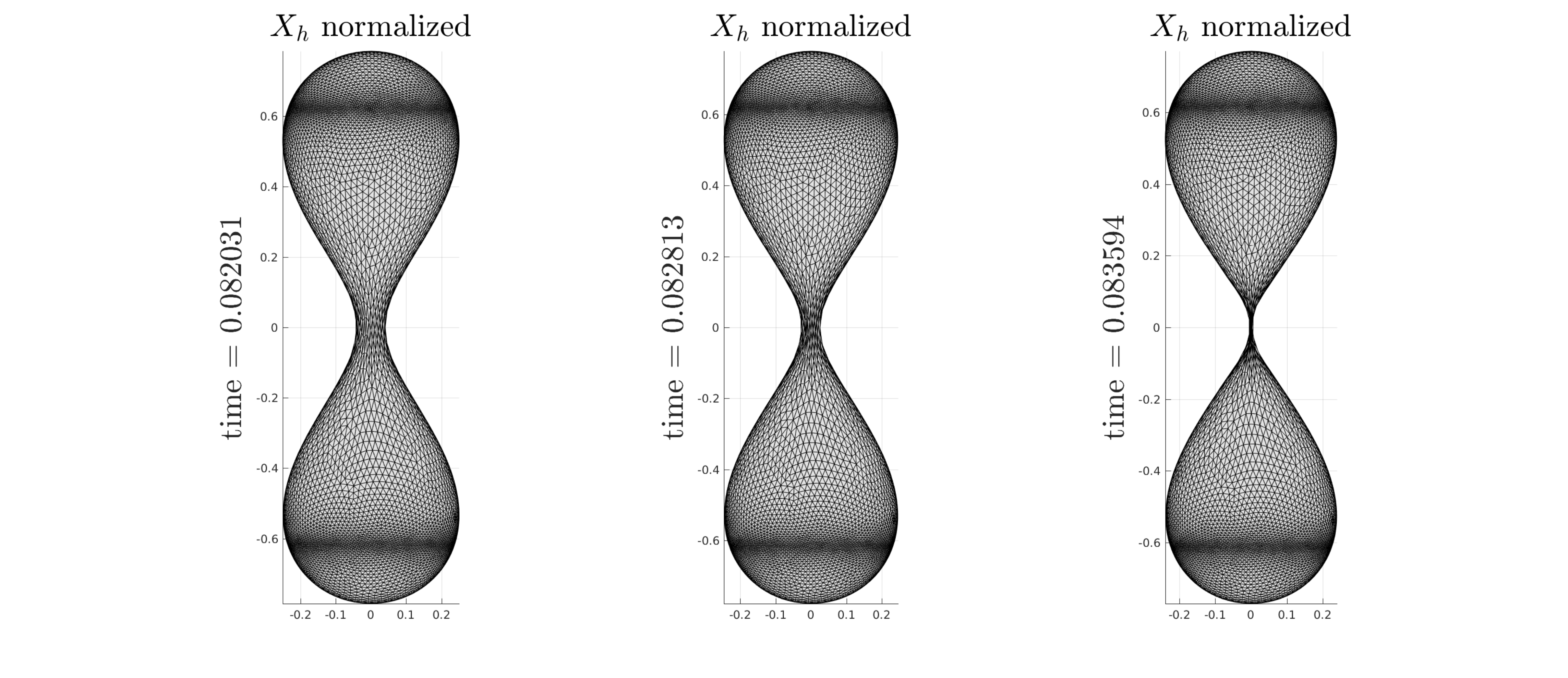}
	\caption{The version of \eqref{BDF} with normalized $\nu_h$ close to the pinch singularity.}
	\label{fig:MCF_normalised}
\end{figure}
\begin{figure}[htbp]
	\centering
	\includegraphics[scale=0.6]{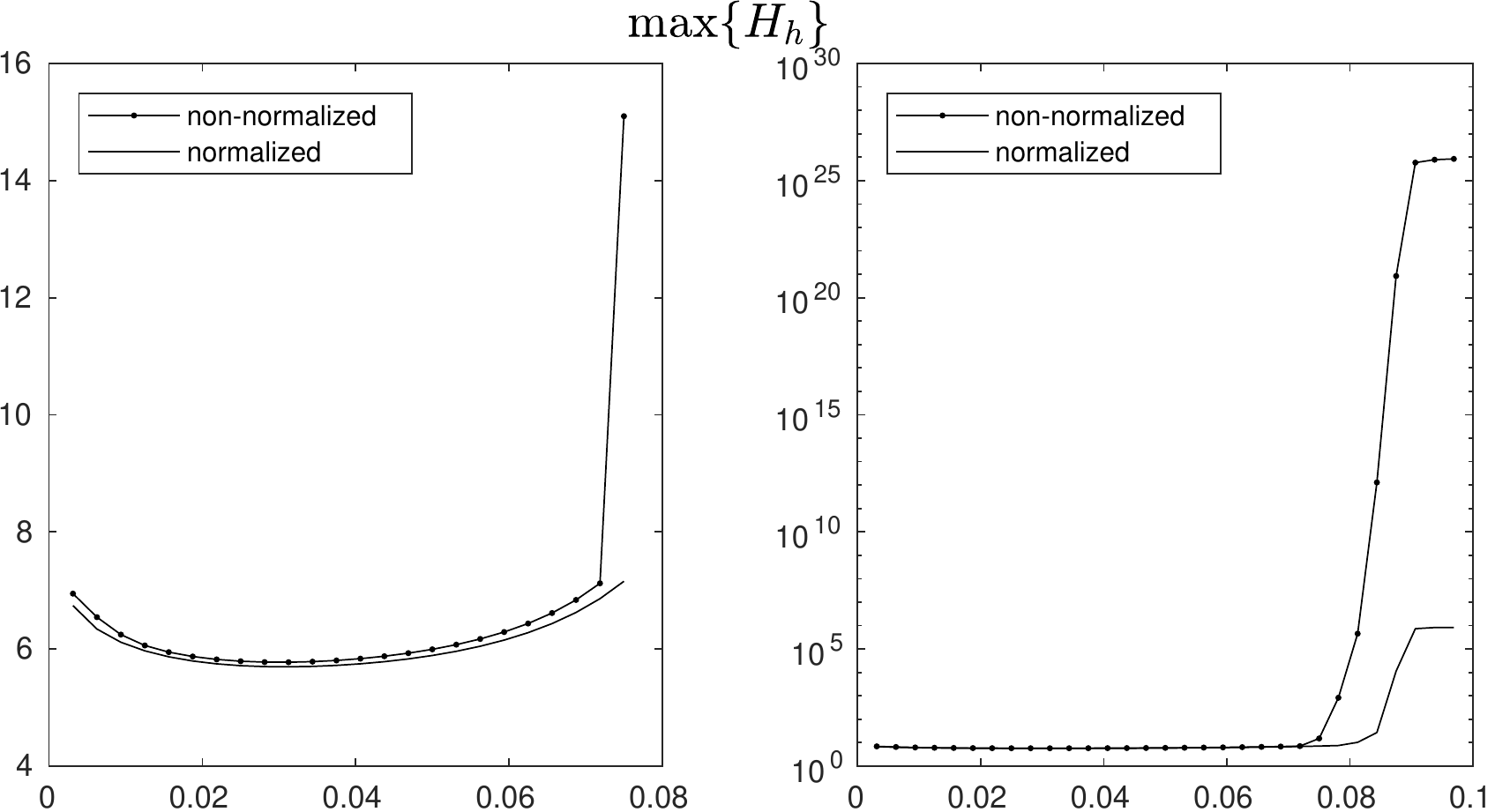}
	\caption{Maximum of mean curvature and blow-up plotted against time.}
	\label{fig:curvature}
\end{figure}

\bbk
\begin{figure}[htbp]
	\centering
	\includegraphics[scale=0.6]{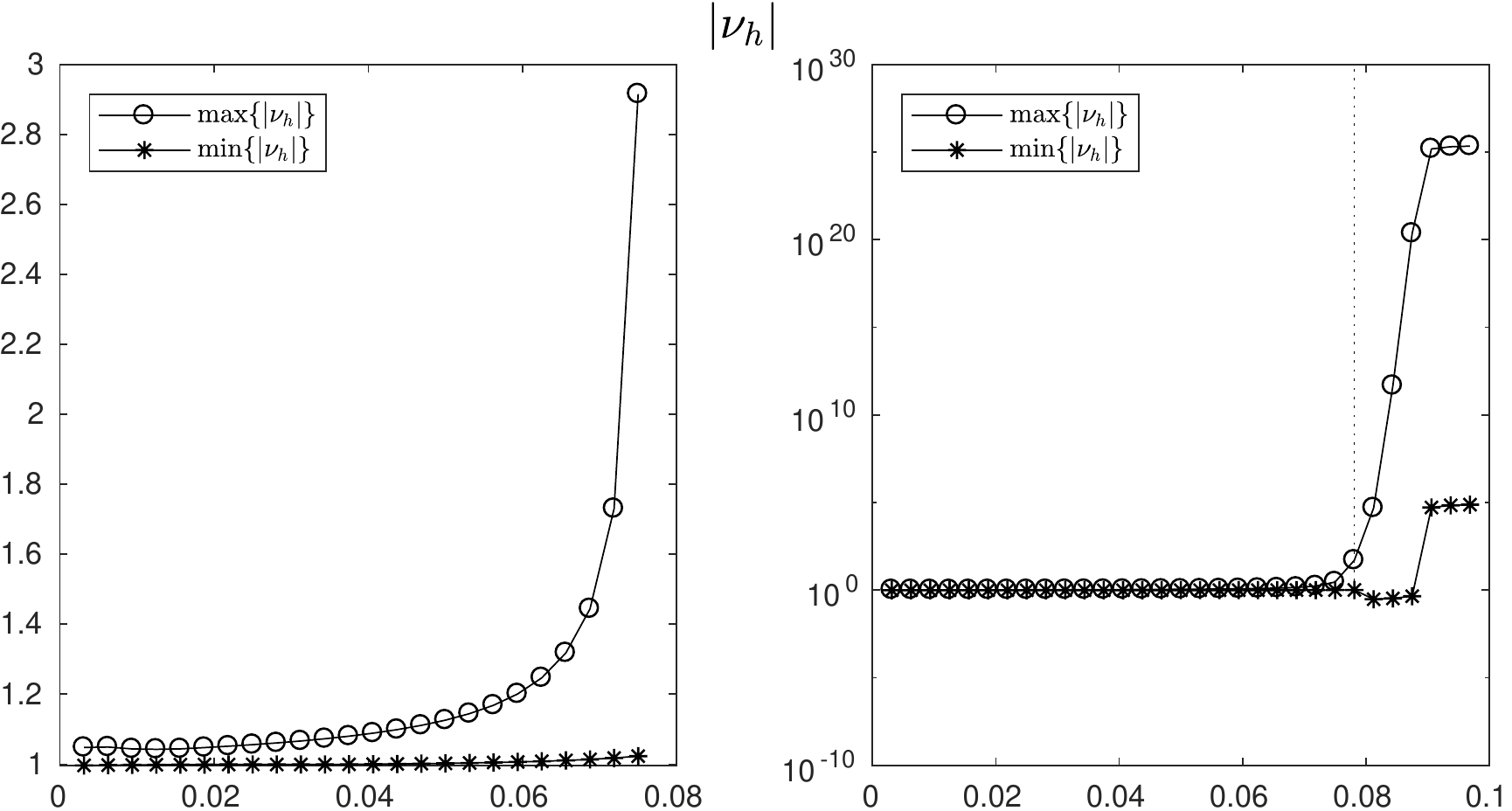}
	\caption{Maximum and minimum length of normal vector and blow-up plotted against time.}
	\label{fig:normal}
\end{figure}
\ebk

\bcl 
We find that all three methods work reasonably well as long as the mean curvature flow stays away from the singularity, say for times $t\le 0.07$ with the singularity near $0.08$.  With all three methods, however, the distribution of the moving nodes on the surface appears most dense near the two widest horizontal circles of the dumbbell, which is not a desired behaviour. The methods behave differently shortly before the singularity (last row of pictures in Figure~\ref{fig:singularMCF}): Dziuk's method develops an instability, the method \eqref{BDF} remains stable but is past the pinch singularity too early, whereas the method \eqref{BDF} with normalized computed normal vectors behaves very well also close to the pinch singularity; see Figure~\ref{fig:MCF_normalised}. In the unnormalized method, the computed normal vector grows fast in norm shortly before the singularity, \ecl \bbk see Figure~\ref{fig:normal}, \ebk \bcl which speeds up the evolution according to $v=-H\nu$. This effect is delayed when a finer time and space grid are used. 
\ecl


\section{Conclusion and further comments}
In this paper we have given the first proof of convergence of an evolving finite element method for mean curvature flow of a closed surface. Under sufficient regularity assumptions on the evolving surface we have proved optimal-order convergence. Unlike {Dziuk's} method, the numerical method proposed here relies on the discretization of evolution equations for geometric quantities, which are then used in the velocity law. 

In the convergence analysis we have clearly separated the issues of consistency and stability. The geometry enters only in the consistency error bounds, but not in the stability analysis. The latter is done in the matrix--vector formulation of the method based on the technique of energy estimates --- here by testing the error equation with the (discretized) time derivative of the error vector --- and on results from \cite{KLLP2017} that relate the mass and stiffness matrices for different finite element surfaces. Together with inverse estimates, this allows us to control the $W^{1,\infty}$ norm of the discrete surface, which is essential in proving stability estimates that, together with consistency estimates, yield convergence. \bcl However, our stability analysis does not work for Dziuk's method, as is explained in Remark~\ref{rem:Dziuk}.\ecl


The numerical method studied here is based on the discretization of the evolution equations for the normal vector and the mean curvature, for which we then obtain optimal-order error bounds in the $H^1$ norm. Instead, one could also consider the evolution equations for other geometric quantities, such as the second fundamental form. When this two-form is expressed in local coordinates on each element, then a discontinuous Galerkin discretization appears as the method of choice to account for different local coordinates on different finite elements. 

We expect that the discretization of evolution equations of geometric quantities will turn out useful also for the numerical treatment of other geometric evolution equations. For example, the extension of our numerical approach and its analysis appear fairly direct \bcl for mean curvature flow coupled to diffusion on the surface and also \ecl for inverse mean curvature flow, for which the corresponding evolution equations are given in \cite{HuiskenPolden}. 

On the practical side, given that we now have a convergent algorithm that provides optimal-order approximations also to the normal vector and to the mean curvature, is this the method that we would recommend for practical computations? Not quite. There are at least two aspects that need to be addressed to increase the robustness of the method beyond very fine meshes, in particular in near-singular situations:
\begin{itemize}
\item Mesh quality: Like Dziuk's method, also the method presented here often does not yield a satisfactory distribution of the nodes on the surface. It is to be expected that an arbitrary Lagrange -- Euler (ALE) approach combined with the procedure of this paper can substantially mitigate this problem. 
In the literature this problem has been addressed, for example, by Elliott and Fritz \cite{ElliottFritz_DT} using the DeTurck trick, and by Barrett, Garcke and N\"urnberg in a series of papers \cite{BGN2007,BGN2008,BGN2008Willmore,BGN2011,BGN2012}.
\item The normal vector $\n_h$ obtained from the discretized evolution equation is not the same as the normal vector $\n_{\Gamma_h[\bfx]}$ of the discrete surface. This poses no problem as long as they are close to each other (which asymptotically for sufficiently fine meshes they are by our convergence result), but leads to artifacts once they differ substantially. The situation can be improved by adding a stabilizing term in the discretization that works against a drift of the two normal vectors: on the right-hand side of \eqref{discrete nu} one can add a stabilizing term $-\alpha \int_{\Gamma_h[\bfx]} (\n_h -\n_{\Gamma_h[\bfx]}) \cdot \phin_h$ with a parameter $\alpha>0$. 
\\
Moreover, the normal vector computed from the discretized evolution equation is not of unit norm. While this appears easy to remedy in practice by rescaling to unit norm at every node (as is done in Section 13),  the effect of such a pointwise rescaling is not easily understood theoretically. This problem is familiar also in numerical methods for harmonic map heat flow and equations for micromagnetism; see, e.g. \cite{Cimrak07,Prohl01}.
\end{itemize}

We presented the numerical method for two-dimensional closed surfaces, but both the formulation and the convergence analysis of the numerical method can be extended to closed smooth hypersurfaces of arbitrary dimension. This is straightforward for three-dimensional hypersurfaces in $\R^4$, where the proof extends verbatim. \bcl The only change is in the use of the inverse inequality for finite element functions between the $W^{1,\infty}$ and $H^1$ norms, where a  factor $h^{-3/2}$ instead of $h^{-1}$ appears. This extra factor $h^{-1/2}$, however, does not affect the course of the proof. In higher dimensions, the error bounds for interpolation need to be replaced by analogous error bounds for quasi-interpolation, and the inverse estimates depend on the dimension, yielding a factor $h^{-d/2}$ in dimension $d$.  As a consequence, the minimal polynomial degree $k$ required for our stability and convergence analysis becomes $k\ge \lfloor d/2\rfloor+1$ for a $d$-dimensional hypersurface in $\R^{d+1}$, where $\lfloor d/2\rfloor$ is the largest integer smaller than $d/2$. Under the stepsize restriction $\tau\le C_0 h$, we then also need the temporal order $q\ge \lfloor d/2\rfloor+1$.\ecl


\section*{Acknowledgement}
We thank Frank Loose and Gerhard Wanner for helpful comments, and we also thank J\"org Nick for helpful discussions.

We thank two anonymous referees for their constructive comments on a previous version.

The work  was partially supported by a grant from the Germany/Hong Kong Joint Research Scheme sponsored by the Research Grants Council of Hong Kong and the German Academic Exchange Service (G-PolyU502/16). The work of Bal\'azs Kov\'acs and Christian Lubich is supported by Deutsche Forschungsgemeinschaft, SFB 1173.


\end{document}